\documentclass[arxiv,reqno,twoside,a4paper,11pt]{amsart}

\UseRawInputEncoding

\usepackage{tikz}
\usepackage{amsmath, verbatim}
\usepackage{amssymb,amsfonts,mathrsfs,mathtools}
\usepackage[colorlinks=true,linkcolor=blue,citecolor=blue]{hyperref}
\usepackage[driver=pdftex,margin=2.5cm,top=2.5cm, bottom=2.5cm,centering]{geometry}
\usepackage{longtable, slashed}
\usepackage{tabularx}
\usepackage{multirow}
\usepackage{caption}

\usepackage{latexsym,enumitem}
\usepackage[all]{xy}

\usepackage[scaled]{helvet} 
\usepackage{courier} 
\usepackage[mathbf]{euler}

\theoremstyle{plain}
\newtheorem{thm}{Theorem}[section]
\newtheorem{prop}[thm]{Proposition}
\newtheorem{cor}[thm]{Corollary}

\newtheorem{lemma}[thm]{Lemma}
\newtheorem{remark}[thm]{Remark}

\theoremstyle{definition}
\newtheorem{defn}[thm]{Definition}

\newtheorem{assump}[thm]{Assumption}

\theoremstyle{remark}

\theoremstyle{plain}

\newcommand{\R}{\mathbb{R}}

\newcommand{\C}{\mathbb{C}}

\newcommand{\N}{\mathbb{N}}

\newcommand{\ind}{\mathrm{ind}}

\newcommand{\supp}{\mathrm{supp}}

\newcommand{\cohr}{\overline{H}_{(2)}}
\newcommand{\coh}{H_{(2)}}

\newcommand{\trace}{\mathrm{tr}}
\newcommand{\Tr}{\mathrm{Tr}}

\newcommand{\image}{\mathrm{im}~}

\DeclareMathOperator{\spec}{spec}
\DeclareMathOperator{\cU}{\mathscr{U}}

\DeclareMathOperator{\cH}{\mathscr{H}}
\DeclareMathOperator{\A}{\alpha}
\DeclareMathOperator{\w}{\omega}
\DeclareMathOperator{\W}{\Omega}

\DeclareMathOperator{\calC}{\mathscr{C}}
\DeclareMathOperator{\dom}{\mathscr{D}}

\numberwithin{equation}{section}

\definecolor{qqwuqq}{rgb}{0,0,0}

\usepackage{color}
\definecolor{darkgreen}{cmyk}{1,0,1,.2}
\definecolor{m}{rgb}{1,0.1,1}
\definecolor{b}{rgb}{0,0.1,1}

\begin{document}

\date{\today}

\title[
Signatures of Witt spaces with boundary]
{
Signatures of Witt spaces with boundary}

\author{Paolo Piazza}
\address{Sapienza University, Rome, Italy} 
\email{paolo.piazza@uniroma1.it} 

\author{Boris Vertman} 
\address{Universit\"at Oldenburg, Germany} 
\email{boris.vertman@uni-oldenburg.de}


\begin{abstract}
Let $\overline{M}$ be a compact smoothly stratified pseudomanifold  with boundary, satisfying the Witt assumption.
In this paper we introduce the de Rham signature and the Hodge signature of 
$\overline{M}$, and prove their equality. Next, building also on recent work of Albin and Gell-Redman,
we extend the Atiyah-Patodi-Singer index theory established in our previous work under the hypothesis that $\overline{M}$
has stratification depth $1$ to the general case, establishing in particular a signature formula on Witt spaces with boundary.

In a parallel way we also pass to the case of a Galois covering $\overline{M}_\Gamma$ of
$\overline{M}$ with  Galois group $\Gamma$. Employing von Neumann algebras
we introduce the de Rham $\Gamma$-signature and the Hodge
$\Gamma$-signature and 
prove their equality, thus extending to Witt spaces a result proved by  L\"uck and Schick
in the smooth case. Finally, extending work of Vaillant in the smooth case,
we establish a formula for the Hodge $\Gamma$-signature. As a consequence we deduce
the fundamental result that equates  the Cheeger-Gromov rho-invariant of the boundary $\partial \overline{M}_\Gamma$ with
the difference of the signatures of and $\overline{M}$ and  $\overline{M}_\Gamma$:
\begin{align*}
\textup{sign}_{{\rm dR}} (\overline{M}, \partial \overline{M}) - 
\textup{sign}^{\Gamma}_{{\rm dR}} (\overline{M}_\Gamma, \partial \overline{M}_\Gamma)
=\rho_{\Gamma} (\partial \overline{M}_\Gamma).
\end{align*}
We end the paper with two
geometric applications of our results.
\end{abstract}

\maketitle \ \\[-18mm]

\tableofcontents

\section{Introduction and statement of the main results}

%
\subsection{Signatures on closed compact manifolds}

Let $(M,g)$ be a compact oriented Riemannian manifold of dimension $4n$ without boundary.
There are various equivalent notions of what the signature of $M$ is. We have
\begin{itemize}
\item the topological signature $\textup{sign}_{{\rm top}} (M)\in\mathbb{Z}$ obtained by considering the signature
of the non-degenerate bilinear symmetric form $H^{2n} (M,\mathbb{R})\times H^{2n} (M,\mathbb{R})\to \mathbb{R}$ assigning to
$\alpha$ and $\beta$ the real number $\langle \alpha\cup\beta,[M]\rangle$. Here we consider, e.g. the singular cohomology
with real coefficients;
\item the de Rham signature  $\textup{sign}_{{\rm dR}} (M)\in\mathbb{Z}$ obtained by considering
the signature of the non-degenerate bilinear symmetric form $H^{2n}_{dR} (M)\times H^{2n}_{dR} (M)\to \mathbb{R}$ assigning to
$\omega$ and $\eta$ the real number $\int_M \omega\wedge \eta$;
\item the Hodge signature $\textup{sign}_{{\rm Ho}} (M)\in\mathbb{Z}$, defined by the same formula as for the de Rham signature but on the
Hodge cohomology $\cH^{2n}(M)$;
\item the index of the signature operator $D$, $\ind (D)\in\mathbb{Z}$;
\item the integral of the Hirzebruch $L$-class $ \int_M L(M)$.
\end{itemize}
A fundamental result in Mathematics is the following chain of equalities:
\begin{equation}\label{equality}
\textup{sign}_{{\rm top}} (M)=\textup{sign}_{{\rm dR}} (M)=\textup{sign}_{{\rm Ho}} (M)=\ind (D)=\int_M L(M)\,.
\end{equation}
The first equality follows from the de Rham theorem and the compatibility between cup product and wedge product, 
the second from the Hodge theorem, the third is a simple 
computation and the fourth follows from the Atiyah-Singer formula applied to the signature operator. 
The equality of the first and last term, $\textup{sign}_{{\rm top}} (M)=\int_M L(M)$, is the celebrated 
Hirzebruch's signature theorem. \medskip

We recall that
the signature operator is defined as follows: we consider 
 $d+d^*$ acting on the space $\Omega^*(M)$ of differential forms of all degrees but with the grading
 $\Omega^*(M)=\Omega^+(M)\oplus \Omega^-(M)$ induced by the involution
  $\tau$ on $\Omega^*(M)$ acting on a form $\omega$ of degree $p$ as
\begin{equation}\label{tau}
\tau \omega := i^{p(p-1)+2n} * \omega\,.
\end{equation}
Here  $*$ denotes the Hodge star operator and we have written $\Omega^+(M)$ and
$\Omega^-(M)$ for the $(+1)$- and $(-1)$-eigenspaces of $\tau$, respectively.
The operator $d+d^*$ anti-commutes with $\tau$ and hence interchanges $\Omega^+(M)$ and
$\Omega^-(M)$. We have set
\begin{align}\label{D}
D:= d+d^* |_{\Omega^+ (M)}: \Omega^+(M) \to \Omega^-(M).
\end{align}

\subsection{Signatures on manifolds with boundary}

Let $(M,g)$ now be a compact $4n$-dimensional oriented Riemannian manifold with boundary $\partial M$. 
We denote the relative and absolute singular cohomologies of $M$ by $H^*(M,\partial M)$ and $H^*(M)$,
respectively\footnote{All our cohomologies are with real coefficients if not otherwise stated; thus we do not carry along
the coefficients $\mathbb{R}$  in the notation.}. Remark that there is a  natural homomorphism  $\iota:H^{*}(M,\partial M) 
\to H^{*}(M)$. 
 The (relative) topological signature of $M$, denoted $\textup{sign}_{{\rm top}} (M, \partial M)$, is defined as the signature 
of the (degenerate) bilinear form given
by the cup product on $H^{2n}(M,\partial M)$. The radical of this symmetric bilinear form is equal to the kernel
of $\iota:H^{*}(M,\partial M) 
\to H^{*}(M)$ and so this signature can be defined directly 
on the image of $\iota:H^{2n}(M,\partial M) \to H^{2n}(M)$.

\medskip
\noindent
A similar definition can be given for the de Rham signature $\textup{sign}_{{\rm dR}} (M, \partial M)$.

\medskip
\noindent
Finally consider the manifold with cylindrical ends $M_\infty$ associated to
$M$ and let $ \cH_{(2)}^{2n}(M_\infty)$ be the Hodge $L^2$-cohomology of $M_\infty$.
The Hodge $L^2$-signature, denoted $\textup{sign}_{{\rm Ho}} (M_\infty)$, is defined as the signature  of the non-degenerate bilinear form
\begin{equation*}
\begin{split}
\cH_{(2)}^{2n}(M_\infty) \times  \cH_{(2)}^{2n}(M_\infty) \to \mathbb{R},
\quad (\w, \eta) \mapsto \int_M \w \wedge \, \eta.
\end{split}
\end{equation*}
One can prove, see  \cite{APSa}, that 
\begin{equation}\label{equal-aps}
\textup{sign}_{{\rm top}} (M, \partial M)=\textup{sign}_{{\rm dR}} (M, \partial M)=\textup{sign}_{{\rm Ho}}  (M_\infty)\,.
\end{equation}
We remark that while the first equality is standard, the second is not
and requires a delicate Hodge-theoretic argument.
Remark also  that this common value is {\it not} equal to the index of the signature operator with Atiyah-Patodi-Singer 
boundary condition. Still, following again the seminal work of Atiyah, Patodi and Singer \cite{APSa}, one can extend Hirzebruch's signature theorem to manifolds with boundary, giving an explicit formula for the signature.
To state their result, assume that $g$ is a product $g = dx^2 \oplus g_{\partial M}$ in a collar neighborhood 
$\cU := [0,1) \times \partial M$ of the boundary, where $x\in [0,1)$ is the normal variable. Then the signature operator 
$D$ takes the following form over $\cU$
\begin{align}\label{D-product}
D =  \sigma \left( \frac{d}{dx} + B\right),
\end{align}
where $\sigma$ is a bundle isomorphism and 
the so-called tangential operator $B$ acts as $B \omega = (-1)^{n+p+1} ((-1)^p *_{\partial M} d_{\partial M} - d_{\partial M} *_{\partial M}) \omega$ on 
a differential form $\omega$ of degree $p$. Here, $*_{\partial M}$ and $d_{\partial M}$ denote the Hodge star operator and
the exterior derivative on the boundary, respectively. \medskip


The operator $B$ is a self-adjoint operator on $\partial M$ with discrete spectrum. Consider an enumeration $\{\lambda_n\}_{n\in \N_0}$ of 
the non-zero eigenvalues of $B$, counted with their multiplicities and ordered in ascending order. 
Denote by $\textup{sign}(\lambda)$ the sign of an eigenvalue $\lambda$. Then the eta function of $B$ is defined by 
\begin{align*}
\eta(B,s) := \sum_{n=0}^\infty \textup{sign}(\lambda_n) \, | \lambda_n| ^{-s}, \quad \Re(s) \gg 0.
\end{align*}
This series is absolutely convergent for $\Re(s) \gg 0$ sufficiently large, and as a consequence 
of the short time asymptotics of the trace $\textup{Tr} \, B e^{-tB^2}$, the 
eta function $\eta(B,s)$ extends to a meromorphic 
function on the whole of $\C$ by the following integral expression
\begin{equation*}
\eta(B,s) = \frac{1}{\Gamma((s+1)/2)}
\int_0^\infty t^{(s-1)/2} \, \textup{Tr} \, B e^{-tB^2} dt
\end{equation*}
Atiyah, Patodi and Singer \cite{APSa} assert the regularity of $\eta(B,s)$ at zero
and define the eta invariant 
$$
\eta(B):= \eta(B,s=0).
$$

\begin{remark}
Since $B$ preserves the splitting 
$\Omega^*(\partial M) = \Omega^{\textup{even}}(\partial M) \oplus \Omega^{\textup{odd}}(\partial M)$
into forms of even and odd degree, we can denote its restriction to $\Omega^{\textup{even}}(\partial M)$
by $B_{\textup{even}}$. Same constructions as above apply and we define $\eta(B_{\textup{even}}):= \eta(B_{\textup{even}},0)$.
Both eta invariants are related by 
$$
\eta(B) = 2\eta(B_{\textup{even}}).
$$
\end{remark}

\noindent
Atiyah, Patodi and Singer \cite{APSa} then prove, that
\begin{equation}\label{aps-sign-th-compact}
\textup{sign}_{{\rm Ho}}  (M_\infty) = \int_M L(M) - \frac{\eta(B)}{2} =  \int_M L(M) - \eta(B_{\textup{even}}).
\end{equation}
Summarizing, we obtain the following identities
\begin{equation}\label{equal-aps-final}
\textup{sign}_{{\rm top}} (M, \partial M)=\textup{sign}_{{\rm dR}} (M, \partial M)=\textup{sign}_{{\rm Ho}}  (M_\infty)=\int_M L(M) - \eta(B_{\textup{even}})\,.
\end{equation}
Notice that on a manifold with boundary neither of the two summands on the right hand side is of topological nature, but their difference is.
\medskip

To complete the picture, we also point out that while it is true that 
the index of the signature operator with APS boundary condition is {\it not} equal 
to the signature $\textup{sign}_{{\rm dR}} (M, \partial M)$, one can prove that  
there is a  {\it generalized} APS boundary condition, defined in terms
of the so called scattering Lagrangian  $\Lambda\subset {\rm Ker}\, B$, having the signature
$\textup{sign}_{{\rm dR}} (M, \partial M)$
as a Fredholm index. Equivalently, there is a {\it perturbed} signature operator on the manifold with
cylindrical ends $M_\infty$ which is Fredholm  and with index equal to $\textup{sign}_{{\rm dR}} (M, \partial M)$.
See  \cite{BWoj} and \cite{Loya-contemporary}.

\subsection{Signatures  for Galois coverings of closed compact manifolds}

Let $\Gamma$ be a finitely generated discrete group and let $M_\Gamma$ be a Galois covering of a 
closed $4n$-dimensional manifold $M$ with Galois group $\Gamma$.
Using  von Neumann algebra  techniques  it is possible to define 
\begin{itemize}
\item the von Neumann topological $\Gamma$-signature, ${\rm sign}^{\Gamma}_{{\rm top}}(M_\Gamma)\in\mathbb{R}$;
\item the von Neumann de Rham $\Gamma$-signature  ${\rm sign}^{\Gamma}_{{\rm dR}} (M_\Gamma)\in\mathbb{R}$;
\item the  von Neumann Hodge $\Gamma$-signature  
$ {\rm sign}^{\Gamma}_{{\rm Ho}}(M_\Gamma)\in\mathbb{R}$;
\item the von Neumann index of the $\Gamma$-equivariant signature operator $\widetilde{D}$ on $M_\Gamma$, an element in $\mathbb{R}$.
\end{itemize}
The following result holds:
\begin{equation}\label{equality-vonNeumann}
\textup{sign}^{\Gamma}_{{\rm top}} (M_\Gamma)=\textup{sign}^{\Gamma}_{{\rm dR}} (M_\Gamma)=
\textup{sign}^{\Gamma}_{{\rm Ho}} (M_\Gamma)=\ind^{\Gamma} (\widetilde{D})=\ind (D)=\int_M L(M)\,.
\end{equation}
The first two equalities follow from the extension of the de Rham-Hodge theorem to Galois coverings, due to Dodziuk 
\cite{Dodziuk}, and an extra argument having to do with the pairings themselves (this additional argument can be found in 
\cite[proof of Theorem 3.10]{Lueck-Schick}); the third equality rests again on a simple algebraic argument, the fourth equality 
is precisely Atiyah's theorem, asserting the equality of the $\Gamma$-index of $\widetilde{D}$ on $M_\Gamma$ 
and the index of the signature operator $D$ on the base $M$ and the fifth equality follows again from the Atiyah-Singer index formula. Notice that, consequently,
the $\Gamma$-signature of $M_\Gamma$ equals the signature of $M$:
$$\textup{sign} (M)=\textup{sign}^{\Gamma}(M_\Gamma),$$
where we omit the lower indices top, Ho, dR, since all the corresponding signatures coincide.

\subsection{Signatures  for Galois coverings of compact manifolds with boundary}
Let now $M_\Gamma$ be a Galois $\Gamma$-cover of a smooth compact manifold $M$ 
with boundary $\partial M$. Using reduced $L^2$-cohomology, both in the topological and in the de 
Rham case, as well as von Neumann algebra techniques,
one can define 
\begin{itemize}
\item the von Neumann topological $\Gamma$-signature, ${\rm sign}^{\Gamma}_{{\rm top}}(M_\Gamma, \partial M_\Gamma)$;
\item the von Neumann de Rham $\Gamma$-signature  ${\rm sign}^{\Gamma}_{{\rm dR}} (M_\Gamma, \partial M_\Gamma)$;
\item the  von Neumann Hodge $\Gamma$-signature ${\rm sign}^{\Gamma}_{{\rm Ho}}(M_{\Gamma, \infty})$
of the Galois cover $M_{\Gamma, \infty}$ of the manifold $M_\infty$ with cylindrical end, with Galois group $\Gamma$.
\end{itemize}
One can prove that
\begin{equation}\label{equality-vonNeumann-bdry}
\textup{sign}^{\Gamma}_{{\rm top}} (M_\Gamma, \partial M_\Gamma)=\textup{sign}^{\Gamma}_{{\rm dR}} (M_\Gamma, \partial M_\Gamma)=\textup{sign}^{\Gamma}_{{\rm Ho}} (M_{\Gamma, \infty})\,.\end{equation}
See L\"uck and Schick \cite{Lueck-Schick} for a careful proof of these two equalities; the second one 
has a rather intricate proof.\medskip


The signature formula \eqref{aps-sign-th-compact}
in this von Neumann context is a highly non-trivial result
and it is  due to Vaillant \cite{Va}. We briefly describe it: 
we lift $B$ and its even part $B_{\textup{even}}$ to $\Gamma$-invariant operators
$\widetilde{B}$ and $\widetilde{B}_{\textup{even}}$, respectively. By work of Ramachandran \cite{Ram} we can associate to these
operator their respective $\Gamma$-eta invariants $\eta_\Gamma (\widetilde{B})$ and
 $\eta_\Gamma (\widetilde{B}_{\textup{even}}) = \eta_\Gamma (\widetilde{B})/2$. Vaillant \cite{Va} proves
$$
\textup{sign}^\Gamma_{{\rm Ho}} 
(M_{\Gamma, \infty}) = \int_M L(M) - \frac{\eta_\Gamma (\widetilde{B})}{2}
=\int_M L(M) - \eta_\Gamma (\widetilde{B}_{\textup{even}}).
$$

\medskip
\noindent
Summarizing, on a Galois $\Gamma$-cover of a compact manifold with boundary the following 
equalities hold:

\begin{equation}\label{equality-vonNeumann-bdry-final}
\begin{split}
\textup{sign}^{\Gamma}_{{\rm top}} (M_\Gamma, \partial M_\Gamma) &=
\textup{sign}^{\Gamma}_{{\rm dR}} (M_\Gamma, \partial M_\Gamma) =\textup{sign}^{\Gamma}_{{\rm Ho}} (M_{\Gamma, \infty})
\\&= \int_M L(M) - \eta_\Gamma (\widetilde{B}_{\textup{even}}).
\end{split}
\end{equation}
Notice that for manifolds with boundary, in contrast with the closed case (we omit the lower indices
top, Ho and dR, since the corresponding signatures coincide)
$$\textup{sign}^{\Gamma} (M_\Gamma, \partial M_\Gamma)-\textup{sign} (M,\partial M)\not= 0.$$
Indeed, because of the two signature formulae, this difference is equal up to a sign to
$$
\rho_{\Gamma} (\partial M_\Gamma):=\eta_\Gamma (\widetilde{B}_{\textup{even}})- \eta(B_{\textup{even}})\,,
$$
the celebrated Cheeger-Gromov rho invariant, a secondary invariant of the signature operator
which is well-known to be in general different from zero. Hence, the following fundamental equation
holds:
\begin{equation}\label{difference=rho}
\textup{sign} (M,\partial M)- \textup{sign}^{\Gamma} (M_\Gamma, \partial M_\Gamma) =
\rho_{\Gamma} (\partial M_\Gamma).
\end{equation}

\subsection{Goals of this article and main results}
Let now $M$ be the regular part (smooth open interior) of a smoothly stratified pseudomanifold $\overline{M}$
of dimension $4n$. For the time being we assume that $\overline{M}$ is without boundary.
Following the fundamental work of Goresky and MacPherson \cite{intersection}, there is a pairing between the degree-$2n$ intersection cohomology group for the 
 {\it upper} middle perversity and the degree-$2n$ intersection cohomology group for the 
 {\it lower} middle perversity. Even though this pairing is non-degenerate, it does not give us a symmetric
 bilinear form on a single vector space; put it differently, we do not have, in general, a well-defined
 signature. However,
 if $\overline{M}$ satisfies the Witt condition  these two vector spaces 
 are isomorphic and so we do have a well-defined intersection cohomology signature
 denoted  $\textup{sign}_{{\rm top}} (\overline{M})$.
 Here we recall that $\overline{M}$ is Witt if  each link $L$ is odd dimensional or, otherwise,
 the upper middle perversity intersection homology group of $L$ vanishes in degree $\dim L/2$. Using $L^2$-cohomology
 with respect to an iterated conic metric $g$, we can also define  de Rham  and Hodge versions of this signature and prove that
 \begin{equation}\label{witt-compact-signatures}
\textup{sign}_{{\rm top}} (\overline{M}) = \textup{sign}_{{\rm dR}} (\overline{M}) =\textup{sign}_{{\rm Ho}} (\overline{M}) =\ind (D),
\end{equation}
where $D$ is the (unique closure) of the signature operator on the regular part $M$ endowed with the metric $g$,
cf. \cite{Cheeger} and \cite{ALMP1}. Moreover, thanks to the recent work of Albin and Gell-Redman \cite{AGR17} this chain of equalities 
 can be complemented with an explicit index formula 
 \begin{equation}\label{A-GR-signature formula}
\begin{split}
\ind (D) = \int\limits_M L(M) + 
\sum_{\alpha \in A} \int\limits_{\ Y_\alpha} b_{\alpha}.
\end{split}
 \end{equation}
where $\{Y_\alpha\}_{\alpha \in A}$ is the set of singular strata of $\overline{M}$ and 
we refer the reader to \cite{AGR17} for the precise expression of $b_\alpha$. In depth $1$ case,
an index formula for Witt spaces already appears in the work of Br\"uning \cite{Bruning}.
\medskip

 The definition of these signatures  on the Galois $\Gamma$-cover $\overline{M}_\Gamma$ of the Witt space
 $\overline{M}$  is in the literature, although somewhat implicitly:
 we refer to  \cite{Friedman-McLure} for the topological 
 one and again \cite{ALMP1} for the analytic ones. Notice that as discussed
 in \cite[Section 7.2]{AP} there is an Atiyah's theorem, equating the index
 and  the $\Gamma$-index of the signature operator on $M$ (the regular part of $\overline{M}$) 
 and $\widetilde{M}$ (the regular part of $\overline{M}_\Gamma$), respectively.
 The analogue of \eqref{witt-compact-signatures} for these $\Gamma$-signatures  is also implicitly established in the literature. The first equality
 follows from Proposition 11.1 in \cite{ALMP1}, the other equalities are a consequence of the analysis presented
 in \cite{ALMP1, ALMP3, AP} and will be further discussed in this article.
 
 \medskip
\noindent
All these properties hold on a Witt space {\it without} boundary. Now, as we have made clear
in the previous discussion, the extension of these relations to the case in which our Witt space
has a  boundary cannot be straightforward, given that it is already rather involved in the smooth case.
This brings us to the main theme of this article:

\medskip
\noindent
{\it Can one generalize the work of Atiyah-Patodi-Singer, L\"uck-Schick and Vaillant to stratified  Witt spaces {\bf with boundary}
and their $\Gamma$-covers? We shall establish in this article that for the analytically defined signatures this is indeed the case.}

 \medskip
\noindent

Let us state more precisely our main results. Following a well established pattern, we give ourselves 
the freedom to rescale the metric so as to avoid issues related to small eigenvalues. We shall not mention this rescaling in the sequel.

\noindent
Our first  result is about  {\it compact} Witt spaces with boundary.

\begin{thm}\label{main-1} Let $(M,g)$ be the regular part of a compact smoothly 
stratified pseudomanifold $\overline{M}$, with boundary $\partial \overline{M}$, 
endowed with an incomplete iterated cone-edge 
metric $g$ which is of product type near the boundary. We assume that $\overline{M}$ is Witt.
Let $\overline{M}_\infty$ be the Witt-space with cylindrical ends associated to 
$\overline{M}$. Then there exist a well defined $L^2$-de Rham signature $\textup{sign}_{{\rm dR}}(\overline{M},\partial \overline{M})$ and a well defined
$L^2$-Hodge signature $\textup{sign}_{{\rm Ho}} (\overline{M}_\infty)$ and the following equality holds:
\begin{equation}\label{equality-dR-hodge}
\textup{sign}_{{\rm dR}}(\overline{M},\partial \overline{M}) = \textup{sign}_{{\rm Ho}} (\overline{M}_\infty).
\end{equation}
Moreover, the eta invariant $\eta(B_{\textup{even}})$ for the operator $B_{\textup{even}}$
on the regular part $\partial M$ of $\partial \overline{M}$ is well-defined and 
 the following formula holds
$$
\textup{sign}_{{\rm Ho}} (\overline{M}_\infty) = \int_{M} L(M) + \sum_{\alpha \in A} \int\limits_{\ Y_\alpha} b_{\alpha} - \eta(B_{\textup{even}}),
$$
where $\{Y_\alpha\}_{\alpha \in A}$ is the set of singular strata of $\overline{M}$,
and the integrands $b_\alpha$ are given explicitly in \cite{AGR17}.
\end{thm}

The existence of the eta invariant for the boundary operator follows from the proof of the index formula 
as in Atiyah-Patodi-Singer but it can also proved in general, without assuming that the operator is a boundary-operator,
by employing Getzler rescaling. See Albin and Gell-Redman \cite[\S 6.2]{AGR17}. \medskip

Our stratified spaces with boundary (see Definition \ref{psuedo-bdry-def}) have singular strata of 
codimension greater or equal to $2$. By declaring  the boundary to be  a stratum of codimension $1$ we can define  a new stratification
for which our space becomes a stratified space with strata of codimension greater or equal to 
 one in the sense of \cite[Definition 4 and Definition 9]{Bei}. We can then 
use a general de Rham theorem due to Bei, see \cite[Theorem 4]{Bei} and \cite[(95)]{Bei2},
in order to complement our result \eqref{equality-dR-hodge} with the following equality of signatures 
 \begin{equation}\label{equality-dR-GM}
\textup{sign}_{{\rm top}}(\overline{M},\partial \overline{M})=\textup{sign}_{{\rm dR}}(\overline{M},\partial \overline{M}).
\end{equation}
Here on the right hand side  the topological signature defined by Friedman and Hunsicker \cite{FH} appears.
This means that on a compact Witt pseudomanifold with boundary we have the following fundamental
chain of equalities
\begin{equation}\label{equality-total}
\begin{split}
\textup{sign}_{{\rm top}}(\overline{M},\partial \overline{M}) &=\textup{sign}_{{\rm dR}}(\overline{M},\partial \overline{M})
=\textup{sign}_{{\rm Ho}} (\overline{M}_\infty) \\ &= \int_{M} L(M) + \sum_{\alpha \in A} \int\limits_{\ Y_\alpha} b_{\alpha} - \eta(B_{\textup{even}}).
\end{split}
\end{equation}
Notice in particular that $\textup{sign}_{{\rm dR}}(\overline{M},\partial \overline{M})$ and $\textup{sign}_{{\rm Ho}} (\overline{M}_\infty)$ are
metric independent within iterated cone-edge metrics,
 satisfying the Witt condition.

\medskip
Our second main result is about Galois $\Gamma$-coverings of 
stratified Witt spaces with boundary.
\begin{thm}\label{main-2} Let $\overline{M}$ and $(M,g)$ 
be as in the previous theorem. 
Consider a Galois covering $\overline{M}_\Gamma$ of $\overline{M}$ with structure group $\Gamma$.
Let $\overline{M}_{\Gamma, \infty}$ be the Witt space with cylindrical ends associated to $\overline{M}_\Gamma$.
Then there exist a  well-defined $L^2$-$\Gamma$-signature $\textup{sign}^{\Gamma}_{{\rm dR}} (\overline{M}_\Gamma, \partial \overline{M}_\Gamma)$ and a well defined $L^2$-Hodge signature 
 $\textup{sign}^{\Gamma}_{{\rm Ho}} (\overline{M}_{\Gamma,\infty})$ and the following equality holds:
 \begin{equation}\label{equal-witt-gamma}
 \textup{sign}^{\Gamma}_{{\rm dR}} (\overline{M}_\Gamma, \partial \overline{M}_\Gamma)=\textup{sign}^{\Gamma}_{{\rm Ho}} (\overline{M}_{\Gamma,\infty}).
 \end{equation}
Moreover, the  $\Gamma$-eta-invariant $\eta_\Gamma (\widetilde{B}_{\textup{even}})$ 
for the operator $\widetilde{B}_{\textup{even}}$ on the regular part $\partial \widetilde{M}$
of $\partial \overline{M}_\Gamma$ is well-defined and  the following $\Gamma$-signature
formula holds:
\begin{align*}
\textup{sign}^{\Gamma}_{{\rm Ho}} (\overline{M}_{\Gamma,\infty})
= \int_{M} L(M) + \sum_{\alpha \in A} \int\limits_{\ Y_\alpha} b_{\alpha} - \eta_\Gamma (\widetilde{B}_{\textup{even}}),
\end{align*}
\end{thm}

Directly from the previous two theorems we obtain the following corollary, which 
constitutes the most important geometric result of our paper.

\begin{cor}\label{cor-rho} With hypothesis as above, the following  formula holds
\begin{align}\label{relative-signature-theorem}
\textup{sign}_{{\rm dR}} (\overline{M}, \partial \overline{M}) - \textup{sign}^{\Gamma}_{{\rm dR}} (\overline{M}_\Gamma, \partial \overline{M}_\Gamma)
=\rho_{\Gamma} (\partial \widetilde{M}),
\end{align}
where $\rho_{\Gamma} (\partial \widetilde{M}) \in \R$ is the Cheeger-Gromov (signature) rho invariant associated to
$\partial \widetilde{M}$, which denotes the regular part of $\partial \overline{M}_\Gamma$
$$
\rho_{\Gamma} (\partial \widetilde{M})= \eta_\Gamma (\widetilde{B}_{\textup{even}}) -  \eta(B_{\textup{even}})\,.
$$
\end{cor}

Extending to arbitrary depth the (classic) argument given in our previous article, using the results
in  \S \ref{index-section}, one can
prove that if $\Gamma\to \widetilde{N} \to N$ is the universal covering of a closed compact Witt space, then
 $\rho_{\Gamma} (\widetilde{N})$ is metric independent and in fact a stratified diffeomorphism invariant.
  See \cite[Theorem 1.6]{PiVe}. Notice that formula \eqref{relative-signature-theorem} together with
the metric independence of $\textup{sign}_{{\rm dR}}(\overline{M},\partial \overline{M})$ and of $\rho_{\Gamma} (\partial \widetilde{M})$, imply easily that $\textup{sign}^{\Gamma}_{{\rm dR}} (\overline{M}_\Gamma, \partial \overline{M}_\Gamma)$ is also metric 
independent \footnote{It would be more 
natural to extend Bei's work \cite{Bei, Bei2} to Galois coverings and obtain in this way the metric invariance for the $L^2$-$\Gamma$-signature 
$\textup{sign}^{\Gamma}_{{\rm dR}} (\overline{M}_\Gamma, \partial \overline{M}_\Gamma)$. This extension, a Dodziuk theorem \cite{Dodziuk} in the Witt case, lies outside of the scope 
of this paper, which is why we provide an alternative argument.}.
\medskip

As already remarked,  the rho-invariant $\rho_{\Gamma} (\widetilde{N})$ is well-defined in arbitrary depth.
Using this fundamental property we observe that the following result, due to the first author and Albin in depth 1, 
see \cite[Corollary  7.6 and Remark 12]{AP}, holds in fact in general depth.

\begin{thm}
Let $\overline{M}$ be a Witt space of dimension $4\ell-1,$ $\ell>1,$ with regular part equal to $M$. 
We assume  that $\pi_1 (M)$ has an element of finite order and that
$i_*: \pi_1(M)\longrightarrow \pi_1(\overline{M})$ is injective.
Then, there is an infinite number of Witt spaces $\{\overline{N}_j\}_{j\in \mathbb{N}}$ that are stratified-homotopy equivalent 
to $\overline{M}$ but such that  $\overline{N} _i$ is not stratified diffeomorphic to $\overline{N}_k$ for $i\not= k$.
\end{thm}

On the other hand, if $\pi_1 (\overline{M})$ is torsion free, then building on ideas of Weinberger
\cite{weinberger-pnas}
and Chang \cite{Chang}
and on the main result of \cite{ALMP3}, we prove in this article the following result:

\begin{thm}\label{main-3} 
Let $\overline{N}$ be Witt, compact and without boundary and let 
$\overline{N}_\Gamma$ be its universal cover. Let $N$ and $\widetilde{N}$ be
the associated regular parts.. 
Assume that $\pi_1 (\overline{N})$ is torsion-free and satisfies the Baum-Connes conjecture. Then 
 $\rho_{\Gamma} (\widetilde{N})$ is a stratified homotopy invariant.
  \end{thm}
 
We wish to end this Introduction with a remark: we establish in this article that
 most of the classic results on smooth manifolds with boundary do extend to Witt spaces, 
 once the right tools and right definitions are given. However, the actual proofs are far from obvious due to 
 the intricacies of doing  analysis on stratified spaces. For example, the non-uniform behaviour of the heat kernel 
 near the strata of $\overline{M}_\infty$ and $\overline{M}_{\Gamma,\infty}$, see Theorem \ref{off-diagonal}, needs a careful 
 treatment, see also Remark \ref{Vai-differences}.
  Still, some of the arguments do carry over verbatim from the smooth case; whenever this is the case, we simply state the results
 and concentrate instead on those steps that need new arguments because of the singularities.
  
\medskip
The rest of our paper is structured as follows. We continue with a brief review of stratified
spaces with iterated  conic  metrics in \S \ref{stratified-section}. We  call
the pair $(M,g)$ with $M$ smoothly stratified and $g$ an iterated conic metric on its regular part 
a {\it wedge} pseudomanifolds. In \S
\ref{index-section} and in \S \ref{index-section2} we review results of our previous joint work \cite{PiVe} on index theorems on Witt wedge
pseudomanifolds of depth one and their Galois coverings and explain the extension of these results to arbitrary depth
in view of the recent results in \cite{AGR17}. In \S \ref{signature-section} and \S \ref{signature-2-section} we study signatures of stratified Witt spaces with 
boundary and of their Galois coverings. We then prove our first main result, Theorem \ref{main-1}, in \S \ref{main-section}. 
The proof of Theorem \ref{main-2} is obtained in section \S \ref{main-2-section}. We end the paper 
with a section devoted to geometric properties of rho-invariants on Witt spaces in \S \ref{sect:applications}.

\subsection*{Acknowledgements}

The authors gratefully acknowledge financial and inspirational support of the
Priority Programme "Geometry at Infinity" of DFG. The second author also thanks
Sapienza University for hospitality and financial support. We are glad to thank Pierre Albin,
Francesco Bei and Jesse Gell-Redman for useful discussions. We thank the referee for very useful 
remarks.

\section{Stratified spaces with iterated cone-edge metrics}\label{stratified-section}

We recall basic elements in the iterative definition of a compact smoothly stratified (Thom-Mather) space 
of depth $d\in \N_0$. For a full definition, including Thom-Mather conditions, we refer the reader 
e.g. to \cite{ALMP1,ALMP2,Alb}. 

\subsection{Smoothly stratified spaces of depth zero and one}

A compact smoothly stratified space of depth $d=0$ is by definition a smooth compact manifold. 
A compact smoothly stratified space $\overline{M}$ of depth $d=1$ is a compact manifold with an edge singularity,
defined explicitly as follows. By definition $\overline{M}$ consists of a smooth open 
dense stratum $M$ and single singular stratum $Y$, which is itself a closed compact manifold.
There exists a tubular neighborhood $\cU \subset \overline{M}$ which is the total space of a fibration $\phi: \cU \to Y$ 
with fibres given by $\overline{\mathscr{C}(F)} := [0,1) \times F /_{\sim}$, where $(0,\theta_1) \sim (0,\theta_2)$ 
and $F$ is a smooth compact manifold. 
The open smooth stratum $M$ is equipped with an incomplete edge metric $g$ which is by definition 
a smooth Riemannian metric given in the singular neighborhood by 
\begin{equation}
g \restriction \cU \cap M  = dx^2 + \phi^{\ast} g_{Y} + x^2 g_{F} + 
h=:g_0 + h.
\end{equation}
Here, $g_{Y}$ is a smooth Riemannian metric on the stratum $Y$,
$g_F$ is a symmetric two tensor on the level set $\{x=1\}$, which defines a smooth family
of Riemannian metrics on the fibres $F$. The higher order term $h$ satisfies $|h|_{g_0}=O(x)$, as $x\to 0$.

\subsection{Smoothly stratified spaces of arbitrary 
depth d}\label{Section-SSSApdepth}

A compact smoothly stratified space $\overline{M}$
of depth $d\geq 2$  without boundary, with strata $S:=\{Y_{\alpha}\}_{\alpha\in A}$ is a compact space with 
the following inductively defined properties. Identify each stratum with its open interior. Then
\begin{enumerate}
	\item[i)] If $Y_\alpha \cap \overline{Y}_\beta \neq \varnothing$ for any $\alpha, \beta \in A$, then 
	$Y_{\alpha}\subset \overline{Y}_\beta$.	
	\item[ii)] The depth of a stratum $Y \in S$ is defined to be the largest integer $j \in \N_0$ such 
	that there exists a chain of pairwise distinct strata 
	$$
	\{Y=Y_j,\;  Y_{j-1},\ldots, Y_1, Y_0 = M\} \subset S,
	$$
	with $Y_i \subset \overline{Y}_{i-1}$ for all $1	\leq i \leq j$. 
	\item[iii)] The depth $d$ of the stratified space $\overline{M}$ is defined as the maximal depth of any stratum. 
	The stratum of maximal depth is smooth and compact. 	
	\item[iv)] Consider a stratum $Y_\alpha \in S$ of depth $j \in \N_0$. Then any point of $Y_\alpha$ has a tubular 
	neighborhood $\cU_\alpha \subset M$, which is the total space of a fibration 
	$\phi_\alpha: \cU_\alpha \to \phi_\alpha (\cU_\alpha) \subseteq Y_\alpha$ 
	with fibers given by cones $\overline{\mathscr{C}(F_\alpha)}$ with link $F_\alpha$ being a compact 
	smoothly stratified space of depth $(j-1)$.
	\item[v)] Denote by $X_j$ the union of all strata of dimension less or equal than $j\in \N_0$. 
	Denote by $n \in \N_0$ the maximal dimension of any stratum, so that $\overline{M} = X_n$. 
	Then we require that $M:=X_n \setminus X_{n-2}$ is an open smooth manifold dense in $\overline{M}$.
	We call $X_{n-2}$ the singular stratum. Note that $\dim M = n$. We also write $\dim \overline{M} = n$.
\end{enumerate}

The precise definition of compact smoothly stratified spaces 
is more involved, due to additional Thom-Mather conditions. See \cite{Mather}.
The Thom-Mather conditions guarantee that such $\overline{M}$ can be resolved into a compact manifold
with {\it fibered} corners. See \cite{ALMP1,ALMP2} and \cite{Alb}. \medskip

We also need to extend this definition to include stratified spaces with boundary. We follow 
the discussion provided in Banagl \cite[Definition 6.1.3]{Banagl}. 

\begin{defn}\label{psuedo-bdry-def} A compact smoothly stratified space $\overline{M}$ of $\dim \overline{M} = n$ with boundary $\partial \overline{M}$
is a pair $(\overline{M}, \partial \overline{M})$ such that  
\begin{enumerate}
\item $\partial \overline{M}$ is a compact smoothly stratified space of dimension $(n-1)$. 
\item $\overline{M}$ satisfies (i) - (v) as above, with (v) changed insofar that 
$\overline{M} \backslash (X_{n-2} \cup \partial \overline{M})$ is a smooth oriented $n$-dimensional manifold, dense 
in $\overline{M}$. 
\item $\partial \overline{M} \subset \overline{M}$ is collared,
i.e. there exists $\partial \overline{M} \subset \cU \subset \overline{M}$, $\cU$ closed, and an orientation- 
and stratum-preserving isomorphism $\phi: \partial \overline{M} \times [0,1] \xrightarrow{\sim} \cU$.
\item Moreover, writing $S:=\{Y_{\alpha}\}_{\alpha\in A}$ for the strata of $\overline{M}$, 
$\{Y_{\alpha}\cap \partial \overline{M}\}_{\alpha\in A}$ are the strata of $\partial \overline{M}$, and 
$\{Y_{\alpha}\backslash (Y_{\alpha} \cap \partial \overline{M})\}_{\alpha\in A}$ are the strata of 
$\overline{M}\setminus\partial \overline{M}$.
\end{enumerate}
\end{defn}

We finish the review by introducing some terminology. 

\begin{defn} Let  $(\overline{M}, \partial \overline{M})$ be a smoothly stratified space with boundary. 
\begin{enumerate}
\item We write $X_{n-3}(\partial \overline{M})$ for the singular locus of $\partial \overline{M}$.
Then the regular part of $\partial \overline{M}$ is defined as $\partial M := \partial \overline{M} \setminus X_{n-3}(\partial \overline{M})$. 
\item The regular part of $\overline{M}$ is given by $M:=\overline{M} \backslash (X_{n-2} \cup \partial \overline{M})$.
\item The singular part of $\overline{M}$ is given by $X_{n-2}$. 
\end{enumerate}
\end{defn}

We define an iterated cone-edge metric $g$
on $M$ of arbitrary depth   by inductively asking $g$ to be a smooth Riemannian metric away 
from the singular strata, and requiring in each tubular neighborhood $\cU_\alpha$ of any point 
in the singular stratum $Y_\alpha \in X_{n-2}$ that $g$  be of the form
\begin{equation}\label{iemetricUa}
g|_{\cU_\alpha \cap M} = dx^2 + \phi^{\ast}_\alpha g_{Y_\alpha} + x^2 g_{F_\alpha} + 
h=:g_0 + h,
\end{equation}
where $g_{Y_\alpha}$ is a smooth Riemannian metric on $\phi_{\alpha} (\cU_\alpha)$,
$g_{F_\alpha}$ is a symmetric two tensor on the level set $\{x=1\} \equiv \partial \cU_\alpha$, whose restriction to the links
$F_\alpha$ (smoothly stratified spaces of depth at most $(j-1)$) 
is a smooth family of iterated cone-edge metrics. The higher order term $h$ satisfies as before $|h|_{g_0}=O(x)$, when $x\to 0$.
The existence of such iterated cone-edge metrics is discussed e.g. in \cite[Proposition 3.1]{ALMP1}.
Such an open neighborhood $\cU_\alpha$ can be illustrated as in Figure \ref{figure3} (in the next page). \medskip

\begin{figure}[h]
	
\definecolor{cffffff}{RGB}{255,255,255}

\begin{tikzpicture}[y=0.80pt, x=0.80pt, yscale=-1.000000, xscale=1.000000, 
inner sep=0pt, outer sep=0pt]
\begin{scope}[shift={(-106.01298,298.05834)}]
\path[draw=black,line join=miter,line cap=butt,even odd rule,line width=0.800pt]
(240.0967,648.7071) .. controls (208.7675,646.7061) and (175.4286,646.1279) ..
(149.0017,659.4675) .. controls (121.5210,673.3390) and (97.3212,698.3503) ..
(109.9416,726.4275) .. controls (126.1149,762.4088) and (188.7937,735.5132) ..
(226.4242,747.3526) .. controls (271.9113,761.6637) and (330.3244,751.1131) ..
(363.8318,725.7300) .. controls (386.4782,708.5746) and (381.6304,683.3618) ..
(365.2268,660.1650) .. controls (346.4257,633.5776) and (303.4705,651.9467) ..
(270.9215,650.9842);
\path[draw=black,fill=cffffff,line join=round,line cap=round,miter
limit=4.00,even odd rule,line width=2.000pt] (252.1362,698.6361)node[left=1.5cm, 
=80pt]{$Y_\alpha$} -- (312.9573,545.7708)
node[left=3.5cm,below=-1cm]{$F_\alpha$};
\path[draw=black,line join=miter,line cap=butt,even odd rule,line width=0.800pt]
(209.5942,517.3408) -- (251.6996,697.0405);
\begin{scope}[cm={{0.48166,0.0,0.0,0.47789,(116.19249,317.73604)}}]
\path[draw=black,line join=miter,line cap=butt,even odd rule,line width=.8pt]
(343.2858,639.2604) .. controls (343.2858,639.2604) and (269.9420,564.0695) ..
(250.5061,585.6094) .. controls (207.6797,633.0720) and (305.4157,656.0409) ..
(343.2858,639.2604) -- cycle;
\path[draw=black,fill=cffffff,line join=miter,line cap=butt,even odd rule,line
width=.8pt] (303.8643,605.1702) .. controls (299.0602,592.7724) and
(274.9892,631.3365) .. (282.6137,638.9261) .. controls (299.4140,655.6496) and
(309.4403,619.5599) .. (303.8643,605.1702) -- cycle;
\end{scope}
\begin{scope}[cm={{0.68304,0.0,0.0,0.68074,(59.34603,157.6366)}},fill=cffffff]
\path[draw=black,fill=cffffff,line join=miter,line cap=butt,even odd rule,line
width=.8pt] (343.2858,639.2604) .. controls (343.2858,639.2604) and
(269.9420,564.0695) .. (250.5061,585.6094) .. controls (207.6797,633.0720) and
(305.4157,656.0409) .. (343.2858,639.2604) -- cycle;
\path[draw=black,fill=cffffff,line join=miter,line cap=butt,even odd rule,line
width=.8pt] (303.8643,605.1702) .. controls (299.0602,592.7724) and
(274.9892,631.3365) .. (282.6137,638.9261) .. controls (299.4140,655.6496) and
(309.4403,619.5599) .. (303.8643,605.1702) -- cycle;
\end{scope}
\begin{scope}[shift={(-30.25595,-93.66699)},fill=cffffff]
\path[draw=black,fill=cffffff,line join=miter,line cap=butt,even odd rule,line
width=0.800pt] (343.2858,639.2604) .. controls (343.2858,639.2604) and
(269.9420,564.0695) .. (250.5061,585.6094) .. controls (207.6797,633.0720) and
(305.4157,656.0409) .. (343.2858,639.2604) -- cycle;
\path[draw=black,fill=cffffff,line join=miter,line cap=butt,even odd rule,line
width=0.800pt] (303.8643,605.1702) .. controls (299.0602,592.7724) and
(274.9892,631.3365) .. (282.6137,638.9261) .. controls (299.4140,655.6496) and
(309.4403,619.5599) .. (303.8643,605.1702) -- cycle;
\end{scope}
\begin{scope}[cm={{0.21261,0.0,0.0,0.20684,(191.99188,530.9718)}},fill=cffffff]
\path[draw=black,fill=cffffff,line join=miter,line cap=butt,even odd rule,line
width=.8pt] (343.2858,639.2604) .. controls (343.2858,639.2604) and
(269.9420,564.0695) .. (250.5061,585.6094) .. controls (207.6797,633.0720) and
(305.4157,656.0409) .. (343.2858,639.2604) -- cycle;
\path[draw=black,fill=cffffff,line join=miter,line cap=butt,even odd rule,line
width=.8pt] (303.8643,605.1702) .. controls (299.0602,592.7724) and
(274.9892,631.3365) .. (282.6137,638.9261) .. controls (299.4140,655.6496) and
(309.4403,619.5599) .. (303.8643,605.1702) -- cycle;
\end{scope}
\end{scope}

\end{tikzpicture}

	\caption{Tubular neighborhood $\cU_\alpha$ of depth $2$.}
	\label{figure3}
\end{figure}
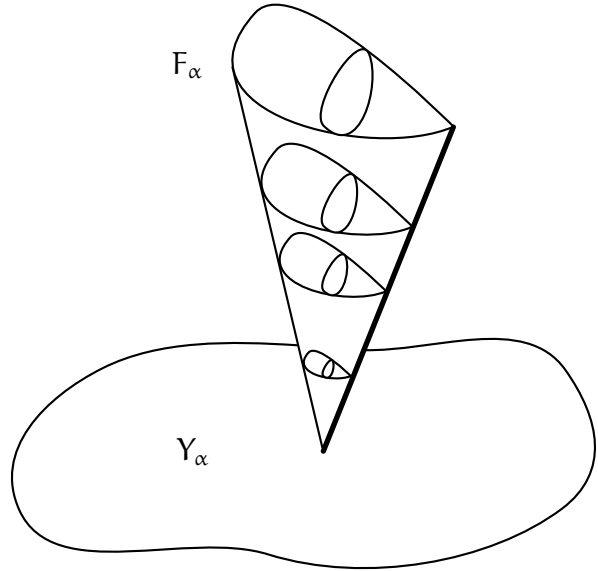

We also assume that on the level set level set $\{x=1\} \equiv \partial \cU_\alpha$
the fibration $\phi_\alpha|_{\partial \cU_\alpha}: (\partial \cU_\alpha, g_{F_\alpha} + \phi^{\ast}_\alpha g_{Y_\alpha})
\to (\phi_{\alpha} (\cU_\alpha), g_{Y_\alpha})$ is a Riemannian submersion. 
More precisely, we may split the tangent bundle $T_p \partial \cU_\alpha$ 
into vertical and horizontal subspaces $T^V_p \partial \cU_\alpha \oplus T^H_p \partial \cU_\alpha$. 
The vertical subspace $T^V_p \partial \cU_\alpha$ is the tangent space to the fibre of 
$\phi_\alpha|_{\partial \cU_\alpha}$ through $p$, and the 
horizontal subspace $T^H_p \partial \cU_\alpha$ is a corresponding complement.  
Then $\phi$ is a Riemannian submersion if $g_{F_\alpha}$ restricted to $T^H_p \partial \cU_\alpha$ vanishes. 
Any level set $(\{x\}\times \partial \cU_\alpha, x^2 g_{F_\alpha} + \phi^*g_B)$ yields then a 
Riemannian submersion as well. We put the same condition in the lower depth. Such metrics
always exist, see \cite[Proposition 3.1]{ALMP1}.
\medskip

In a collar neighborhood $\cU = [0,1) \times \partial M$ of the boundary $\partial M$, writing $x \in [0,1)$ for the radial function, we assume that $g$
is of the product form $g= dx^2 \oplus g_{\partial M}$, where $g_{\partial M}$ is an iterated cone-edge
metric on $\partial M$. Notice again that $\partial M$ is itself a regular part of a smoothly stratified space of depth $n-1$, 
without boundary. 

\subsection{Geometric Witt assumption}
Consider a smoothly stratified space $M$ with an iterated cone-edge 
metric $g$ as above. Assume that $\partial M = \varnothing$ for simplicity here.
Consider a tubular neighborhood $\cU_\alpha$ of any point in a singular stratum $Y_\alpha \in X_{n-2}$.
Then as in \eqref{iemetricUa}, $g$ takes the form
\begin{equation}
g|_{\cU_\alpha \cap M} = dx^2 + \phi^{\ast}_\alpha g_{Y_\alpha} + x^2 g_{F_\alpha} + h,
\end{equation}
where $h$ is a higher order term and the symmetric $2$-tensor 
$g_{F_\alpha}$ restricts to an iterated cone-edge metric $g_{F_\alpha}(y)$ on the link $F_\alpha$ at any 
base point $y \in  \phi_\alpha (\cU_\alpha)$. \medskip

Consider the Hodge Dirac, also called Gauss-Bonnet, operator 
$\slashed{\partial} := d+d^*$ associated to $(M,g)$, acting on compactly supported differential
forms $\Omega^*_c(\cU_\alpha)$. By \cite[(4.9)]{ALMP1}, $\slashed{\partial}$ takes the following form near the singular stratum
$$
\slashed{\partial} \equiv d + d^*  \sim \Gamma_\alpha \left( \frac{d}{dx} + \frac{1}{x} Q_\alpha(y) + T_\alpha(y) \right), 
$$
up to terms which are higher order in a certain sense. Here, $\Gamma_\alpha$ is a self-adjoint and 
unitary operator in $L^2\Omega^*((0,1) \times Y_\alpha, L^2\Omega^*(F_\alpha))$, $T_\alpha$ is a Dirac operator on $Y_\alpha$
and $Q_\alpha(y)$ is a family of symmetric operators in $L^2\Omega^*(F_\alpha, g_{F_\alpha}(y))$.
The geometric Witt assumption is a condition on the tangential operators $Q_\alpha(y)$ for any $y\in Y_\alpha$
and any $\alpha \in A$.

\begin{assump}\label{Witt} (Geometric Witt assumption)
The smoothly stratified space $\overline{M}$ with an iterated cone-edge metric $g$ on $M$ satisfies the geometric
Witt assumption if at each $y\in Y_\alpha \in X_{n-2}$ the tangential operator family $Q_\alpha(y)$ satisfies
$$
\spec Q_\alpha(y) \cap \left(-\frac{1}{2}, \frac{1}{2}\right) = \varnothing.
$$
The geometric Witt assumption implies the topological Witt assumption, 
that is the vanishing of the middle-degree intersection cohomology $IH^\nu(F_\alpha), \nu= \dim F_\alpha / 2$, for all links $F_\alpha$. 
Conversely, the topological Witt assumption yields the geometric Witt assumption after scaling of the metric.

\medskip
\noindent
We refer to $(M,g)$ in an abbreviated form as a stratified Witt space.
\end{assump}

\begin{remark}
We finish the section by pointing out that the related work \cite{HaLeVe1, HaLeVe}
uses a larger spectral gap to conclude that the domain of the unique 
self-adjoint extension of $\slashed{\partial}$ in $L^2\Omega^*(M,g)$ is 
identified with an {\em explicit} weighted edge Sobolev space. \end{remark}

\subsection{Galois coverings of stratified spaces}

In this subsection we recall for convenience of the reader 
 the preliminaries on Galois coverings of stratified Witt spaces, 
as discussed e.g. in \cite[\S 7.1]{PiVe} and \cite{PiZen}.
Consider as before the smoothly stratified Witt space $\overline{M}$ with boundary and 
the Riemannian metric $g$ on its regular part including boundary. 
Consider a Galois covering $\pi: \overline{M}_{\Gamma} \to \overline{M}$
with Galois group $\Gamma$ and  fundamental domain $\mathscr{F}_{\Gamma}$. 
\medskip

The topological and the smooth (Thom-Mather) stratifications on 
$\overline{M}_{\Gamma}$ are obtained in the following canonical way.
Decompose the covering $\overline{M}_{\Gamma}$ into the
preimages of the strata in $ \overline{M}$ under the projection $\pi$ . 
Surjectivity of $\pi$ ensures that each stratum in the covering is mapped 
surjectively onto the corresponding stratum in $ \overline{M}$. 
Since $\pi$ is  a local homeomorphism, it is 
straightforward to check that  $\overline{M}_{\Gamma}$ and 
its fundamental domain  are again topological 
stratified spaces. Pulling up the smooth stratification from the base, we obtain
a smooth Thom-Mather stratification on $\overline{M}_{\Gamma}$. 
By definition, the link of a point $\widetilde{p}\in \overline{M}_{\Gamma}$ is
equal to the link of its image in the base.
\medskip

We denote by $\widetilde{M}$ the regular part of 
$\overline{M}_{\Gamma}$ and observe that it is a Galois covering of the regular part $M$ of 
$\overline{M}$ with fundamental domain $\mathscr{F}$ equal to the regular part of $\mathscr{F}_{\Gamma}$.
Any singular stratum $\widetilde{Y}_\alpha$ of $\overline{M}_{\Gamma}$
is a Galois covering of a singular stratum $Y_\alpha$ of same depth with fundamental domain $\mathscr{F}_{Y_\alpha}$.
Similarly, the boundary $\partial \overline{M}_{\Gamma}$ of $\overline{M}_{\Gamma}$ is a Galois covering of $\partial \overline{M}$,
with regular part $\partial \widetilde{M}$ and fundamental domain $\mathscr{F}_{\partial M}$. 
The lift $\widetilde{g}$ of the iterated cone-edge metric $g$
defines a $\Gamma$-invariant iterated cone-edge metric on $\widetilde{M}$, which is still product near the boundary. 
Finally, there are isometric embeddings of $\mathscr{F}, \mathscr{F}_{Y_\alpha}, \mathscr{F}_{\partial M}$ into 
$M, Y_\alpha, \partial M$, respectively, with complements of measure zero.

\subsection{$L^2$-Stokes theorem for stratified spaces}\label{stokes}

Let $(X,h)$ be (the regular part of) a smoothly stratified space with an iterated cone-edge 
metric $h$ and boundary $\partial X$, such that $h$ is of a product form
$h= dx^2+ h_{\partial X}$ near the boundary with $(\partial X, h_{\partial X})$ being a smoothly 
stratified space with an iterated cone-edge metric as well. We do not assume here that $X$ or its boundary $\partial X$
is compact, in order to encompass the following three cases of interest: $(X,g)$ may be either a compact 
smoothly stratified space $(M,g)$ as introduced above, its Galois covering $(\widetilde{M},\widetilde{g})$,
or the half-cylinders $\partial M \times (-\infty,0]$ and $\partial \widetilde{M} \times (-\infty,0]$. 
\medskip

Denote by $L^2 \Omega^* (X)$ the $L^2$-completion (with respect to $h$) of smooth compactly 
supported differential forms $\Omega^*_c(X)$ on $X$. The space $L^2  \Omega^* (\partial X)$ is defined similarly
as the $L^2$-completion of smooth compactly supported differential forms 
on $\partial X$ with respect to $h_{\partial X}$. \medskip

We define the maximal domains for
the exterior differential $d$ and its formal adjoint $d^t$ acting on $\Omega^*_c(X)$ by
\begin{equation}
\begin{split}
&\dom_{\max}(d) := \{u \in L^2\Omega^*(X) \mid du \in L^2\Omega^*(X)\}, \\
&\dom_{\max}(d^t) := \{u \in L^2\Omega^*(X) \mid d^tu \in L^2\Omega^*(X)\}, 
\end{split}
\end{equation}
where $du$ and $d^tu$ are defined in the distributional sense. We write $d_{\max}$ and 
$d^t_{\max}$ for the corresponding closed extensions. We also introduce the minimal 
closed extensions $d_{\min}$ and $d^t_{\min}$ with respective domains
\begin{equation}
\begin{split}
&\dom_{\min}(d):= \{u\in \dom_{\max}(d) \mid \exists (u_n) \subset \Omega^*_c(X): 
u_n \xrightarrow{L^2} u, du_n \xrightarrow{L^2} du \}, \\
&\dom_{\min}(d^t):= \{u\in \dom_{\max}(d^t) \mid \exists (u_n) \subset \Omega^*_c(X): 
u_n \xrightarrow{L^2} u, d^tu_n \xrightarrow{L^2} d^tu \}.
\end{split}
\end{equation}
The maximal and minimal closed extensions are related by 
\begin{align}\label{minmax-relation}
(d^t)_{\max} \equiv (d_{\min})^*\,.
\end{align}
Our next theorem is well known to the experts, 
however the authors are not aware of any such result written up {\it explicitly} elsewhere.
We shall build on Cheeger \cite[Theorem 2.1]{Cheeger}.

\begin{thm}\label{Stokes-thm} \textup{($L^2$-Stokes theorem on Witt stratified spaces)}
Assume that $(X,h)$ and $(\partial X, h_{\partial X})$ satisfy the geometric Witt assumption.
Consider any smooth $u \in \dom_{\max}(d)$ and $v \in \dom_{\max}(d^t)$ such that their pullbacks to the 
boundary are $qu, qv \in L^2\Omega^*(\partial X)$. Then 
$$
\langle du, v\rangle_{L^2\Omega^*(X)} - \langle u, d^t v\rangle_{L^2\Omega^*(X)} = \langle qu, qv \rangle_{L^2\Omega^*(\partial X)}.
$$
\end{thm}

\begin{proof}
We define the double $X_d := X \cup_{\partial X} X$ with the natural involutive diffeomorphism 
$\alpha$, interchanging the two copies of $X$. This is again a smoothly stratified space, satisfying 
the geometric Witt assumption. Due to the product structure of $h$ near the boundary $\partial X$,
there exists an iterated cone-edge metric $h_d$ such that $\alpha^* h_d = h_d$ and $h_d\restriction X = h$.
\medskip

The notions of minimal and maximal closed extensions extends verbatim to the exterior
derivatives on $X_d$, $\partial X$ and $\partial X \times (-\infty,0]$. Whenever necessary, we indicate which space $d$
and $d^t$ act on, by putting that space in the brackets behind, e.g. $d_{\max}(X_d)$ denotes the maximal 
closed extension of the exterior derivative on $X_d$. Then by the geometric Witt assumption on 
$(X_d, h_d)$ and  $(\partial X, h_{\partial X})$
$$
d_{\max}(X_d) = d_{\min}(X_d), \quad d_{\max}(\partial X) = d_{\min}(\partial X).
$$
In particular, we obtain by \eqref{minmax-relation} the following equalities of closed extensions
\begin{equation}
\begin{split}
&d^t_{\max}(X_d) \equiv d^*_{\min}(X_d) = d^*_{\max}(X_d), \\
&d^t_{\max}(\partial X) \equiv d^*_{\min}(\partial X) = d^*_{\max}(\partial X), 
\end{split}
\end{equation}
This can be reformulated as the $L^2$-Stokes theorems
\begin{equation}\label{Stokes}
\begin{split}
&\forall \w \in \dom ( d_{\max}(X_d)), \eta \in \dom (d^t_{\max}(X_d)): \quad 
\langle d\w, \eta \rangle_{L^2\Omega^*(X_d)} = \langle \w, d^t \eta\rangle_{L^2\Omega^*(X_d)}, \\
&\forall \w' \in \dom (d_{\max}(\partial X)), \eta' \in \dom (d^t_{\max}(\partial X)): \ 
 \langle d\w', \eta' \rangle_{L^2\Omega^*(\partial X)} = \langle \w', d^t \eta' \rangle_{L^2\Omega^*(\partial X)}.
\end{split}
\end{equation}
The latter property, the $L^2$-Stokes theorem on $\partial X$ implies by Cheeger \cite[Theorem 2.1]{Cheeger} 
that for any $\alpha \in \dom (d_{\max}(\partial X\times (-\infty,0]))$ and $\beta \in \dom (d^t_{\max}(\partial X \times (-\infty,0]))$ 
with compact support, such that their pullbacks to the 
boundary $q\alpha, q\beta$ are in $L^2\Omega^*(\partial X)$, the $L^2$-Stokes theorem for collars holds
\begin{align}\label{Stokes-collars}
\langle d\alpha, \beta \rangle_{L^2\Omega^*(\partial X\times (-\infty,0])} - 
\langle \alpha, d^t \beta \rangle_{L^2\Omega^*(\partial X\times (-\infty,0])} = 
\langle q\alpha, q\beta \rangle_{L^2\Omega^*(\partial X)}.
\end{align}
Consider now the collar neighborhood $[0,1) \times \partial X \subset X$ of the boundary and
cutoff functions $\phi, \psi , \chi \in C^\infty_c[0,1)$ as in Figure \ref{fig:CutOff}.

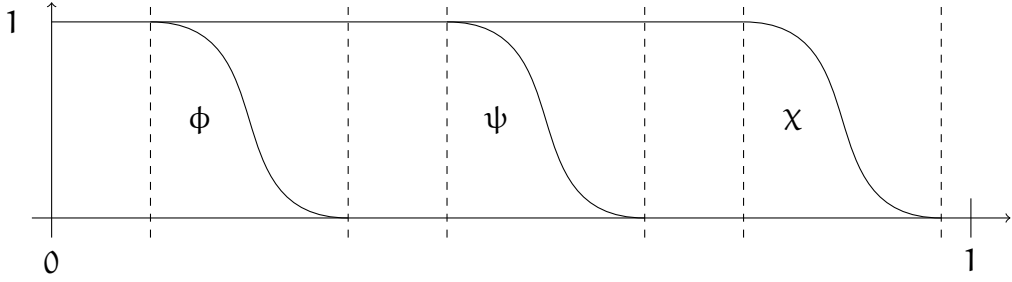
\begin{figure}[h]
\begin{center}

\begin{tikzpicture}[scale=1.3]
\draw[->] (-0.2,0) -- (9.7,0);
\draw[->] (0,-0.2) -- (0,2.2);

\draw (-0.2,2) node[anchor=east] {$1$};

\draw (9.3,-0.2) node[anchor=north] {$1$} -- (9.3,0.2);
\draw (0,-0.45) node {$0$};
\draw (0,2) -- (7,2);
\draw (1,2) .. controls (2.4,2) and (1.6,0) .. (3,0);
\draw[dashed] (1,-0.2) -- (1,2.2);
\draw[dashed] (3,-0.2) -- (3,2.2);
\draw (4,2) .. controls (5.4,2) and (4.6,0) .. (6,0);
\draw[dashed] (4,-0.2) -- (4,2.2);
\draw[dashed] (6,-0.2) -- (6,2.2);
\draw (7,2) .. controls (8.4,2) and (7.6,0) .. (9,0);
\draw[dashed] (7,-0.2) -- (7,2.2);
\draw[dashed] (9,-0.2) -- (9,2.2);

\draw (1.5,1) node {$\phi$};
\draw (4.5,1) node {$\psi$};
\draw (7.5,1) node {$\chi$};

\end{tikzpicture}

\caption{The cutoff functions $\phi, \psi$ and $\chi$.}
\label{fig:CutOff}
\end{center}
\end{figure}

The cutoff functions define functions on the collar $[0,1) \times \partial X \subset X$
and extend trivially to the interior of $X$. We still denote these extensions by 
$\phi, \psi , \chi \in C^\infty(X)$. We compute for any $u \in \dom_{\max}(d)$ and $v \in \dom_{\max}(d^t)$
\begin{equation}\label{Stokes-0}
\begin{split}
\langle du, v\rangle_{L^2\Omega^*(X)} &= \langle d (\psi u), v\rangle_{L^2\Omega^*(X)} + \langle d (1-\psi) u, v\rangle_{L^2\Omega^*(X)}
\\ &= \langle d (\psi u), \chi v\rangle_{L^2\Omega^*(X)} + \langle d (1-\psi) u, (1-\phi) v\rangle_{L^2\Omega^*(X)}.
\end{split}
\end{equation}
Before we proceed, we need a substitute of the Leibniz rule for the co-differential $d^t$
acting on a product of a smooth function and a differential form.
Taking such a product $\chi v$ from above as an example, we compute using linearity
of the Hodge star operator $*$ and the Leibniz rule for the exterior derivative $d$
\begin{equation}\label{Leibniz-dt}
\begin{split}
d^t (\chi v) &= \pm * d * (\chi v) = \pm * d (\chi  (* v))
\\ &= \pm * ( d\chi \wedge * v + \chi  d * v) = \pm * (d\chi \wedge * v) +  \chi  d^t v.
\end{split}
\end{equation}
The function $\psi u$ can be viewed as an element of $\dom (d_{\max}(\partial X\times (-\infty,0]))$.
The function $\chi v$ can be viewed as an element of $\dom (d^t_{\max}(\partial X\times (-\infty,0]))$.
Hence by the $L^2$-Stokes theorem on collars \eqref{Stokes-collars} we compute 
\begin{equation}\label{Stokes-1}
\begin{split}
\langle d (\psi u), \chi v\rangle_{L^2\Omega^*(X)} 
&= \langle \psi u, d^t (\chi v) \rangle_{L^2\Omega^*(X)} + \langle qu , qv \rangle_{L^2\Omega^*(\partial X)} \\
& = \langle \psi u, \chi d^t v \rangle_{L^2\Omega^*(X)} + \langle \psi u, d^t (\chi v) - \chi d^t v \rangle_{L^2\Omega^*(X)}
+ \langle qu , qv \rangle_{L^2\Omega^*(\partial X)}  \\
& = \langle \psi u, d^t v \rangle_{L^2\Omega^*(X)} +  \langle qu , qv \rangle_{L^2\Omega^*(\partial X)},
\end{split}
\end{equation}
where in the last equality we used that by construction $\chi \psi = \psi$, and by \eqref{Leibniz-dt}
the support of $d^t (\chi v) - \chi d^t v$ is contained in the support of $d\chi$, which is disjoint from support of $\psi$. 
This is why $\langle \psi u, d^t (\chi v) - \chi d^t v \rangle_{L^2\Omega^*(X)}$ vanishes.
\medskip

The function $(1-\psi) u$ can be viewed as an element of $\dom (d_{\max}(X_d))$.
The function $(1-\phi) v$ can be viewed as an element of $\dom (d^t_{\max}(X_d))$.
Hence by the first property in \eqref{Stokes} we compute 
\begin{equation}\label{Stokes-2}
\begin{split}
\langle d (1-\psi) u, (1-\phi) v\rangle_{L^2\Omega^*(X)} &= \langle (1-\psi) u, d^t ((1-\phi) v) \rangle_{L^2\Omega^*(X)}.
\\ & = \langle (1-\psi) u, (1-\phi) d^t v \rangle_{L^2\Omega^*(X)}
\\ & + \langle (1-\psi) u, d^t ((1-\phi ) v) -(1-\phi) d^t v \rangle_{L^2\Omega^*(X)} 
\\ & = \langle (1-\psi) u, d^t  v \rangle_{L^2\Omega^*(X)},
\end{split}
\end{equation}
where in the last equality we used that by construction $(1-\phi) (1-\psi) = (1-\psi)$, and again, by a computation as in
\eqref{Leibniz-dt}, the support of $d^t ((1-\phi ) v) -(1-\phi) d^t v$ is contained in the support of $d(1-\phi)$, which is disjoint from support of $(1-\psi)$. 
This is why $\langle (1-\psi) u, d^t ((1-\phi ) v) -(1-\phi) d^t v \rangle_{L^2\Omega^*(X)}$ vanishes.
By \eqref{Stokes-1} and \eqref{Stokes-2}, we conclude from \eqref{Stokes-0}
\begin{equation}
\begin{split}
\langle du, v\rangle_{L^2\Omega^*(X)} &= \langle (\psi u), d^t v\rangle_{L^2\Omega^*(X)} + \langle (1-\psi) u, d^t v\rangle_{L^2\Omega^*(X)}
+ \langle qu , qv \rangle_{L^2\Omega^*(\partial X)} \\ &= \langle u , d^t v\rangle_{L^2\Omega^*(X)} 
+ \langle qu , qv \rangle_{L^2\Omega^*(\partial X)}.
\end{split}
\end{equation}
This proves the statement.
\end{proof}

\section{Index theorems on stratified Witt spaces with boundary}\label{index-section}

Our current analysis is based on the arguments in our previous work \cite{PiVe}, where we 
have established index theorems on wedge manifolds with boundary, i.e. smoothly stratified Witt spaces 
with boundary, of depth one and with a cone-edge metric. 
Our analysis in the depth-one case used in a crucial way the heat kernel results in \cite{MazVer}.
The extension from wedge manifolds to stratified Witt spaces of arbitrary depth 
uses centrally the heat kernel construction by Albin and Gell-Redman \cite{AGR17} and we begin this Section by a 
brief summary of their results.


\subsection{The heat-kernel construction of Albin and Gell-Redman}\label{a-gr}
Let us 
 provide a short overview of the results in \cite{AGR17} used in this work. As the statements
 and the notations are somewhat long to write down, we refer the reader to the original paper for details.

\begin{enumerate}
\item Albin and Gell-Redman establish a microlocal description of the heat kernel as a polyhomogeneous
conormal function on an appropriate manifold with corners, obtained by 
an iterative sequence of parabolic blowups of $[0,\infty) \times \overline{M}^2$. 
This result has been obtained in \cite[Theorem 4.4]{AGR17}. \medskip

\item Albin and Gell-Redman establish the trace class property of the heat operator and short-time asymptotic
expansion of the heat trace. This result has been obtained in \cite[Theorem 1]{AGR17}. \medskip

\item Albin and Gell-Redman develop  Getzler rescaling  in the stratified setting and obtain an
improved short-time asymptotic expansion of the supertrace of the heat kernel. 
This result has been obtained in \cite[Corollary 5.7]{AGR17}
\end{enumerate}

\subsection{The Atiyah-Patodi-Singer index formula on Witt spaces with boundary}
We now continue and explain how the arguments in \cite{PiVe} extend to the 
setting of an even dimensional smoothly stratified Witt space $\overline{M}$ with boundary,
 with an iterated cone-edge metric $g$ on its regular part $M$, 
which is assumed to be product $g= dx^2 \oplus g_{\partial M}$
in a collar neighborhood of the boundary $\partial M$.
Consider the signature operator $D$, defined as in
\eqref{D}.  Exactly as in \eqref{D-product}, in the collar neighborhood of the boundary $D$ takes the form
\begin{align}\label{boundary}
D =  \sigma \left( \frac{d}{dx} + B\right),
\end{align}
where $\sigma$ is a bundle isomorphism and $B$ is the tangential operator acting on $\partial M$.
By the geometric Witt assumption, \cite[Theorem 1.1]{ALMP1} assert that $B$ is essentially self-adjoint with discrete spectrum. 
Consider the positive spectral projection $P_+(B)$ of $B$. We are going to fix a closed domain of $D$ by putting Atiyah-Patodi-Singer boundary 
conditions $P_+(B)$ at $\partial M$. (We do not need to impose boundary conditions at the singular 
strata of $M$, since the Hodge Dirac operator $\slashed{\partial}$, and hence also the signature operator
$D$, are essentially self-adjoint due to the geometric Witt assumption in Assumption \ref{Witt}. \footnote{As already pointed out
we pass from the Witt condition to the geometric Witt condition by suitably rescaling the metric.})
More precisely, we write $L^2\Omega^*(M,g)$ for the $L^2$-completion of either $\Omega_c^\pm(M)$ and 
define the maximal domain of $D$ by
\begin{equation}
\dom_{\max}(D) := \{u \in L^2\Omega^*(M,g) \mid Du \in L^2\Omega^*(M,g)\}.
\end{equation}
We want to single out the smooth subspace of elements in $\dom_{\max}(D)$ that are 
 smooth up to the boundary. 
We write $\Omega^*(M\cup\partial M)$ for the smooth forms on $M\cup \partial M $ that are  smooth up to the boundary, with
$M$ and $\partial M$  denoting the regular parts of 
the  corresponding smoothly stratified Witt spaces. 
We then define the core domain for $D$ with Atiyah-Patodi-Singer boundary 
conditions by
\begin{align}\label{core-domain}
\dom^c_+(D) := \{u \in \dom_{\max}(D) \cap \Omega^*(M\cup \partial M) \mid P_+(B) u|_{\partial M} = 0\}.
\end{align}
We fix the closed extension of $D$ with domain $\dom(D)$ defined as the 
graph-closure of the core domain $\dom^c_+(D)$ in $L^2\Omega^*(M,g)$. Then,
using the microlocal heat kernel description by Albin and Gell-Redman in \cite[Theorems 1 and 4.13]{AGR17}, 
we may proceed exactly as in \cite[Proposition 8.1]{PiVe} and obtain the McKean-Singer index formula.

\begin{prop}\label{index-trace}
$D$ is Fredholm with index 
\begin{align*}
\ind \, D = \textup{Tr} \, e^{-tD^*D} 
- \textup{Tr} \, e^{-tDD^*}. 
\end{align*}
\end{prop}

We can now proceed with the derivation of the index formula as in \cite[Theorem 8.4]{PiVe}, which 
is in turn directly based on the seminal work of Atiyah-Patodi-Singer. .
Recall that $B$ is discrete by \cite[Theorem 1.1]{ALMP1} and let $\{\lambda_n\}_{n\in \N_0}$ 
be an enumeration of the non-zero eigenvalues of $B$, 
counted with their multiplicities and ordered in ascending order. 
Denote by $\textup{sign}(\lambda)$ the sign of an eigenvalue $\lambda$. 
Then the eta function of $B$ is defined for $\Re(s) \gg 0$ by 
\begin{align*}
\eta(B,s) := \sum_{n=0}^\infty \textup{sign}(\lambda_n) \, | \lambda_n| ^{-s}
= \frac{1}{\Gamma((s+1)/2)}
\int_0^\infty t^{(s-1)/2} \, \textup{Tr} \, B e^{-tB^2} dt.
\end{align*}
Due to the short time asymptotic expansion of $\textup{Tr} \, B e^{-tB^2}$, which can be inferred
from \cite[Theorem 4.13]{AGR17} similarly to \cite[Theorem 1]{AGR17}, $\eta(B,s)$ extends 
meromorphically to the complex plane $\C$ with $s=0$ being, a priori a singular point. \medskip

However, exactly as in the work of Atiyah-Patodi-Singer \cite{APSa}, the regularity of 
the eta function at $s=0$ is connected to absence  of certain logarithmic terms in the heat-kernel
trace-asymptotic on the double $M_d$, cf. \cite[(8.9)]{PiVe}. In our previous work, see \cite[\S 8.2]{PiVe}, without the 
possibility of using Getzler rescaling, we could not 
exclude these terms in the heat trace asymptotics in general. Hence we concluded
that the eta-invariant exists only under the additional assumption that the dimension of the singular strata was even. 
\smallskip

Here, in view of the recent work by Albin and Gell-Redman \cite[\S 6]{AGR17}, see Subsection \ref{a-gr} for precise references,
one can apply Getzler rescaling to the heat kernel on the double and conclude that the eta function 
of the boundary operator $B$ is indeed regular at $s=0$. \medskip

All this establishes the regularity of the eta function at $s=0$ for the odd signature operator when it arises as a boundary operator. 
\smallskip

Consider the closed double $\overline{M}_d$ of $M$ 
with the natural involutive diffeomorphism $\alpha$
interchanging the two copies of $\overline{M}$. Consider the iterated cone-edge metric $g_d$
on the regular part $M_d$ of $\overline{M}_d$, such that $\alpha^* g_d=g_d$ and 
$g_d \restriction M = g$. Consider the signature operator $D_d$ of $(M_d, g_d)$.
Then by \cite[Theorem 6.4]{AGR17} the constant term in the short time asymptotic expansion of the trace of 
$\exp (-tD_d^*D_d) - \exp (-tD_dD_d^*)$ can be written as 
\begin{align}\label{heat-kernel-asymptotics}
\int\limits_M a_0 + \sum_{\alpha \in A} \int\limits_{\ Y_\alpha} b_{\alpha},
\end{align}
In this formula the integrand $a_0$ is the usual Hirzebruch L-form (computed with respect to the Levi-Civita
connection $\nabla^g$ associated to the metric $g$): $a_0=L(M,\nabla^g)$. With a small abuse of notation we shall simply write
$L(M)$ from now on.
The integrands $b_\alpha$ are worked out explicitly by Albin and Gell-Redman \cite{AGR17}
and involve the Hirzebruch L-form of the stratum $Y^\alpha$ and the J-form associated to  the link fibration over $Y_\alpha$.
We will be mainly interested in the fact that $a_0,b_\alpha$ are the same for the 
index theorems on $M$ and its Galois coverings; thus we refer the reader to \cite{AGR17}
for the exact formula of $b_\alpha$. \medskip

Now we can proceed exactly as in  \cite[Theorem 8.4]{PiVe}
and deduce the index formula below.

\begin{thm}\label{index-main}
\begin{align*}
\ind \, D = \int\limits_M L(M) + 
\sum_{\alpha \in A} \int\limits_{\ Y_\alpha} b_{\alpha}
- \frac{\dim \ker B + \eta(B)}{2}.
\end{align*}
\end{thm}

\begin{remark}
The above theorem holds unchanged for any Dirac-type operator $D$ for which the Albin-Gell-Redman version
of Getzler rescaling holds. Assume now that the boundary operator of $D$, call it again $B$, is invertible.
In the smooth case one can prove easily that the APS-index is then equal to the $L^2$-index of $D_\infty$,
the extension
of $D$ to $M_\infty$,  the manifolds with cylindrical ends associated to $M$. 
There are direct approaches to the $L^2$-index formula for $D_\infty$, most notably the one of Melrose 
through the $b$-calculus, see \cite{MelATP}.
Albin and Gell-Redman have extended this analysis to Witt spaces with cylindrical ends, establishing in particular an $L^2$-index
formula on Witt spaces with cylindrical ends (under the assumption that the boundary operator $B$ is invertible).
See \cite[Theorem 7.2]{AGR17}. Their result and Theorem \ref{index-main} (when the boundary operator is invertible) 
are of course compatible 
exactly as in the smooth case.
\end{remark}

\subsection{Existence of the eta invariant in the general case}

Consider now the general case, that is, the odd signature operator $D$ on an odd dimensional
Witt space $\overline{X}$  (not necessarily arising as the boundary of a Witt space with boundary). Here, we use again 
Getzler rescaling in the form provided by Albin and Gell-Redman \cite[\S 6]{AGR17}
in order to prove that the integral
\begin{equation}\label{eta-at-0}
\frac{1}{\sqrt{\pi}}\int_0^\infty \Tr(De^{-tD^2})\frac{dt}{\sqrt{t}}
\end{equation}
is convergent at $t=0$ (the convergence at infinity is clear from the discretness of the spectrum of $D$).
Indeed, the existing proofs of the convergence of  \eqref{eta-at-0}, see in particular  \cite[\S 8.13]{MelATP},
work perfectly well in the present singular context as we are now going to briefly explain. Let $K(t)$ be the Schwartz kernel
of $De^{-tD^2}$. We know from \cite[Proposition 5.5]{PiVe} that 
$\Tr(De^{-tD^2})$ equals the convergent integral  
\begin{equation}\label{lid}
\int_X  \textup{tr}_p K(t)(p,p) \textup{dvol}_{h} (p)\,.
\end{equation}
 Now, Melrose' proof, see \cite[Theorem 8.35]{MelATP}, in turn inspired by the original
 treatment by Bismut and Freed  \cite[Theorem 2.4]{BF}, connects the integrand $  \textup{tr}_p K(t)(p,p)$ in \eqref{lid} to the trace-asymptotic of the heat kernel 
 $H(t)$ on $X\times S^1$. Using Getzler rescaling for $H(t)$ one proves that
  $\textup{tr}_p K(t)(p,p)= \sqrt{t} g(t,p)$, with $g(t,p)\in C^\infty ([0,\infty)\times X)$ and of course integrable 
  for eact $t\geq 0$.  This shows that the $t$ integral 
  \eqref{eta-at-0} is convergent
  at $t=0$. Notice that, in fact,  Albin and Gell-Redman \cite{AGR17} have treated the more general case of families of Dirac operators
  satisfying the geometric Witt condition and proved the convergence 
of the Bismut-Cheeger eta form in this context, following \cite[Theorem 10.31]{BGV} instead of  \cite[Section 8.13]{MelATP}.\medskip

Summarizing this discussion we can finally state 
that it is possible to define the eta-invariant of the odd signature operator on a general Witt space of odd dimension.

\section{Index theorems on Galois coverings of stratified Witt spaces}\label{index-section2}

In this section we extend the statements of \S \ref{index-section} to Galois coverings of a smoothly
stratified Witt space with boundary. In order to shorten the total length of the paper, we refer the
reader to our previous work \cite[\S 7.2]{PiVe} for background material. The definitions in \cite[\S 7.2]{PiVe}
were given in the depth one case, but they can be easily generalized to arbitrary depth. Our characterizations for the ideal 
of $\Gamma$-Hilbert-Schmidt and $\Gamma$-trace class operators
can in fact be considered as definitions by Shubin \cite[\S 2.23 Theorems 1. and 3.]{Shubin2} and 
Atiyah \cite[\S 4]{Ati}. In order to recall the notation, let us recall the following.

\begin{defn}\label{HS-definition-galois}
An operator $A$ with Schwartz kernel $K_A \in L^2\Omega^*(\widetilde{M}\times \mathscr{F})$ is called 
$\Gamma$-Hilbert-Schmidt. We denote the space of $\Gamma$-Hilbert-Schmidt operators by $\calC_2^\Gamma(\widetilde{M})$.
A $\Gamma$-trace class operator is given by
$A = \sum B_j \circ C_j$ (finite sum) with $B_j, C_j \in \calC_2^\Gamma (\widetilde{M})$. We denote the space 
of $\Gamma$-trace class operators by $\calC_1^\Gamma (\widetilde{M}) \subset 
\calC_2^\Gamma (\widetilde{M})$. We define for any $A = B \circ C \in \calC_1^\Gamma 
(\widetilde{M})$ the $\Gamma$-trace in terms of the characteristic function $\phi$
of the fundamental domain $\mathscr{F}$ by 
\begin{equation}\label{trace-definition-galois}
\Tr_\Gamma(A) := \Tr (\phi A \phi) =
\iint\limits_{\mathscr{F} \times \widetilde{M}} \textup{tr}_p \, 
K_B(p,q) K_{C}(q,p) \, \textup{dvol}_{\widetilde{g}} (q) \, \textup{dvol}_{\widetilde{g}} (p).
\end{equation}
These definitions are independent of the choice of the fundamental domain $\mathscr{F} $.
\end{defn}

\begin{remark}\label{HS-remark-galois}
If the Schwartz kernel $K_A$ is continuous at the diagonal, we may replace the inner integral in \eqref{trace-definition-galois}
by $K_A(p,p)$ thus obtaining the $\Gamma$-trace of $A$ by integrating its Schwartz kernel at the
diagonal over the fundamental domain
\begin{equation}\label{gammatrace}
\Tr_\Gamma(A) := \Tr (\phi A \phi)=\int_{\mathcal{F}} \textup{tr}_p \,  K_A (p,p)\textup{dvol}_{\widetilde{g}} (p).
\end{equation}
This is generally wrong, however, if $K_A$ is not continuous at the diagonal. 
\end{remark}

\subsection{$\Gamma$-eta invariant on a covering of a stratified Witt space}\label{G-eta-subsection}
Consider the operator $B$ of $(\partial M,g_{\partial M})$ and the corresponding operator 
$\widetilde{B}$ of the covering $(\partial \widetilde{M}, \widetilde{g}_{\partial \widetilde{M}})$, which is the $\Gamma$-equivariant lift
of $B$.  Our analysis in fact also applies to Witt spaces that are not boundaries, i.e. we may consider any odd dimensional Witt space $(X,h)$  and a 
Galois $\Gamma$-cover $(\widetilde{X},\widetilde{h})\xrightarrow{\pi} (X,h)$, with $\widetilde{h}=\pi^* h$.  We consider the respective
odd signature operators $D$ and $\widetilde{D}$. \medskip

Assume the geometric Witt condition for $D$. Let $Y_\alpha$ be a singular stratum in $X$. Since the link $F_\alpha$ at any point $p \in Y_\alpha$ is equal 
to the link at any lift $\widetilde{p} \in \pi^{-1}(p) \subset \widetilde{Y}_\alpha$, the operators induced on the links for
$D$ and $\widetilde{D}$ coincide and so the geometric Witt assumption also  holds for $\widetilde{D}$. 
\medskip

Now, even though our previous work \cite{PiVe} is concerned with stratified Witt spaces of depth one,
the results of \cite[Proposition 7.3 and 7.4]{PiVe} which are based on \cite{ALMP3}, hold
and can be  formulated in the general setting of smoothly stratified Witt spaces. Hence  we have the following.
\begin{prop}\label{resolvent-trace-class}
The operator $\widetilde{D}$ is essentially self-adjoint. Its unique self-adjoint extension, 
denoted again by $\widetilde{D}$ 
satisfies the following properties.
\begin{enumerate}
\item If $2N > \dim X$, then $(\textup{Id}+ \widetilde{D})^{-N}$ is $\Gamma$-trace class.
\item The heat operator $e^{-t\widetilde{D}^2}$ and the operator $\widetilde{D}e^{-t\widetilde{D}^2}$ are $\Gamma$-trace class.
\end{enumerate}
\end{prop}
Then precisely as in \cite[Corollary 7.7]{PiVe}. we conclude 
that the Schwartz kernels of $e^{-t\widetilde{D}^2}$ and $\widetilde{D}e^{-t\widetilde{D}^2}$ are smooth 
in $\mathscr{F} \times \mathscr{F}$ for any fixed $t>0$, with $\mathscr{F}$ a  fundamental domain
for $\widetilde{X}$. We arrive in this way at the analogue 
of Lidski theorem in \cite[Proposition 7.8]{PiVe} for smoothly stratified Witt spaces.
\begin{prop}\label{lidskii-theorem-galois}
The operators $e^{-t\widetilde{D}^2}$ and $\widetilde{D}e^{-t\widetilde{D}^2}$ are $\Gamma$-trace class 
and their $\Gamma$-traces can be represented by integrals of their corresponding
Schwartz kernels (with a small abuse of notation, we denote the Schwartz kernels by 
the same symbol as the corresponding operators)
\begin{equation}\label{lidski}
\begin{split}
\Tr_\Gamma \left(e^{-t \widetilde{D}^2}\right)&=
\int_{\mathscr{F}} \textup{tr}_p \, \left(e^{-t \widetilde{D}^2}\right) (p, p) 
\textup{dvol}_{\widetilde{h}} (p) , \\
\Tr_\Gamma \left(\widetilde{D} e^{-t \widetilde{D}^2}\right)&=
\int_{\mathscr{F}}  \textup{tr}_p \, \left(\widetilde{D} e^{-t \widetilde{D}^2}\right) 
(p, p) \textup{dvol}_{\widetilde{h}} (p).
\end{split} 
\end{equation}
\end{prop}

\noindent The $\Gamma$-eta invariant of $\widetilde{D}$ is defined by the integral
\begin{align*}
\eta_\Gamma(\widetilde{D}) := \frac{1}{\sqrt{\pi}}
\int_0^\infty \textup{Tr}_\Gamma \, \widetilde{D} e^{-t \widetilde{D}^2} \frac{dt}{\sqrt{t}}.
\end{align*}

Convergence of the integral at infinity is discussed exactly as in \cite[(7.10)]{PiVe},
following \cite{Ram}. Alternatively, we can proceed as in L\"uck and Schick \cite{Lueck-Schick2} which builds
on an argument by Cheeger and Gromov. Convergence at $t=0$ is obtained
 by applying Getzler rescaling to the Schwartz kernel of $\widetilde{D} e^{-t \widetilde{D}^2}$; indeed, following again Melrose' proof,
 we can use Getzler rescaling as developed by Albin-Gell Redman, in order to show that 
 the function 
$$ \textup{tr}_p \, \left(\widetilde{D} e^{-t \widetilde{D}^2}\right) 
(p, p)$$
on  $ \mathscr{F}$ 
is equal to  $\sqrt{t} g$ with $g\in C^\infty ([0,\infty)\times \mathscr{F})$ and uniformly  integrable with respect to $\textup{dvol}_{\widetilde{h}} $
for each $t\geq 0$. \medskip

Summarizing: for the odd signature operator $\widetilde{D}$ on a Galois $\Gamma$-cover of an odd dimensional
Witt space, the $\Gamma$-eta invariant $\eta_\Gamma(\widetilde{D}) $ is well-defined. This result also holds for Dirac operators 
satisfying the geometric Witt assumption.

\subsection{Index theorem on a covering of a stratified Witt spaces with boundary}

In this final subsection we give a brief idea on how the index theorem of \cite[Theorem 9.4]{PiVe} on 
Galois coverings of wedge spaces (smoothly stratified Witt spaces of depth one) extends to the general case of 
smoothly stratified Witt spaces of general depth. Consider the signature operator $\widetilde{D}$ 
of $(\widetilde{M}, \widetilde{g})$, which is the $\Gamma$-equivariant
lift of the signature operator $D$ of $(M,g)$. Near the boundary $\partial \widetilde{M}$, the signature operator
takes the following product form (cf. \eqref{boundary})
\begin{align}
\widetilde{D} =  \widetilde{\sigma} \left( \frac{d}{dx} + \widetilde{B}\right),
\end{align}
where $\widetilde{\sigma}$ is the equivariant lift of $\sigma$ to $\partial \widetilde{M}$.
The positive spectral projection $P_+(\widetilde{B})$ of $\widetilde{B}$ is defined 
using the Browder-Garding spectral decomposition of $\widetilde{B}$ as follows. 
The Browder-Garding spectral decomposition asserts for $\widetilde{B}$, and
in fact for any self-adjoint operator in $L^2\Omega^*(\partial \widetilde{M}) \equiv L^2\Omega^*(\partial \widetilde{M}, 
\widetilde{g}_{\partial \widetilde{M}})$, the following result,
cf. \cite[Theorem 2.2.1]{Ram}. 

\begin{thm}\label{Browder-Garding}
There exists a sequence $\{e_n:\R \times \partial \widetilde{M} \to \Lambda^*T^*\partial \widetilde{M}\}_{n\in \N}$
of maps, which are measurable and define for each fixed $\lambda \in \R$ and $n\in \N$
a smooth $L^2$-integrable differential form $e_n(\lambda, \cdot) \in \Omega^*(\partial \widetilde{M})$, 
such that 
$$
\widetilde{B} e_n(\lambda, \cdot) = \lambda e_n(\lambda, \cdot).
$$
Moreover there exists a sequence $\{\mu_n\}_{n\in \N}$ of measures on $\R$ such that
for any smooth compactly supported differential form $s \in \Omega^*_c(\partial \widetilde{M})$ the map 
$$
(V s)_n(\lambda) := \int_{\partial \widetilde{M}} ( s(p), e_n(\lambda, p) )_{\widetilde{g}_{\partial \widetilde{M}}}
\textup{dvol}_{\widetilde{g}_{\partial \widetilde{M}}}(p),
$$
extends to an isometry of Hilbert spaces $V: L^2\Omega^*(\partial \widetilde{M}) \to \oplus_{n} L^2(\mu_n)$, i.e.
$$
\int_{\partial \widetilde{M}} \| s(p) \|^2_{\widetilde{g}_{\partial \widetilde{M}}}
\textup{dvol}_{\widetilde{g}_{\partial \widetilde{M}}}(p) = \sum_n \int_{\R}
|(Vs)_n(\lambda)|^2 d\mu_n(\lambda).
$$
Moreover, for any bounded $f:\R \to \R$, the operator $f(\widetilde{B})$
on $L^2\Omega^*(\partial \widetilde{M})$ is defined by 
\begin{align}
(V \, f(\widetilde{B}) s)_n(\lambda) := \int_{\partial \widetilde{M}} f(\lambda)
( s(p), e_n(\lambda, p) )_{\widetilde{g}_{\partial \widetilde{M}}}
\textup{dvol}_{\widetilde{g}_{\partial \widetilde{M}}}(p).
\end{align}
\end{thm}

Now the positive spectral projection $P_+(\widetilde{B})$ is defined as $f(\widetilde{B})$ with $f(\lambda)$
equal to $\lambda$ for $\lambda \geq 0$, and equal to zero for $\lambda < 0$.
As in the case of $D$, we fix a closed domain of $\widetilde{D}$ by putting Atiyah-Patodi-Singer boundary 
conditions $P_+(\widetilde{B})$ at $\partial \widetilde{M}$. 
More precisely, we define the maximal domain of $\widetilde{D}$ by
\begin{equation}
\dom_{\max}(\widetilde{D}) := \{u \in L^2\Omega^*(\widetilde{M},\widetilde{g}) \mid 
\widetilde{D}u \in L^2\Omega^*(\widetilde{M},\widetilde{g})\}.
\end{equation}
Also in this case we want to single out the smooth subspace of elements in $\dom_{\max}(\widetilde{D})$ that are 
 up to the boundary. We introduce $\Omega^*(\widetilde{M}\cup \partial \widetilde{M})$, the space of smooth differential
 forms on $ \widetilde{M}\cup \partial \widetilde{M}$, smooth up to the boundary.
We then define the core domain for $\widetilde{D}$ with Atiyah-Patodi-Singer boundary 
conditions by
\begin{align}\label{core-domain-tilde}
\dom^c_+(\widetilde{D}) := \{u \in \dom_{\max}(\widetilde{D}) \cap \Omega^*(\widetilde{M}\cup \partial \widetilde{M}) \mid 
P_+(\widetilde{B}) u|_{\partial \widetilde{M}} = 0\}.
\end{align}
We fix the closed extension of $\widetilde{D}$ with domain $\dom(\widetilde{D})$ defined as the 
graph-closure of the core domain $\dom^c_+(\widetilde{D})$ in $L^2\Omega^*(\widetilde{M},\widetilde{g})$. 
Then the arguments of \cite[Proposition 9.1 and 9.2, Theorem 9.4]{PiVe} 
carry over verbatim to the case of smoothly stratified Witt spaces, 
where we use the microlocal heat kernel construction in \cite{AGR17} on stratified Witt spaces instead of its 
special case in \cite{MazVer} for depth one case. Thus we obtain the following
non-compact analogue of Theorem \ref{index-main}.

\begin{thm}\label{index-main-2}
The heat operators of $\widetilde{D}^* \widetilde{D}$ and $\widetilde{D} \widetilde{D}^*$, as well as the 
orthogonal projections $P_{\ker \widetilde{D}}$ and $P_{\ker \widetilde{D}^*}$ of $L^2\Omega^*(\widetilde{M}, \widetilde{g})$
onto the kernel of $\widetilde{D}$ and $\widetilde{D}^*$ respectively, are $\Gamma$-trace class.
The operator $\widetilde{D}$ is $\Gamma$-Fredholm,
i.e. admits a finite $\Gamma$-index
\begin{equation*}
\ind_\Gamma \widetilde{D} := \textup{Tr}_\Gamma (P_{\ker \widetilde{D}}) 
-  \textup{Tr}_\Gamma (P_{\ker \widetilde{D}^*})
\end{equation*}
and the following McKean-Singer formula holds:
\begin{equation*}
\ind_\Gamma \widetilde{D} = \textup{Tr}_\Gamma \left(e^{-t(\widetilde{D})^*\widetilde{D}}\right) 
- \textup{Tr}_\Gamma \left(e^{-t \widetilde{D}\widetilde{D}^*}\right) 
\end{equation*}
Its $\Gamma$-index can be computed in terms of the regularized eta invariant $\eta_{\Gamma}(\widetilde{B})$ by 
\begin{align*}
\ind_\Gamma \, \widetilde{D} = \int\limits_M L(M) + 
\sum_{\alpha \in A} \int\limits_{\ Y_\alpha} b_{\alpha}
- \frac{\dim_\Gamma \ker \widetilde{B} + \eta_{\Gamma}(\widetilde{B})}{2}.
\end{align*}
\end{thm}

\begin{remark}
The integrands $b_\alpha$ are the same as before, and are worked out explicitly by Albin and Gell-Redman \cite{AGR17}.
As already remarked, we are only interested in the fact that $b_\alpha$ are the same for the 
index theorems on $M$ and its Galois covering $\widetilde{M}$, and for this reason we refer the reader to \cite{AGR17}
for the exact expression of $b_\alpha$. 
\end{remark}

\section{Signatures of stratified Witt spaces with boundary}\label{signature-section}

We consider the regular part $M$ of a smoothly stratified Witt space $\overline{M}$
with boundary $\partial \overline{M}$. We denote as usual by $M$ and $\partial M$ the respective regular parts. We assume the
existence of an iterated cone-edge metric $g$ of product form $dx^2 \oplus g_{\partial M}$ in a collar of $\partial M$. 
The metric $g_{\partial M}$ is itself an iterated cone-edge metric on $\partial M$.
As always throughout the paper, we assume that $\overline{M}$ is Witt 
and rescale the metric so that we have  the geometric Witt assumption on $M$.
Notice that, in particular, $\partial M$  also satisfies the geometric Witt assumption. We shall briefly refer to $M$
and $\partial M$ as Witt spaces.

\medskip
We denote by $M_\infty = (M\cup \partial M)  \cup_{\partial M} ((-\infty,0] \times \partial M)$ the associated Witt space with cylindrical ends.
This is in fact the regular part of $\overline{M}_\infty:= \overline{M}  \cup_{\partial \overline{M}} ((-\infty,0] \times \partial \overline{M})$.
The metric $g$ extends smoothly to an iterated cone-edge metric $g_\infty$ on $M_\infty$ with $g_\infty \restriction M = g$
and $g_\infty \restriction (-\infty,0] \times \partial M = dx^2 \oplus g_{\partial M}$. \\
\medskip

\subsection{Various Hilbert complexes and their cohomologies}
In this subsection we want to introduce the relevant Hilbert complexes that will be used in the sequel. All these complexes 
are obtained as suitable closed extensions of the complex of smooth differential forms with compact support
on the regular part of our stratified Witt space. This discussion is based on the seminal work by Br\"uning and Lesch
\cite{BrLe}. 

\subsubsection{Hilbert complexes on a closed Witt space $\overline{X}$}
Consider a closed Witt space $\overline{X}$, e.g. $\overline{X} = \partial \overline{M}$. We denote by $X$ the regular part of $\overline{X}$,
equipped with an iterated cone-edge metric $g_X$.
We start with the de Rham complex  $(\W_c^*(X),d)$ of smooth compactly supported differential
forms on the regular part $X$. We denote by $L^2\Omega^*(X)$
the $L^2$-completion of $\W_c^*(X)$ with respect to the volume form of $g_{X}$. 
Consider the maximal and the minimal closed extensions of the exterior derivative $d$:
\begin{equation}\label{min-max}
\begin{split}
&\dom_{\textup{max}}(d) := \{u \in L^2\Omega^*(X) \mid du \in L^2\Omega^*(X)\}, \\
&\dom_{\textup{min}}(d) := \{u \in \dom_{\textup{max}}(d) \mid \exists (u_n) \subset \W_c^*(X): 
u_n \xrightarrow{L^2} u, du_n \xrightarrow{L^2} du\},
\end{split}
\end{equation}
where $du$ for any $u \in \dom_{\textup{max}}(d)$ is defined in the distributional sense.
By the Witt assumption, this complex admits a {\it unique} closed extension, or,
in the language of Br\"uning-Lesch, a unique  ideal boundary condition, i.e.
\begin{align}\label{uniqueness-partial-M}
\dom_{\textup{min}}(d) = \dom_{\textup{max}}(d)=:\dom(X).
\end{align} 
Abusing notation we keep the same symbol for the unique closed 
extension of $d$ and  denote the associated Hilbert complex by
\begin{equation*}
(\dom^*(X),d)
\end{equation*}
Thanks to  the weak Kodaira decomposition, which holds on all 
Hilbert complexes, see \cite[Lemma 2.1]{BrLe}, we can write
\begin{equation}\label{Kodaira}
\begin{split}
&L^2\Omega^*(X) = \cH^*(X) \oplus \overline{\image d} \oplus \overline{\image d^*}, 
\\ &\textup{where} \ \cH^* (X)  = \ker d \cap \ker d^*.
\end{split}
\end{equation}
We remark that for the harmonic forms $\cH^*(X)$ we have the isomorphism
\begin{equation*}
\cH^*(X)  \cong \frac{(\image d^*)^\perp}{\overline{\image d}} = \frac{\ker d}{\overline{\image d}}	
=: \cohr^*(X),
\end{equation*}
where $\cohr^*(X)$ is the {\it reduced} $L^2-$cohomology 
of $X$. One defines {\it the} $L^2$-cohomology of $X$ as 
the cohomology of the Hilbert complex $(\dom^*(X),d)$ by
$$
\coh^*(X) := \frac{\ker d}{\image d}.
$$
Thus $\coh^*(X)$ and $\cH^*(X)$
are not in general equal. In fact, see \cite[Theorem 2.4, Corollary 2.5]{BrLe}, we have that 
\begin{equation*}
\dim \coh^*(X) < \infty \Rightarrow \coh^*(X) \cong \cH^*(X), 
\end{equation*} 
which is true in our compact setting, see \cite[Theorem 6.1]{Cheeger},
\cite[Theorem 1.1]{ALMP1}.
Finally, we also introduce a smooth subcomplex of $(\dom^*(X),d)$.
We denote smooth (and otherwise unrestricted) differential forms on $X$ by $\W^*(X)$
and define the smooth subcomplex exactly  as Br\"uning and Lesch \cite[(3.14)]{BrLe} by
\begin{equation}
\begin{split}
(\mathscr{E}^*(X),d), \ \textup{where} \ \mathscr{E}^*(X) := \dom^*(X) \cap \W^* (X).
\end{split}
\end{equation}
Then \cite[Theorem 3.5]{BrLe} asserts that the inclusion
$\mathscr{E}^*(X) \hookrightarrow \dom^*(X)$
induces an isomorphism on cohomology
\begin{equation}\label{smooth-coh1}
H^*(\mathscr{E}^*(X),d) \cong \coh^*(X). 
\end{equation}
\subsubsection{Hilbert complexes on a Witt space with boundary $\overline{M}$}
We consider the de Rham complex $(\W^*_c (M),d)$ of smooth compactly supported 
differential forms over $M$, the regular part of a Witt space with boundary $\overline{M}$. This complex admits different ideal boundary conditions,
among which we single out the following two, defined exactly as in \eqref{min-max}
\begin{equation*}
(\dom_{{\rm min}}^*(M),d)
\;\;\text{and}\;\;(\dom_{{\rm max}}^*(M),d). 
\end{equation*}
We denote the cohomology associated to the first Hilbert complex $(\dom_{{\rm min}}^*(M),d)$, the minimal one,
as  $\coh^*(M,\partial M)$; we denote the cohomology associated to the second Hilbert complex $(\dom_{{\rm max}}^*(M),d)$,
the maximal one, as $\coh^*(M)$. The notation is inspired from the case where $M$ is
a smooth manifold with boundary, see \cite[Theorem 4.1]{BrLe}. In fact, as already remarked in the Introduction, these $L^2$-cohomologies can be 
identified with relative and absolute intersection cohomologies, see Bei \cite[Theorem 4]{Bei}. 
 \medskip

Exactly as in \eqref{Kodaira}, we have a Kodaira decomposition for each Hilbert complex 
$(\dom_{{\rm min}}^*(M),d)$ and $(\dom_{{\rm max}}^*(M),d)$, defining each the corresponding space of 
harmonic forms. We denote these two spaces of harmonic forms as 
\begin{equation*}
 \cH^*(M,\partial M) \;\;\;\text{and}\;\;\;\cH^*(M),
\end{equation*}
respectively. As before, harmonic forms and $L^2$-cohomologies coincide if the 
latters are finite-dimensional, which is true in the compact setting considered here, see \cite[Theorem 4]{Bei}. 
Notice that later in this subsection we shall give a purely analytic proof of the finite dimensionality of 
 $\coh^*(M,\partial M)$ and  $\coh^*(M)$, see Remark \ref{remark: finite-dimensionality}.
Hence we find
\begin{equation*}
 \cH^*(M,\partial M) \cong \coh^*(M,\partial M) \;\;\;\text{and}\;\;\;\cH^*(M) \cong \coh^*(M).
\end{equation*}

We also need to consider  smooth subcomplexes of $(\dom_{{\rm min} / {\rm max}}^*(M),d)$.
To do so, we consider the smooth differential forms $\Omega^*(M \cup \partial M)$. The differential forms $\Omega^*(M \cup \partial M)$
are smooth up to the boundary $\partial M$ and otherwise unrestricted in their behaviour at the 
singular strata. Now we can define, in a slightly different way than in
Br\"uning and Lesch \cite[(3.14)]{BrLe}, core complexes of minimal and maximal differential forms that
are smooth in $M$ and extend smoothly up to the (regular) boundary $\partial M$:
\begin{equation}\label{core-up-to-bdry}
\begin{split}
&(\mathscr{E}_{{\rm min}}^*(M),d), \ \textup{where} \ \mathscr{E}_{{\rm min}}^*(M) := \dom_{{\rm min}}^*(M) \cap \W^* (M \cup \partial M), \\
&(\mathscr{E}_{{\rm max}}^*(M),d),  \ \textup{where} \ \mathscr{E}_{{\rm max}}^*(M) := \dom_{{\rm max}}^*(M) \cap \W^* (M \cup \partial M).
\end{split}
\end{equation}
Note that these \emph{core complexes} are different from those in \cite[(3.14)]{BrLe}, which are
defined by requiring smoothness only in the open interior. However, our choice here allows to 
apply the restriction to the boundary map. We claim in the next lemma, using \cite[(4.12)]{BrLe}, that the core complexes $(\mathscr{E}_{{\rm min}}^*(M),d)$ and 
$(\mathscr{E}_{{\rm max}}^*(M),d)$ have the same cohomology as $\dom_{{\rm min}}^*(M)$ and 
$\dom_{{\rm max}}^*(M)$, respectively. 

\begin{prop}\label{core-subcomplexes-up-to-bdry}
\begin{equation}\label{smooth-coh2}
H^*(\mathscr{E}_{{\rm min}}^*(M),d) \cong \coh^*(M,\partial M) \;\;\;\text{and}\;\;\;
H^*(\mathscr{E}_{{\rm max}}^*(M),d) \cong \coh^*(M).
\end{equation}
\end{prop}

\begin{proof}
We employ an argument of \cite[(4.12)]{BrLe}. 
Consider the double manifold 
$$M_d = (M \cup \partial M) \cup_{\partial M} (M \cup \partial M)$$
with the natural involutive diffeomorphism $\alpha$
interchanging the two copies of $M$. The double $M_d$ is again a smoothly 
stratified Witt space. Equip $M_d$ with a metric $g_d$ such that $\alpha^*g= g$ and $g_d \restriction M = g$. 
Due to the product structure of $g$ in a collar of the boundary $\partial M$, $g_d$ exists and is again an iterative cone-edge metric. 
As in the case of the closed Witt space $(X, g_{X})$, the double $(M_d, g_d)$ has a unique ideal boundary condition, i.e. 
\begin{align}
\dom_{\textup{min}}(M_d) = \dom_{\textup{max}}(M_d)=:\dom_{(2)}^*(M_d).
\end{align} 

The diffeomorphism $\alpha$ induces an involution on the complex $(\dom_{(2)}^*(M_d), d)$
and hence we obtain a decomposition into $(\pm 1)$-eigenspaces of $\alpha$
$$
\dom_{(2)}^*(M_d) = \dom_{(2)}^+(M_d) \oplus \dom_{(2)}^-(M_d),
$$
where $\alpha \restriction \dom_{(2)}^\pm(M_d) = \pm \textup{Id}$. We want to prove that 
for $\W^*(M_d)$ denoting the space of smooth differential forms on $M_d$, 
$\mathscr{E}^-(M_d) := \dom_{(2)}^-(M_d) \cap \W^*(M_d)$ satisfies exactly as in 
\cite[(4.12)]{BrLe}
\begin{equation}\label{different-restrictions}\begin{split}
\mathscr{E}^-(M_d) \restriction M \cup \partial M &= \dom_{{\rm min}}^*(M) \cap \W^* (M \cup \partial M) \Bigl( = \mathscr{E}_{\min}^*(M) \Bigr), \\
\mathscr{E}^-(M_d) \restriction M &= \dom_{{\rm min}}^*(M) \cap \W^* (M)
\end{split}\end{equation}
It obviously suffices to prove the first equality.
Consider any $\widetilde{\w} \in \mathscr{E}^-(M_d)$. In a collar neighborhood 
$\cU \cong (-\varepsilon, \varepsilon) \times \partial M$
of the join $\partial M \subset M_d$, the form $\widetilde{\w}$ decomposes as 
$$
\widetilde{\w} = \w_0(x) + \w_1(x) \wedge dx, \quad \w_0, 
\w_1 \in C^\infty((-\varepsilon, \varepsilon), \Omega^*(\partial M)).
$$
Since $\alpha^*\widetilde{\w} = -\widetilde{\w}$ by construction, $\w_0(x) = -\w_0(-x)$ and $\w_1(x)=\w_1(-x)$. 
In particular, $\w_0(0)=0$. Pick any $\phi \in C^\infty_0(\R)$, with $\phi \equiv 1$
over $(-\varepsilon / 2, \varepsilon / 2)$. We define $\phi_n(x):= \phi (nx)$ for any $n \in \N$. 
This defines a smooth function on $\cU \cong (-\varepsilon, \varepsilon) \times \partial M$
which we extend trivially by zero to a smooth function on $M_d$. We set $\widetilde{\w}_n:= (1-\phi_n) \widetilde{\w}$.
Clearly $\widetilde{\w}_n \to \widetilde{\w}$ in $L^2\Omega^*(M_d)$. Moreover, we compute over $\cU$
$$
d\widetilde{\w}_n = (1-\phi_n) d\widetilde{\w} - \frac{\partial \phi_n}{\partial x} \, dx \wedge \w_0.
$$
The second summand converges to zero in $L^2\Omega^*(M_d)$, since $\w_0$ is smooth at $x=0$ with $\w_0(0)=0$. 
Hence $d\widetilde{\w}_n$ converges to $d\widetilde{\w}$ in $L^2\Omega^*(M_d)$. Consequently, setting 
$\w := \widetilde{\w} \restriction M \in \Omega^*(M)$, and $\w_n := \widetilde{\w}_n \restriction M$,
we conclude that $\w_n \to \w$ and $d\w_n \to d\w$ in $L^2\Omega^*(M)$ and hence $\w \in \mathscr{E}_{\min}^*(M)$ by definition.
Thus we find
$$
\mathscr{E}^-(M_d) \restriction M \cup \partial M \subseteq \mathscr{E}_{\min}^*(M).
$$
The converse inclusion is obtained as follows. Consider any $\w \in \mathscr{E}_{\min}^*(M)$. By definition there
exists a sequence $(\w_n)$ of smooth compactly supported differential forms on $M$ that converges in graph norm to $\w$.
Since the forms are compactly supported away from $\partial M$, we extend $(\w_n)$ to a sequence of forms $(\widetilde{\w}_n) \subset \mathscr{E}^-(M_d)$ on $M_d$.
Now, $\mathscr{D}^-_{(2)}(M_d)$ is a Hilbert complex and thus $(\widetilde{\w}_n)$ converges in graph norm to $\widetilde{\w} \in \mathscr{D}^-_{(2)}(M_d)$.
Since $\widetilde{\w}$ restricts to a smooth form $\w$, we conclude $\widetilde{\w} \subset \mathscr{E}^-(M_d)$.
This proves \eqref{different-restrictions}. Now the restriction maps 
\begin{align*}
&\mathscr{E}^-(M_d) \to \mathscr{E}^-(M_d) \restriction M \cup \partial M = \mathscr{E}_{\min}^*(M), 
\\ &\mathscr{E}^-(M_d) \to \mathscr{E}^-(M_d) \restriction M = \dom_{{\rm min}}^*(M) \cap \W^* (M),
\end{align*}
are well defined and are isomorphisms of complexes. Thus the cohomologies of $\mathscr{E}_{\min}^*(M)$
and $\dom_{{\rm min}}^*(M) \cap \W^* (M)$ agree. The latter is a core subcomplex 
exactly as in \cite[Theorem 3.5]{BrLe} and hence the first statement follows. The second statement follows
similarly by relating $ \dom_{(2)}^+(M_d) \cap \W^*(M_d)$ to $ \mathscr{E}_{\max}^*(M)$ 
as in \cite[(4.12)]{BrLe}.
\end{proof}

\begin{remark}\label{remark: finite-dimensionality}
Notice that exactly as in \cite[Theorem 4.1]{BrLe} the above proof establishes analytically, without using
Bei's result,  that 
$H^*(\mathscr{E}_{{\rm min}}^*(M),d)$ (and thus  $\coh^*(M,\partial M)$) and 
$H^*(\mathscr{E}_{{\rm max}}^*(M),d)$  (and thus $\coh^*(M)$) are finite dimensional; indeed, we have 
identified these cohomology groups with subcomplexes of the $L^2$-de Rham complex on the double and we 
know that the latter  has finite dimensional cohomology.
\end{remark}

\subsubsection{Hilbert complexes on the space $M_\infty$ with cylindrical ends}
Next we consider the de Rham complex $(\W_c^*(M_\infty),d)$ of smooth compactly supported 
differential forms in $M_\infty$. As before we may define the minimal and maximal closed extensions
$(\dom_{{\rm min} / {\rm max}}^*(M_\infty),d)$ and begin with the following fundamental observation.

\begin{prop}
Under the Witt assumption on $M$, 
the de Rham complex $(\W_c^*(M_\infty),d)$ has a unique closed extension in $L^2$, denoted
by $(\dom_{(2)}^*(M_\infty),d)$, i.e.
\begin{align}
\dom_{\min}(M_\infty) = \dom_{\max}(M_\infty)=:\dom_{(2)}^*(M_\infty).
\end{align} 
\end{prop}
	
\begin{proof}
Let us write $L^2\Omega^*(M_\infty)$ for the $L^2$-completion of $\W_c^*(M_\infty)$.
Uniqueness of domains $\dom_{\textup{min}}(M_\infty) = \dom_{\textup{max}}(M_\infty)$ is equivalent to the 
$L^2$-Stokes theorem, which states that for any $u \in \dom_{\min}(d) \equiv \dom_{(2)}(M_\infty)$ and 
any $v \in \dom_{\min}(d^t)$ we have 
$$
(du, v)_{L^2\Omega^*(M_\infty)} = (u, d^t v)_{L^2\Omega^*(M_\infty)}.
$$
This holds by Theorem \ref{Stokes-thm} since the boundary of $M_\infty$ is empty.
\end{proof}

As before, we write $d$ for the unique closed extension of the exterior derivative.
Define the $L^2$-harmonic forms, the reduced and non-reduced $L^2$-cohomologies of the 
Hilbert complex $(\dom_{(2)}^*(M_\infty),d)$ by 
\begin{equation*}
\cH_{(2)}^*(M_\infty) := \ker d \cap \ker d^*, \quad \cohr^*(M_\infty):= \frac{\ker d}{\overline{\image d}},
\quad \coh^*(M_\infty) := \frac{\ker d}{\image d},
\end{equation*}
respectively. We have as before $\cH_{(2)}^*(M_\infty) \cong \cohr^*(M_\infty)$.
Notice that $\cohr^*(M_\infty)$ and $\coh^*(M_\infty)$ do not 
coincide in general, due to non-compactness of  $M_\infty$.
By Proposition \ref{prop3_2} below, see also \eqref{useful0} and \eqref{useful1}, we know that $\cH_{(2)}^*(M_\infty)$ is finite dimensional.

\subsection{Maps between various Hilbert complexes and their cohomologies}\label{subsect:maps}

The discussion in this subsection is inspired by L\"uck and Schick \cite{Lueck-Schick}.
We shall be interested in defining certain homomorphisms between the cohomology groups introduced above.
As one of these homorphisms is given by "restriction to the boundary", which is problematic 
in $L^2$, we will use the smooth subcomplexes for the definition of that restriction. 

\begin{defn}\label{define-maps} 
Obvious inclusions and restrictions define the following maps.
\begin{enumerate}
\item Consider the natural map $r: \dom_{(2)}^*(M_\infty) \to \dom_{{\rm max}}^*(M)$ given by
restriction to $M\subset M_\infty$. The map $r$ commutes with $d$ and hence for any $\w \in \cH_{(2)}^*(M_\infty)$,
$r \w \in \ker d \subset \dom_{{\rm max}}^*(M)$. Taking the corresponding $L^2$-cohomology class 
$[r \w] \in \coh^* (M)$, we obtain a well-defined map 
\begin{equation*}
[r]: \cH_{(2)}^*(M_\infty) \to \coh^* (M).
\end{equation*}
\item Consider the natural map $\iota: \dom_{{\rm min}}^*(M) \hookrightarrow \dom_{{\rm max}}^*(M)$ 
given by inclusion. $\iota$ commutes with $d$ and hence yields a well-defined map on $L^2$-cohomology
\begin{equation*}
[\iota]: \coh^*(M,\partial M) \rightarrow \coh^*(M).
\end{equation*}
\item Consider the natural map $q: \mathscr{E}_{{\rm max}}^*(M) \to \mathscr{E}^*(\partial M)$
on smooth subcomplexes, given by restriction to the boundary. $q$ commutes with $d$ and hence yields
a well-defined map on cohomology 
\begin{equation*}
[q]: H^*(\mathscr{E}_{{\rm max}}^*(M), d) \rightarrow H^*(\mathscr{E}^*(\partial M), d).
\end{equation*}
By \eqref{smooth-coh1} and \eqref{smooth-coh2} we obtain the map on $L^2$-cohomology
\begin{equation*}
[q]: \coh^*(M) \rightarrow \coh^*(\partial M).
\end{equation*}
\end{enumerate}
\end{defn}

\begin{prop}\label{prop_2}
	$\ker [q] = \image [\iota].$
\end{prop}

\begin{proof}
By the first statement in Proposition \ref{core-subcomplexes-up-to-bdry}, we can characterize
\begin{align*}
\mathscr{E}_{\min}^*(M) &= \mathscr{E}^-(M_d) \restriction M \cup \partial M 
\\ &=\{\w \in \mathscr{D}_{\max}^*(M) \cap \Omega^*(M \cup \partial M) \mid q\w \equiv \w_0(0)=0\}.
\end{align*}
Thus, the maps $\iota$ and $q$ yield a short exact sequence of smooth subcomplexes 
\begin{equation*}
\xymatrix{0 \ar[r] & \mathscr{E}_{{\rm min}}^* (M) \ar@{^{(}->}[r]^{\iota} & 
\mathscr{E}_{{\rm max}}^*(M) \ar[r]^{q} & \mathscr{E}^*(\partial M) \ar[r] & 0}
\end{equation*}
This yields an exact sequence in cohomology
\begin{equation*}
\xymatrix{...  \ar[r] & \coh^*(\mathscr{E}^*_{{\rm min}}(M), d) \ar[r]^{[\iota]} & \coh^*(\mathscr{E}^*_{{\rm max}}(M), d)
 \ar[r]^{[q]} & \coh^*(\mathscr{E}^*(\partial M), d) \ar[r] & ...}
\end{equation*}
Exactness means in particular $\ker [q] = \image [\iota]$ as maps on cohomologies of the smooth subcomplexes.
The statement follows from the fact that cohomologies
of the Hilbert complexes and their corresponding subcomplexes coincide. see \eqref{smooth-coh1}
and \eqref{smooth-coh2}.
\end{proof}

\begin{prop}\label{prop_3}
	$\image [r] \subseteq \ker [q]$.
\end{prop}

\begin{proof}
We adapt an argument of \cite[Lemma 3.12]{Lueck-Schick}.
	Consider a harmonic form $\w \in \cH_{(2)}^*(M_\infty)$. Notice that $\w \in \Omega^*(M_\infty)$ is smooth by elliptic regularity and 
	hence $q (r (\omega))$ makes sense, since $\omega$ and $r (\omega)$ are smooth. 
	Over the cylinder $(-\infty,0] \times \partial M \subset M_\infty$ we can decompose $\w$ as 
	$$
	\w(x) = \w_0(x) + \w_1(x) \wedge dx, \quad 
	\w_0, \w_1 \in C^\infty((-\infty,0], \Omega^*(\partial M)).
	$$
	Noting $\w \in \ker d$, we obtain for $d_{\partial M}$ denoting the exterior derivative on $\partial M$
	\begin{align*}
	0 = d\w = d_{\partial M}\w_0(x) \pm \frac{\partial \w_0(x)}{\partial x} \wedge dx + d_{\partial M}\w_1(x) \wedge dx. 
	\end{align*}
	From linear independence of summands with and without $dx$, we conclude 
	\begin{align*}
	\frac{\partial \w_0(x)}{\partial x} = \pm  d_{\partial M} \w_1(x). 
	\end{align*}
	Integration with respect to $x$ yields for any $x_0 \in (-\infty,0]$
	\begin{align}\label{q}
	q[x] \w - q[x_0] \w \equiv \w_0(x) - \w_0(x_0) =  \pm  d_{\partial M} \int_{x_0}^{x} \w_1(u) du,
	\end{align}
	where we also write $q[x]\w \equiv \w_0(x)$ for the pullback of $\w$ to $\{x\}\times \partial M \subset M_\infty$.
	Consider now the Hodge Dirac operator $\slashed{\partial} \equiv d + d^*$. 
	Over the cylinder $(-\infty,0] \times \partial M \subset M_\infty$
	the operator $\slashed{\partial}$ takes the product form 
	\begin{equation*}
		\slashed{\partial} = \Gamma(\partial_x + \slashed{\partial}_{\partial M}),
	\end{equation*}
	where $\Gamma$ is a bundle homomorphism on $\partial M$ and $\slashed{\partial}_{\partial M}$
	is essentially self-adjoint in $L^2\Omega^*(\partial M)$ with discrete spectrum, 
	due to the geometric Witt assumption on the Witt space $\partial M$.
	Writing $(\lambda, \phi_\lambda)$ for the eigenvalues and eigensections (counted with their multiplicities)
	of the unique self-adjoint extension of $\slashed{\partial}_{\partial M}$, we conclude from $\slashed{\partial} \w = 0$
	\begin{equation*}
		\w(x)  = 
		\sum_{(\lambda, \phi_\lambda)} [a_\lambda e^{-\lambda x} + b_\lambda e^{\lambda x}] \phi_\lambda (\cdot).
	\end{equation*}
	Since $\w \in L^2\Omega^*(M_\infty)$ we conclude that $a_\lambda = 0$ for $\lambda \le 0$ and $b_\lambda = 0$
	for $\lambda \ge 0$. Hence $\w(x)$ is exponentially vanishing as $x \to -\infty$ and taking $x\to -\infty$ in \eqref{q}, 
	we conclude
	$$
	q[x] \w =  \pm  d_{\partial M} \int_0^{\infty} \w_1(u) du. 
	$$
	Consequently, $q (r \w)) \equiv q[0]\w$ defines the zero-class in cohomology $\coh^*(\partial M)$.
	Thus $[r][\w]\equiv [r\w] \in \ker [q]$ and the statement follows.
\end{proof}

\begin{prop}\label{prop_5}
	$\image [r] = \ker [q]$.
\end{prop}

\begin{proof}
Consider $[\w] \in \ker [q] \subset \coh^* (M) \cong H^*(\mathscr{E}_{{\rm max}}^*(M),d)$ 
with a closed smooth representative $\w \in \mathscr{E}_{{\rm max}}^*(M)$. Note that 
by smoothness, $q\w \in \mathscr{E}^*(\partial M)$ is well-defined and $[q][\w] = [q \w]$. 
Since $[q\w]=0$, we find $q\w = d \A$ for some $\A \in \mathscr{E}^*(\partial M)$. 
Consider a cutoff function $\psi:(-\infty,0]\to \mathbb{R}$,
with $\psi \restriction [0,\delta]\equiv 1$ for some $\delta \in (0,1)$ 
and compact support in $[0,1)$. We can define $\tilde{\A} := \A \cdot \psi \in 
\mathscr{E}^*_{\rm max}(\partial M \times (-\infty,0])$, where $\mathscr{E}^*_{\rm max}(\partial M \times (-\infty,0])$
is defined similar to $\mathscr{E}^*_{\rm max}(M)$. 
We also have the inclusion $j_-: \partial M \equiv \partial M \times \{0\} \hookrightarrow \partial M \times \left[ 0,\infty \right)$.
By construction $j_-^*\,\tilde{\A} = \A$ and $j_-^*\,d\tilde{\A} = d\A$. \medskip
We define $W \in L^2\Omega^*(M_\infty) \equiv L^2\Omega^*(\W^*(M_\infty))$ as follows:
	\begin{equation*}
	W\vert_M := \w; \ W\vert_{\partial M \times \left[ 0, \infty \right)} := d\tilde{\A}
	\end{equation*}
Note that $W$ is smooth, except possibly at $\partial M \times \{0\}$.
Consider the smooth subcomplex $(\mathscr{E}^*_{\max}(\partial M \times \left[ 0, \infty \right)),d)$
and the restriction $q_+: (\mathscr{E}^*_{\max}(\partial M \times \left[ 0, \infty \right)),d)\to (\mathscr{E}^*(\partial M),d)$
defined by the pullback of $j: \partial M \equiv \partial M \times \{0\} 
\hookrightarrow \partial M \times \left[ 0,\infty \right)$. We find by construction,
	\begin{equation*}
		q \left( W\vert_M \right)  = q_+ \left( W \vert_{\partial M \times \left[ 0, \infty \right)} \right).
	\end{equation*}
From the $L^2$-Stokes theorem in Theorem \ref{Stokes-thm} we conclude (note that 
$dW=0$ weakly, and hence $W$ lies in the maximal domain of $d$)
\begin{equation*}
\begin{split}
&\langle \, W,d^t\phi \, \rangle_{L^2\Omega^*(M)}  = - \int_{\partial M} d\A \, \wedge \,q (\ast \phi)\\
&\langle \, W,d^t\phi \, \rangle_{L^2\Omega^*(\partial M \times \left[ 0, \infty \right))}  = - 
\int_{\partial M} d\A \, \wedge \, q_+ (\ast \phi)
\end{split}
\end{equation*}
where $d^t$ is the formal adjoint of $d$ and $\phi \in \W_c^*(M_\infty)$ is a 
smooth form  with compact support. Since $\partial M$ is included into 
$M$ and $\partial M \times (-\infty,0]$ with opposite orientations, we conclude
\begin{equation*}
\langle \, W,d^t\phi \, \rangle_{L^2\Omega^*(M_\infty)}  = 0,
\end{equation*}
which impies that $W \perp \overline{\image d^*}$. 
Due to the weak Kodaira decomposition
$L^2\Omega^*(M_\infty) = \cH_{(2)}^* (M_\infty) \oplus \overline{\image d} \oplus \overline{\image d^*}$
we conclude
\begin{equation*}
\begin{split}
W \perp \overline{\image d^*} & \Rightarrow W 
\in \cH_{(2)}^* (M_\infty) \oplus \overline{\image d}\\
& \Rightarrow W = h + \eta
\end{split}
\end{equation*}
Now, $h \in \cH_{(2)}^* (M_\infty)$ and $[r]h \equiv [rh] \in \coh^*(M)$. 	
Furthermore, $\eta \in \overline{\image d}$ and hence the restriction to $M$ defines 
$r \eta \equiv \eta \vert_M \in \overline{\image d_M}$, where the lower index $M$ 
indicates that $d_M$ refers to the maximal closed extension of $d$ in $L^2\Omega^*(M)$. 
The Hilbert complex $(\dom_{\textup{max}}^*(M),d = d_M)$ is Fredholm by Proposition \ref{core-subcomplexes-up-to-bdry} 
and Remark  \ref{remark: finite-dimensionality}
 and hence $\image d_M$ is closed. Hence $r \eta \in \overline{\image d_M} = \image d_M$ and 
	\begin{equation*}
	 [rh]  = [rh + r\eta] \in \coh^*(M).
	\end{equation*}
On the other hand we compute by construction
	\begin{equation*}
		[rh] = [r (h+ \eta)] = [rW] = [\w].
	\end{equation*}
We conclude: $[\w] = [rh] \in  \coh^*(M)$, i.e. for any $[\w] \in \ker [q]$ 
we found a harmonic representative $h \in \cH_{(2)}^*(M_\infty)$ so that
$[r]h = [\w]$, establishing the statement $\ker [q] \subseteq \image [r]$.
Equality follows from Proposition \ref{prop_3}.
\end{proof}

Proposition \ref{prop_2} and Proposition \ref{prop_5} imply
\begin{align}\label{rqi}
\image [r] =  \ker [q] = \image [\iota].
\end{align}

\subsection{Various signatures}
Let $M$ now have dimension $4n$. In this final subsection we are now in the position 
to introduce different signatures on $M$ and compare them. In this section we restrict the 
actions of the maps $[\iota]$, $[r]$ and $[q]$ to differential forms and cohomology classes of 
degree $2n$ and hence write e.g.
\begin{equation}
\begin{split}
&\image [r] \equiv \image ([r]: \cH^{2n}_{(2)} (M_\infty)\to \coh^{2n} (M)), \\
&\image [\iota] \equiv \image ([\iota]: \coh^{2n} (M,\partial M)\to \coh^{2n} (M)), \\
&\ker [\iota] \equiv \ker ([\iota]: \coh^{2n} (M,\partial M)\to \coh^{2n} (M)).
\end{split}
\end{equation}
We recall at this point that 
$$ \cH^{*}_{(2)} (M_\infty)\,,\quad  \coh^{*} (M,\partial M)\,,\quad 
\coh^{*} (M)$$
have been shown  to be finite dimensional.

\medskip

We begin with studying certain bilinear forms, whose signatures give equivalent definitions
for the signature of $(M, \partial M)$.
\begin{lemma}\label{deRham-signature}
The de Rham signature pairing 
\begin{equation}
\begin{split}
s :  \coh^{2n} (M,\partial M) \times    \coh^{2n} (M,\partial M) \to \mathbb{C}, 
\quad  ([v], [w]) \mapsto \int_M v \wedge w,
\end{split}
\end{equation}
is a well-defined degenerate bilinear symmetric form with radical 
given by $\ker [\iota]$.
\end{lemma}

\begin{proof}
We first show that $s$ is well-defined. Consider $[v], [w] \in \coh^{2n} (M,\partial M)$
with representatives $v+ d\alpha, w+ d\beta \in \mathscr{E}^*_{\rm min}(M)$, respectively;
$v,w$ as well as $\alpha, \beta$ are elements of $\mathscr{E}^*_{\rm min}(M)$.
Note that we can equivalently work with smooth subcomplexes, since their cohomologies coincide with 
the cohomologies of the corresponding Hilbert complexes. 
Observe that 
$$
d(v \wedge \beta) = (-1)^{2n} v \wedge d\beta, \quad 
d(\alpha \wedge w) = d\alpha \wedge w, \quad
d(\alpha \wedge d\beta) = d\alpha \wedge d\beta.
$$
We compute using the $L^2$-Stokes theorem in Theorem \ref{Stokes-thm} for the Witt space $M$
\begin{align*}
&\int_M (v+ d\alpha) \wedge (w+ d\beta) \\
&= \int_M v \wedge w + (-1)^{2n} \int_M d(v\wedge \beta) + \int_M d(\alpha \wedge w) + \int_M d(\alpha \wedge d\beta)
\\ &= \int_M v \wedge w + (-1)^{2n} \int_{\partial M} q (v\wedge \beta) + \int_{\partial M} q(\alpha \wedge w)
+\int_{\partial M} q(\alpha \wedge d\beta).
\end{align*}
Note that $qv, q\alpha = 0$ since $v,\alpha \in \mathscr{E}^*_{\rm min}(M)$. 
Hence we conclude
\begin{align*}
\int_M (v+ d\alpha) \wedge (w+ d\beta) = \int_M v \wedge w.
\end{align*}
Thus $s$ is indeed well-defined. We now identify the radical of $s$ 
\begin{equation*}
\textup{Rad}(s) := \{ [v] \in  \coh^{2n} (M,\partial M) \mid \forall [w] \in  \coh^{2n} (M,\partial M): 
s([v],[w]) = 0\}.
\end{equation*}
Denote the Hodge star operator of $(M,g)$ by $*$. 
We note that the adjoint $d_{\max}^*$ of the maximal closed extension $d_{\max}$
is given by the minimal closed extension of the formal adjoint $d^t = \pm * d*$, cf. \cite[p. 105]{BrLe}.
Thus the Hodge star $*$, preserving smoothness and support, yields an isomorphism between the minimal domain of
$d$ and the minimal domain of $d^t$
\begin{align}\label{star-min-max}
*: \ker d_{\min} \to \ker \bigl( d^t\bigr)_{\min} \equiv \ker d_{\max}^*.
\end{align}
Consider now any $[v] \in \textup{Rad}(s)$. By definition, $v \perp * (\ker d_{\min})$ with respect
to the $L^2$-inner product. Thus $v \perp \ker d_{\max}^*$ by \eqref{star-min-max}. Consequently,
$v \in \image d_{\max}$. This proves the second statement of the lemma
\begin{equation*}
\textup{Rad}(s) = \ker [\iota].
\end{equation*}
\end{proof}

\begin{lemma}\label{Hodge-signature}
The Hodge $L^2$-signature pairing is defined by 
\begin{equation}
\begin{split}
s_{\infty}:  \cH_{(2)}^{2n}(M_\infty) \times  \cH_{(2)}^{2n}(M_\infty) \to \mathbb{C},
\quad (\w, \eta) \mapsto \int_M \w \wedge \, \eta.
\end{split}
\end{equation}
It is a non-degenerate symmetric bilinear form.
\end{lemma}

\begin{proof}
Consider $h \in \cH_{(2)}^{2n}(M_\infty)$ in the radical of $s_{\infty}$. 
Then $*h \in \cH_{(2)}^{2n}(M_\infty)$ and 
$$
0=s_{\infty}(h,*h)= \|h\|_{L^2}^2.
$$ 
Hence $h=0$ and thus the radical of $s_{\infty}$ is trivial.
\end{proof}

\begin{defn}
\begin{enumerate}
\item  The de Rham signature $\textup{sign}(s)$ is defined as the signature of the non-degenerate bilinear form
\begin{equation}
s : \image [\iota] \times   \image [\iota] \to \mathbb{C},
\end{equation}
induced by the de Rham signature pairing $s$ on 
$$
\frac{\coh^{2n} (M,\partial M) }{\textup{Rad}(s)} \equiv \frac{\coh^{2n} (M,\partial M) }{\ker [\iota]} 
\cong \image [\iota].$$
\item The Hodge $L^2$-signature $\textup{sign}(s_\infty)$ is defined as the signature
of the non-degenerate bilinear form $s_\infty$. 
\end{enumerate}
\end{defn}

Next we prove that the two signatures $\textup{sign}(s)$ and $\textup{sign}(s_\infty)$ coincide.

\begin{prop}\label{prop_6}
The Hodge $L^2$-signature $s_\infty$ descends to the de Rham signature $s$ in the following sense. 
Given $\w, \eta \in \cH^{2n}_{(2)}(M_\infty)$ and $[r\w], [r\, \eta] \in \image [r] = \image [\iota]$
we have the following equality
$$
s_\infty(\w, \eta) = s([r\w], [r\, \eta]).
$$
\end{prop}

\begin{proof}
Consider $\w, \eta \in \cH_{(2)}^{2n}(M_\infty)$. 
As shown in Proposition \ref{prop_3}, $q\w \equiv q (r\w)$ and 
$q\eta \equiv q(r\, \eta)$ define zero-classes in cohomology $\coh^*(\partial M)$ and hence 
there exist some $\alpha, \beta \in \dom^*(\partial M)$ such that 
	\begin{equation*}
		q\w = d\A, \quad q \eta = d \beta.
	\end{equation*}
Consider a cutoff function $\psi \in C^\infty_0[0,1)$ which is identically $1$ near $0$.
Then $\psi \cdot \alpha$ and $\psi \cdot \beta$ extend smoothly to the interior of $M$ and define 
	\begin{equation*}
	 \tilde{\A}' := \psi \cdot \alpha \in \dom^*_{\rm max}(M), \quad \tilde{\beta}' := \psi \cdot \beta \in \dom^*_{\rm max}(M).
	\end{equation*}
Set $v:= r \w -d\tilde{\A}', \ w:=r\, \eta-d\tilde{\beta}'$. By construction 
	\begin{equation*}
		\begin{split}
		[v] = [r\w] \in \coh^{2n}(M), \quad  & qv = 0, \\
		[w] = [r\, \eta] \in \coh^{2n}(M), \quad  & qw = 0.
		\end{split}
	\end{equation*}
The claim of the proposition is now
\begin{equation}\label{star}
\int_{M} v \land w = \int_{M_\infty} \omega \land \eta.
\end{equation}
To prove \eqref{star} we begin with its left hand side and compute
using $L^2$-Stokes theorem in Theorem \ref{Stokes-thm} on the Witt space $M$
\begin{equation*}
\begin{split}
\int_{M} v \land w &= \int_{M} \w \wedge \, \eta - \int_{M} w \land d \tilde{\beta}' - \int_{M} d \tilde{\A}' \land \eta
\\ &= \int_{M} \w \land \, \eta - \int_{\partial M} \A \land d \beta.
\end{split}
\end{equation*}	
Hence \eqref{star} is equivalent to showing (write $\R^+:= (-\infty,0]$)
	\begin{equation}\label{6.1}
		\int_{\partial M \times \R^+} \w \land \, \eta = - \int_{\partial M} \A \land d \beta.
	\end{equation}
Consider the weak Kodaira decomposition
	\begin{equation*}
		\begin{split}
		L^2\Omega^*(\partial M \times \R^+) = \cH_{(2)}^*(\partial M \times \R^+, \partial M)
		\oplus \overline{\image d_{\textup{min}}}  \oplus \overline{\image d_{\textup{min}}^*}.
		\end{split}
	\end{equation*}
Extend $\A,\beta$ to $\tilde{\A},\tilde{\beta} \in \dom_{\max}^*(\partial M \times \R^+)$ 
exactly as we did in Proposition \ref{prop_5}, by multiplying with a smooth cutoff function that is identically $1$
in an open neighborhood of $\partial M \times \{0\}$. 	By construction $(\w - d\tilde{\A}),(\eta-d\tilde{\beta})$ 
both pull back to zero at $\partial M \times {0}$ and hence $(\w-d\tilde{\A}),(\eta -d\tilde{\beta}) \in 
\dom_{\min}(\partial M \times \R^+)$. Moreover, $d(\w - d\tilde{\A}) = d(\eta-d\tilde{\beta}) = 0$ and hence 
\begin{equation*}
	\begin{split}
(\w-d\tilde{\A}),(\eta-d\tilde{\beta}) &\in \ker d_{\min} = 
(\image d_{\min}^*)^\perp \\ &= \cH_{(2)}^* (\partial M \times \R^+, \partial M) \oplus \, \overline{\image d_{\min}}.
\end{split}
\end{equation*}
With respect to this decomposition, we write
	\begin{equation*}
		\begin{split}
		\w -d\tilde{\A} = h_1 + x, & \text{ with } qh_1 = 0, \\
		\eta - d \tilde{\beta} = h_2 + y, & \text{ with } qh_2 = 0.
		\end{split}
	\end{equation*}
Since $x,y \in \overline{\image d_{\min}}$, there exist $(t_n),(s_n) \subset 
\dom_{\min}^* (\partial M \times \R^+)$ such that $dt_n \to x$ and $ds_n \to y$
in $L^2$. Hence we compute using in the second equality the $L^2$-Stokes theorem in Theorem \ref{Stokes-thm} 
on the cylinder over the Witt space $\partial M$
\begin{equation*}
\begin{split}
\int_{\partial M \times \R^+} (h_1 + x) \land (h_2 +y) & 
= \lim\limits_{n \to \infty} \int_{\partial M \times \R^+} (h_1 + d t_n) \land (h_2 + ds_n)\\
& = \int_{\partial M \times \R^+} h_1 \land h_2.
\end{split}
\end{equation*}
Consider for each $i=1,2$ the decomposition $h_i = a_i+b_i \wedge dx$ with
contractions $\partial_x \, \lrcorner \, a_i,  \partial_x \, \lrcorner \, b_i = 0$.
Now, by the product structure of the Hodge-Laplacian on $\partial M \times \R^+$,
the forms $a_i,b_i$ are harmonic individually. Since $qh_i=a_i(0)=0$, we conclude that
$a_i\equiv 0$ identically. Hence
$$
h_1 \land h_2 = \left(b_1 dx\right) \wedge \left(b_2 dx\right) = 0.
$$
This proves \eqref{6.1} as follows
\begin{equation*}
0 = \int_{\partial M \times \R^+} (\w-d\tilde{\A}) \land (\eta -d\tilde{\beta}) 
= \int_{\partial M \times \R^+} \w \land \, \eta + \int_{\partial M} \A \land d\beta.
\end{equation*}
\end{proof}

This fact also proves the following.
\begin{cor}\label{image-r}
	$[r]: \cH_{(2)}^{2n}(M_\infty) \to \image [\iota]$ is bijective
\end{cor}

\begin{proof}
	Let $h\in \cH_{(2)}^*(M_\infty)$ with $[r]h\equiv [rh] = 0$. Then by Proposition \ref{prop_6}
	\begin{equation*}
		0 = s(0,0) = s([rh],[r\ast h]) = s_\infty(h,\ast h) = \Vert h \Vert_{L^2} \Rightarrow h = 0.
	\end{equation*}
This proves that $[r]$ is injective. It is surjective since by \eqref{rqi}
	\begin{equation*} 
	\image [r] =  \image [\iota].
	\end{equation*}
\end{proof}
We conclude that the de Rham and the Hodge $L^2$-signatures coincide
\begin{equation}\label{s-s-infty}
	\left( \textup{sign}_{\textup{dR}}(\overline{M},\partial \overline{M}):= \right) \ \textup{sign}(s) = \textup{sign}(s_\infty)
		\left( =: \textup{sign}_{\textup{Ho}}(\overline{M}_\infty)\right),
\end{equation}
which is the first statement in our main Theorem \ref{main-1}. As already remarked
in the introduction in \eqref{equality-total}, we also have the equality of the topological signature 
$\textup{sign}_{\textup{top}}(\overline{M},\partial \overline{M})$ of Friedman and Hunsicker \cite{FH} with the
de Rham signature considered here, see \cite[(95)]{Bei2}. Summarizing, we find
\begin{equation}
	\textup{sign}_{\textup{top}}(\overline{M},\partial \overline{M})= 
	 \textup{sign}_{\textup{dR}}(\overline{M},\partial \overline{M}) = \textup{sign}_{\textup{Ho}}(\overline{M}_\infty).
\end{equation}
\section{Signatures of coverings of stratified spaces with boundary}\label{signature-2-section}

We extend the analysis of \S \ref{signature-section} to the setting of non-compact Galois coverings. 
Consider the Galois coverings $\overline{M}_\Gamma, \overline{M}_{\Gamma,\infty},\partial \overline{M}_\Gamma$ of 
$\overline{M}, \overline{M}_{\infty},\partial \overline{M}$, respectively, all with Galois group $\Gamma$. We denote
the regular part of $\overline{M}_\Gamma$ by $\widetilde{M}$, the regular part of $\overline{M}_{\Gamma,\infty}$
by $\widetilde{M}_\infty$, and the regular part of $\partial \overline{M}_\Gamma$ by $\partial \widetilde{M}$.
The Witt assumption still holds on the coverings. \medskip

In this section we shall build heavily on the analysis of L\"uck-Schick \cite{Lueck-Schick}.

\subsection{Hilbert complexes on Galois coverings}

\subsubsection{Hilbert complexes on a covering of a closed compact Witt space $\overline{X}$}\label{X-Gamma-subsection}
Let $\overline{X}$  be a (closed) Witt space, e.g.  $\overline{X}= \partial \overline{M}$.  
Let  $\overline{X}_\Gamma$ be a Galois $\Gamma$-cover of $\overline{X}$. We denote by $X$ the regular part  
of $\overline{X}$ and by $\widetilde{X}$ the regular part of $\overline{X}_\Gamma$.
By the Witt assumption, the de Rham complex $(\Omega_c^*(\widetilde{X}), d)$ of compactly
supported smooth differential forms over $\widetilde{X}$, admits a unique closed extension, 
with the associated Hilbert complex denoted by $(\dom_{(2)}^*(\widetilde{X}),d)$.
Its smooth subcomplex is denoted by $(\mathscr{E}^*(\widetilde{X}),d)$, where as before we set $\mathscr{E}^*(\widetilde{X}) = 
\dom_{(2)}^*(\widetilde{X}) \cap \Omega^*(\widetilde{X})$. Here, $\Omega^*(\widetilde{X})$ denotes smooth differential 
forms on $\widetilde{X}$, unrestricted otherwise.
The weak Kodaira decomposition still holds, and hence
\begin{equation*}
L^2\Omega^*(\widetilde{X}) = \cH_{(2)}^*(\widetilde{X}) \oplus \overline{\image d} \oplus \overline{\image d^*}.
\end{equation*}
The harmonic forms $\cH_{(2)}^*(\widetilde{X})$ can be identified with reduced cohomology
$$
\cH_{(2)}^*(\widetilde{X}) := \ker d \cap \ker d^* 
\cong \frac{\ker d}{\overline{\image d}} =: \cohr^*(\widetilde{X}),
$$ 
where in contrast to the notation $\coh^*(\widetilde{X})$ 
for the $L^2$-cohomology of the Hilbert complex $(\dom_{(2)}^*(\widetilde{X}),d)$, 
$\cohr^*(\widetilde{X})$ refers to the reduced cohomology.
Since the Hilbert complex $(\dom_{(2)}^*(\widetilde{X}),d)$ is not Fredholm,
the reduced cohomology does not equal the $L^2$-cohomology of the complex:
$$
\coh^*(\widetilde{X}) \neq \cohr^*(\widetilde{X})
$$
As explained for example in \cite[Theorem 1]{Dodziuk} the isomorphism $\cH_{(2)}^*(\widetilde{X}) 
\cong \cohr^*(\widetilde{X})$ is in fact an isomorphism of Hilbert $\mathcal{N}\Gamma$-modules.
We have proved in \cite[Prop. 9.1]{PiVe} that the orthogonal projection onto  $\cH_{(2)}^*(\widetilde{X}) $
is of $\Gamma$-trace class. Consequently  $\cH_{(2)}^*(\widetilde{X}) $, and thus $\cohr^*(\widetilde{X})$,
is of finite $\Gamma$-dimension.

\subsubsection{Hilbert complexes on $\widetilde{M}$}
The de Rham complex $(\W^*_c(\widetilde{M}),d)$ of smooth compactly supported differential forms
over $\widetilde{M}$ does not admit a unique closed extension, and as before we single out the minimal 
and maximal closed extensions
$$
(\dom_{\min}^*(\widetilde{M}),d), \quad (\dom_{\max}^*(\widetilde{M}),d).
$$
We denote their smooth subcomplexes (defined as in \eqref{core-up-to-bdry}) by 
$(\mathscr{E}_{\min}^*(\widetilde{M}),d)$ and \linebreak $(\mathscr{E}_{\max}^*(\widetilde{M}),d)$.
Let $\cH_{(2)}^*(\widetilde{M},\partial \widetilde{M})$ and $\cohr^*(\widetilde{M},\partial \widetilde{M})$
denote the harmonic forms and reduced cohomology of $(\dom_{\min}^*(\widetilde{M}),d)$, respectively.
Similarly, $\cH_{(2)}^*(\widetilde{M})$ and $\cohr^*(\widetilde{M})$ shall
denote the harmonic forms and reduced cohomology of $(\dom_{\max}^*(\widetilde{M}),d)$, respectively.
Again, by the weak Kodaira decomposition
\begin{equation*}
		\begin{split}
		 \cH_{(2)}^*(\widetilde{M},\partial \widetilde{M}) \cong \cohr^*(\widetilde{M},\partial \widetilde{M}), \quad
		 \cH_{(2)}^*(\widetilde{M}) \cong \cohr^*(\widetilde{M}),
		\end{split}
	\end{equation*}
As before, the complexes $(\dom_{\min}^*(\widetilde{M}),d)$ and $(\dom_{\max}^*(\widetilde{M}),d)$
are not Fredholm and hence the reduced and $L^2$-cohomologies differ
$$
\coh^*(\widetilde{M},\partial \widetilde{M}) \neq \cohr^*(\widetilde{M},\partial \widetilde{M}),
 \quad \coh^*(\widetilde{M}) \neq \cohr^*(\widetilde{M}).
$$
In this article, we shall be exclusively interested in reduced $L^2$-cohomology.

\begin{remark}\label{finite-gamma-dimension} 
We can follow the arguments of \cite[Theorem 4.1]{BrLe} as in Proposition
\ref{core-subcomplexes-up-to-bdry} to identify $(\dom_{\min}^*(\widetilde{M}),d)$ and $(\dom_{\max}^*(\widetilde{M}),d)$
with subcomplexes of $(\dom^*(\widetilde{M}_d),d)$ on the double $\widetilde{M}_d$. 
The reduced $L^2$-cohomology of the latter is of finite $\Gamma$-dimension, as already observed above in 
\S \ref{X-Gamma-subsection}. Thus the reduced  $L^2$-cohomologies $\cohr^*(\widetilde{M},\partial \widetilde{M})$
and $\cohr^*(\widetilde{M})$ are of finite $\Gamma$-dimension as well.
\end{remark}

\subsubsection{Hilbert complexes on $\widetilde{M}_\infty$}\label{subsub:cylindrical}
Finally, the de Rham complex $(\W_0^*(\widetilde{M}_\infty),d)$ admits a unique closed extension
$(\dom_{(2)}^*(\widetilde{M}_\infty),d)$.
The corresponding harmonic forms $\cH_{(2)}^*(\widetilde{M}_\infty)$ and 
reduced cohomology $\cohr^*(\widetilde{M}_\infty)$ satisfy again, by the weak Kodaira decomposition
\begin{equation*}
\cH_{(2)}^*(\widetilde{M}_\infty) \cong \cohr^*(\widetilde{M}_\infty) \neq \coh^*(\widetilde{M}_\infty).
\end{equation*}
By the results established  in Subsection \ref{subsection:gamma-index}, see in particular \eqref{finite-dim2}, we know that 
$\cH_{(2)}^*(\widetilde{M}_\infty)$ has finite $\Gamma$-dimension.

\subsection{Various maps between Hilbert complexes on Galois coverings}

In this section we define certain homomorphisms between the reduced cohomology groups defined above.
These homorphisms correspond to the maps introduced in Definition \ref{define-maps}, with slight changes in 
the definition of  "restriction to the boundary" $[q]$ in the setting of reduced cohomologies. 

\begin{defn} Obvious inclusions and restrictions define the following maps.
\begin{enumerate}
\item Consider the natural map $r: \dom_{(2)}^*(\widetilde{M}_\infty) \to \dom_{{\rm max}}^*(\widetilde{M})$ given by
restriction to $\widetilde{M}\subset \widetilde{M}_\infty$. The map $r$ commutes with $d$ and hence for any 
$\w \in \cH_{(2)}^*(\widetilde{M}_\infty)$,
$r \w \in \ker d \subset \dom_{{\rm max}}^*(\widetilde{M})$. Taking the corresponding reduced cohomology class 
$[r \w] \in \coh^* (\widetilde{M})$, we obtain a well-defined map 
\begin{equation*}
[r]: \cH_{(2)}^*(\widetilde{M}_\infty) \to \cohr^* (\widetilde{M}).
\end{equation*}
\item Consider the natural map $\iota: \dom_{{\rm min}}^*(\widetilde{M}) \hookrightarrow \dom_{{\rm max}}^*(\widetilde{M})$ 
given by inclusion. $\iota$ commutes with $d$ and is continuous in $L^2$. Due to continuity, $\iota$ not only induces a well-defined
map on (non-reduced) $L^2$-cohomology, but in fact on reduced cohomology as well
\begin{equation*}
[\iota]: \cohr^*(\widetilde{M},\partial \widetilde{M}) \rightarrow \cohr^*(\widetilde{M}).
\end{equation*}
\item Consider the natural map $q: \mathscr{E}_{{\rm max}}^*(\widetilde{M}) \to \mathscr{E}^*(\partial \widetilde{M})$
on smooth subcomplexes, given by restriction to the boundary. $q$ commutes with $d$ and hence its action on
harmonic forms $\cH_{(2)}^*(\widetilde{M})$ maps into $\ker d$ and thus yields
a well-defined map into reduced cohomology 
\begin{equation*}
[q]: \cH_{(2)}^*(\widetilde{M}) \rightarrow \cohr^*(\partial \widetilde{M}).
\end{equation*}
Using $\cH_{(2)}^*(\widetilde{M}) \cong \cohr^*(\widetilde{M})$, we obtain the map on reduced cohomology
\begin{equation*}
[q]: \cohr^*(\widetilde{M}) \rightarrow \cohr^*(\partial \widetilde{M}).
\end{equation*}
\end{enumerate}
\end{defn} 

\begin{remark}
Note that the map $[q]$ is not induced from a map on Hilbert complexes
$q: \dom^*_{\max}(\widetilde{M}) \to \dom^*_{(2)}(\partial \widetilde{M})$, 
since such $q$ a priori does not exist. One may try to salvage the 
situation by considering the restriction map $q[x]$ given by the pullback by 
$\partial \widetilde{M} \times \{x\} \hookrightarrow \widetilde{M}$. By the Fubini theorem, $q[x]\w$ 
makes sense for $\w \in \dom^*_{\max}(\widetilde{M})$, for $x\in (0,1)$ outside of a set of measure zero. 
Even though that set depends on the particular $\w$, the difference $q[x]\w-q[x']\w$ can be shown to be
an exact form, regardless of $x,x'\in (0,1)$, cf. \eqref{q}. This would be enough to define a map on cohomology,
however not on reduced cohomology, since $q[x]$ is not continuous in $L^2$. 
Alternatively, one might try to use Hilbert complexes in Sobolev spaces as in Schick \cite{Schick}.
Then the sequence of maps in reduced cohomology
	\begin{equation}\label{LES}
	\xymatrix{\cohr^*(\widetilde{M},\partial \widetilde{M}) 
	\ar[r]^-{[\iota]} & \cohr^*(\widetilde{M}) \ar[r]^-{[q]} & \cohr^*(\partial \widetilde{M}).}
	\end{equation}
would become part of the usual long weakly exact sequence in reduced 
cohomology due to Cheeger-Gromov \cite[Theorem 2.1]{ChGr}, coming from a short exact sequence of 
Hilbert complexes. In the present discussion we avoid the definition of Hilbert complexes in Sobolev 
spaces of higher order, and hence establish weak exactness of \eqref{LES} directly. This is done
below, starting with Proposition \ref{prop_3'} until the end of the subsection.
\end{remark}

\begin{prop}\label{prop_3'}
	$\image [\iota] \subseteq \image [r] \subseteq \ker [q].$
\end{prop}

\begin{proof}
In order to prove the first inclusion $\image [\iota] \subseteq \image [r]$, consider
$[\w]\in \image [\iota]$. By definition, any representative $\w \in [\w]$ lies in $\ker d_{\min}$, where $d_{\min}$
is the closed extension of $d$ with domain $\dom^*_{\min}(\widetilde{M})$. We define $W\in L^2\Omega^*(\widetilde{M}_\infty)$ by
$$
W \restriction \widetilde{M} := \w, \quad W \restriction \partial \widetilde{M}\times (-\infty,0] := 0.
$$
Then, $dW=0$ in the weak sense. We compute for any $\phi \in \Omega^*_c(\widetilde{M}_\infty)$
$$
\langle W, d^t \phi \rangle_{L^2\Omega^*(\widetilde{M})} =0, \quad \langle W, d^t \phi \rangle_{L^2\Omega^*(\partial \widetilde{M}\times (-\infty,0])} =0.
$$ 
Consequently, $W \in \overline{\image d^*}^\perp = \ker d$. By the weak Kodaira decomposition 
$$
W = h + \eta \in \cH^*_{(2)}(\widetilde{M}_\infty) \oplus \overline{\image d} = \ker d.
$$
Consider the restriction $[r]h = [rh] \in \cohr^*(\widetilde{M})$. By continuity of $r$ in $L^2$, 
the restriction $r\eta$ lies again in $\overline{\image d}$, where $d$ now refers to the closed 
extension of the exterior derivative with domain $\dom^*_{\max}(\widetilde{M})$.
Consequently
$$
[rh] = [rh + r\eta] = [r W] = [\w]. 
$$
This proves that $[\w] \in \image [r]$ and the first inclusion follows. 
For the second inclusion $\image [r] \subseteq \ker [q]$ we proceed almost exactly
as in Proposition \ref{prop_3}. Consider $\w \in \cH^*_{(2)}(\widetilde{M}_\infty)$ 
and the restrictions $q[x], x\geq 0$, defined as pullbacks by the inclusions 
$\partial \widetilde{M}\times \{x\} \hookrightarrow \widetilde{M}_\infty$. Exactly as in 
Proposition \ref{prop_3} we find for any $x,x_0 \leq 0$
$$
q[x] \w - q[x_0]\w \in \image d.
$$
Due to the fact that $\w \in L^2\Omega^*(\widetilde{M}_\infty)$, we conclude that 
$q[x]\w$ vanishes as $x\to -\infty$. This can also be seen using the Browder-Garding decomposition:
consider the Hodge Dirac operator $\widetilde{\slashed{\partial}} \equiv d + d^t$
over the cylinder $\partial \widetilde{M}  \times (-\infty,0] \subset \widetilde{M}_\infty$, where it takes 
the product form 
	\begin{equation*}
		\widetilde{\slashed{\partial}} = \widetilde{\Gamma}(\partial_x + \widetilde{\slashed{\partial}}_{\partial M}),
	\end{equation*}
with $\widetilde{\Gamma}$ being a bundle homomorphism on $\partial \widetilde{M}$ and $\widetilde{\slashed{\partial}}_{\partial M}$
an essentially self-adjoint operator in $L^2\Omega^*(\partial \widetilde{M})$. By the Browder-Garding decomposition on 
$\partial \widetilde{M}$ as stated in Theorem \ref{Browder-Garding}, there exist countably many sections
$e_j: \R \to \dom_{(2)}^*(\partial \widetilde{M})$ such that
$$
\widetilde{\slashed{\partial}}_{\partial M} e_j(\lambda) = \lambda e_j(\lambda),
\quad (V \widetilde{\slashed{\partial}}_{\partial M} \w)_j(\lambda) = \lambda (V\w)_j(\lambda).
$$ 
Using $\widetilde{\slashed{\partial}} \w = 0$, we conclude exactly as in Proposition \ref{prop_3}
that each $(V\w)_j(\lambda)(x)$ is vanishing exponentially as $x\to -\infty$. Consequently, same holds for $\w$
and we find 
$$
q[0]\w \in \image d \subset \overline{\image d}.
$$
Thus $[q] ( r \w ) \equiv [q[0]\w] = 0$. This proves the second inclusion.
\end{proof}

We would like to prove a result corresponding to  Proposition 
\ref{prop_5}. However, for any $[\w] \in \ker [q]$ with harmonic representative $\w \in \cH_{(2)}^*(\widetilde{M})$ we find
$q\w \in \overline{\image d}$, since we use reduced cohomologies. 
Hence $q \w$ is not exact and we cannot continue as on $M$. Instead, we consider the characteristic function 
$\chi_{(0,\lambda]}$ and the Hodge Laplacian $\Delta = d^*d + dd^*$ on $\partial \widetilde{M}$. As noted already in \cite[Proposition 7.3]{PiVe}, the spectral projection $\chi_{(0,\lambda]}(\Delta)$ is $\Gamma$-trace-class 
\begin{equation*}
	\trace_\Gamma \, \chi_{(0,\lambda]}(\Delta) < \infty.
\end{equation*}

In the following we rely on the basic result of the von Neumann theory 
concerning the properties of the $\Gamma$-traces and $\Gamma$-dimensions. Note that the $\Gamma$-dimension 
of a $\mathcal{N}\Gamma$-Hilbert module
$V\subset L^2\Omega^*(\partial \widetilde{M})$ is defined as the $\Gamma$-trace of the
orthogonal projection $\Pi_V$ onto $V$
\begin{equation}
\dim_{\Gamma} V := \trace_\Gamma (\Pi_V).
\end{equation}

\begin{thm}[Theorem 1.12 (4) in \cite{Lueck-book}]\label{Lueck-von-Neumann}
Let $V\subset L^2\Omega^*(\partial \widetilde{M})$ be an $\mathcal{N}\Gamma$-Hilbert module with $\dim_{\Gamma} V < \infty$.
Let $\{V_i \mid i\in I\}$ be a directed system of $\mathcal{N}{\Gamma}$-Hilbert submodules of $V$, 
directed by $\supset$. Then the following holds
$$
\dim_{\Gamma} \bigcap_{i\in I} V_i = \inf_{i\in I} \ \{\dim_{\Gamma} V_i\}.
$$
\end{thm}

\begin{cor}\label{codim}
The image of the spectral projection $\chi_{(\lambda,\infty)}(d^*d)$ in $L^2\Omega^*(\partial \widetilde{M})$ 
has finite $\Gamma$-codimension $\trace_\Gamma
 \chi_{\left( 0,\lambda \right] }(d^*d)$. This codimension is monotonous, right-continuous 
 in $\lambda >0$, and tends to zero as $\lambda \downarrow 0$. 
\end{cor}

\begin{proof}
The spectral projections $\chi_{(0,\lambda]}(d^*d)$ and $\chi_{(0,\lambda]}(dd^*)$
of $d^*d$ and $dd^*$, respectively, are well-defined, since both operators are self-adjoint. 
Note 
\begin{equation}\label{trace-gamma}
	\trace_\Gamma \, \chi_{(0,\lambda]} (\Delta) 
	= \trace_\Gamma \, \chi_{(0,\lambda]} (d^*d) 
	+ \trace_\Gamma \, \chi_{(0,\lambda]} (dd^*) < \infty.
\end{equation}
Hence the $\Gamma$-traces are finite individually. The fact that $\trace_\Gamma \chi_{\left( 0,\lambda \right] }(d^*d)$ is monotonously
decreasing with $\lambda \to 0$, is a basic property of $\Gamma$-traces, see e.g. 
\cite[Theorem 1.9 (1)]{Lueck-book}. Let us now show right-continuity in $\lambda > 0$.
We apply Theorem \ref{Lueck-von-Neumann}, which yields the following relation 
\begin{align*}
\trace_\Gamma \,
 \chi_{\left( 0,\lambda \right] }(d^*d) &\equiv \dim_{\Gamma}
 \left(  \textup{im} \, \chi_{\left( 0,\lambda \right] }(d^*d)\right) \\ &= 
 \inf \, \left\{ \dim_{\Gamma}
 \left(\textup{im} \, \chi_{\left( 0,\lambda' \right] }(d^*d)\right) \mid \lambda' \geq \lambda\right\}
 \\ &= \inf \, \left\{ \trace_\Gamma \,
 \chi_{\left( 0,\lambda' \right] }(d^*d) \mid \lambda' \geq \lambda\right\} = \lim_{\lambda' \downarrow \, \lambda}
 \trace_\Gamma \, \chi_{\left( 0,\lambda' \right] }(d^*d),
 \end{align*}
where in the last equality we used the fact that infimum and limit of a monotonously decreasing sequence agree.
This proves right continuity in $\lambda > 0$. 
It remains to prove that $\trace_\Gamma \, \chi_{\left( 0,\lambda \right] }(d^*d)$
converges to zero as $\lambda \downarrow 0$. This follows again by Theorem \ref{Lueck-von-Neumann}. Indeed
\begin{equation}\label{computation-gamma}\begin{split}
0 &= \dim_{\Gamma} \Bigl( \{ 0 \} \Bigr) 
= \dim_{\Gamma} \left(  \, \bigcap_{\lambda> 0} \textup{im} \, \chi_{(0,\lambda]}(d^*d)\right) \\ &= 
 \inf \, \left\{ \dim_{\Gamma}
 \left(\textup{im} \, \chi_{\left( 0,\lambda \right] }(d^*d)\right) \mid \lambda > 0\right\}
= \lim_{\lambda \downarrow 0} \, \trace_\Gamma \,
 \chi_{\left( 0,\lambda \right] }(d^*d),
 \end{split}\end{equation}
where in the first equality we used \cite[Theorem 1.12 (1)]{Lueck-book}. In the second equality we used
that the Borel function $\chi_{(0,\lambda]}$ converges pointwise to zero as $\lambda \downarrow 0$, and
hence by the Borel spectral calculus, $\chi_{(0,\lambda]}(d^*d)$ converges strongly to zero. As a consequence, 
for any $\omega \in L^2\Omega^*(\partial \widetilde{M})$, the image $\chi_{(0,\lambda]}(d^*d) \omega$ converges
to $0$ as $\lambda \downarrow 0$ and thus
$$
\bigcap_{\lambda> 0} \textup{im} \, \chi_{(0,\lambda]}(d^*d) = \{0\}.
$$
In the third equality in \eqref{computation-gamma}
we used Theorem \ref{Lueck-von-Neumann}. This finishes the proof. 
\end{proof}

\begin{defn} We introduce for any $\lambda > 0$ the following subspaces
\begin{enumerate}
\item $E_\lambda := \image(d \circ \chi_{(\lambda,\infty)} 
(d^*d)) \subset L^2\Omega^*(\partial \widetilde{M})$.
\item For the restriction $q: \cH^*_{(2)}(\widetilde{M}) \to \ker d \subset \dom^*_{(2)}(\partial \widetilde{M})$ we define
$$
K_\lambda := q^{-1}(E_\lambda) \subset  \cH^*_{(2)}(\widetilde{M}) \cong  \cohr^*(\widetilde{M}).
$$
\item For the restriction $q[0]: \cH^*_{(2)}(\widetilde{M}_\infty) \to \image d \subset \dom^*_{(2)}(\partial \widetilde{M})$ we define
$$
H_\lambda := q[0]^{-1}(E_\lambda) \subset  \cH^*_{(2)}(\widetilde{M}_\infty) \cong  \cohr^*(\widetilde{M}_\infty).
$$
\end{enumerate}
\end{defn}

\noindent By construction, we have an estimate of the $\Gamma$-codimension of
$E_\lambda \subset \overline{\image d}$ by
\begin{equation*}
	\textup{codim}_\Gamma(E_\lambda \subset  \overline{\image d}) \le
         \trace_\Gamma \, \chi_{(0,\lambda]} (d^*d).
\end{equation*}

\noindent By Corollary \ref{codim}, for each $\epsilon > 0$ there exists 
$\lambda > 0$ sufficiently small so that 
\begin{equation}\label{small-codimensions}
\begin{split}
&\textup{codim}_\Gamma(E_\lambda \subset \overline{\image d}) < \epsilon, \\
&\textup{codim}_\Gamma (K_\lambda \subset q^{-1} (\overline{\image d}) \equiv \ker [q] )< \epsilon \\
&\textup{codim}_\Gamma (H_\lambda \subset \cH_{(2)}^*(\widetilde{M}_\infty)) < \epsilon.
\end{split}
\end{equation}

\begin{prop}\label{prop_6'}
	$K_\lambda \subset \image [\iota]$.
\end{prop}

\begin{proof}
Consider $[\w] \in K_\lambda \subset \cohr^*(\widetilde{M})$ with harmonic representative $\w \in \cH_{(2)}^*(\widetilde{M})$.
Then $q\w \in E_\lambda \subset \image d$ and hence there exists $\alpha \in \dom_{(2)}^*(\partial \widetilde{M})$
such that $q\w = d\alpha$. Fix any smooth cutoff function $\phi \in C^\infty_0[0,1)$, such that
$\phi \equiv 1$ identically near $x=0$. Consider $\phi \alpha$, extended trivially to a form in $\dom^*_{\max}(\widetilde{M})$.
Set
$$
W:= \w - d(\phi \alpha).
$$
Note that by definition, $qW=0$.
Denote by $d_{\min}$ and $d_{\max}$ the minimal and maximal closed extensions of 
the exterior derivative on $\widetilde{M}$.	Then by the $L^2$-Stokes theorem in Theorem \ref{Stokes-thm}, we find 
for any smooth $v \in \dom_{\max}(d^t) = \dom(d_{\min}^*)$
$$
\langle W, d^t v \rangle_{L^2\Omega^*(\widetilde{M})} = \pm \int_{\partial \widetilde{M}} qW \wedge * qv = 0.
$$	
Consequently, $W \in \overline{\image d_{\min}^*}^\perp = \ker d_{\min}$. 
By the weak Kodaira decomposition
$$
\iota W = \w - d(\phi \alpha) \in \cH^*_{(2)}(\widetilde{M}) \oplus \overline{\image d_{\max}}.
$$
Hence $[\iota] [W] \equiv [\iota W] = [\w]$. This proves the inclusion.
\end{proof}

\begin{prop}
$\overline{\image [\iota]} = \overline{\image [r]} = \ker [q].$
\end{prop}

\begin{proof}
First note that by Propositions \ref{prop_3'} and \ref{prop_6'}
$$
K_\lambda \subset \image [r] \subset \ker [q].
$$
Since $K_\lambda \subset \ker [q]$ is of arbitrarily small $\Gamma$-codimension as $\lambda \to 0$,
we conclude that $\overline{\image [r]} = \ker [q]$. Moreover, since by Proposition \ref{prop_6'},
$K_\lambda \subset \image [\iota]$, we conclude again from $K_\lambda \subset \ker [q]$ being
of arbitrarily small $\Gamma$-codimension as $\lambda \to 0$, that $\ker [q] \subseteq \image [\iota]$.
Lining up the inequalities, we find
$$
\ker [q] \subseteq \image [\iota] \subseteq \overline{\image [r]} = \ker [q],
$$
where the second inclusion follows from Proposition \ref{prop_3'}. Hence all inclusions are equalities and 
the statement follows. 
\end{proof}

\subsection{Various $\Gamma$-signatures}
Let $M$ and hence also $\widetilde{M}$ now have dimension $4n$ and we restrict the 
actions of $[\iota], [r]$ and $[q]$ to degree $2n$. We then define as before in Lemmas
\ref{deRham-signature} and \ref{Hodge-signature} the signature pairings of $\widetilde{M}$.
We recall, preliminarely, that the $\Gamma$-dimensions of  $\cohr^{2n} (\widetilde{M},\partial \widetilde{M})$
and $\cH_{(2)}^{2n}(\widetilde{M}_\infty)$ are finite.  See Remark \ref{finite-gamma-dimension}
and Subsection \ref{subsub:cylindrical}.

\begin{lemma} The signature pairings on $\widetilde{M}$ are defined as follows.
\begin{enumerate}
\item The de Rham signature pairing on $\widetilde{M}$ is defined by
\begin{equation}
\begin{split}
s :  \cohr^{2n} (\widetilde{M},\partial \widetilde{M}) \times    \cohr^{2n} (\widetilde{M},\partial \widetilde{M}) \to \mathbb{C}, 
\quad  ([v], [w]) \mapsto \int_{\widetilde{M}} v \wedge w,
\end{split}
\end{equation}
and is a well-defined degenerate bilinear form with radical 
given by $\ker [\iota]$.
\item The Hodge $L^2$-signature pairing on $\widetilde{M}$ is defined by 
\begin{equation}
\begin{split}
s_{\infty}:  \cH_{(2)}^{2n}(\widetilde{M}_\infty) \times  \cH_{(2)}^{2n}(\widetilde{M}_\infty) \to \mathbb{C},
\quad (\w, \eta) \mapsto \int_{\widetilde{M}} \w \wedge \, \eta,
\end{split}
\end{equation}
and is a non-degenerate bilinear form. 
\end{enumerate}
\end{lemma}

\begin{proof}
In order to see that $s$ is well-defined, consider any $[v], [w] \in \cohr^{2n} (\widetilde{M},\partial \widetilde{M})$
with representatives $v+v'$ and $w+w'$, respectively, where $v',w' \in \overline{\image d_{\min}}$. We may choose
sequences $(\alpha_j), (\beta_j) \subset \mathscr{E}_{\min}(\widetilde{M})$ such that $d\alpha_j \to v'$ and $d\beta_j \to w'$ 
in $L^2\Omega^*(\widetilde{M})$. Then we compute using the argument of Lemma \ref{deRham-signature} in the second equation
$$
\int_{\widetilde{M}} (v+v') \wedge (w+w') = \lim_{j\to \infty} \int_{\widetilde{M}} (v+d\alpha_j) \wedge (w+d\beta_j)
= \int_{\widetilde{M}} v\wedge w.
$$
The proof that the radical of $s$ is given by $\ker [\iota]$ is exactly the same as in Lemma \ref{deRham-signature}
with $\image d_{\min}$ replaced by $\overline{\image d_{\min}}$. The proof of $s_\infty$ being non-degenerate
is exactly the same as in Lemma \ref{Hodge-signature}.
\end{proof}

We can now define the corresponding $\Gamma$-signatures.

\begin{defn}\label{gamma-sign-def}
\begin{enumerate}
\item  The de Rham $\Gamma$-signature $\textup{sign}_\Gamma(s)$ is defined as the 
$\Gamma$-signature of the non-degenerate symmetric bilinear form
\begin{equation}
s : \overline{\image [\iota]} \times \overline{\image [\iota]} \to \mathbb{C},
\end{equation}
induced by the de Rham signature pairing $s$ on 
$$
\frac{\cohr^{2n} (\widetilde{M},\partial \widetilde{M}) }{\textup{Rad}(s)} \equiv 
\frac{\cohr^{2n} (\widetilde{M},\partial \widetilde{M}) }{\ker [\iota]} 
\cong \overline{\image [\iota]}.$$
Recall, this is by definition the $\Gamma$-dimension of $V_+$ minus the $\Gamma$-dimension of $V_-$
with $V_+:=\chi_{(0,\infty)} (A)$ and $V_-:= \chi_{(-\infty,0)} (A)$, where $A$ the unique self-adjoint operator 
associated to $s$. \medskip

\item The Hodge $\Gamma$-signature $\textup{sign}_\Gamma(s_\infty)$ is defined as the $\Gamma$-signature
of the non-degenerate bilinear symmetric form $s_\infty$.  \end{enumerate}
\end{defn}

In order to compare intersection forms, we need a result, 
corresponding to Proposition \ref{prop_6}. However, 
for any $\w \in \cH_{(2)}^*(\widetilde{M}_\infty)$ the restriction $q[0]\w \in \overline{\image d}$
is not necessarily exact and therefore the argument of Proposition \ref{prop_6} doesn't translate 
to the setting of coverings directly. However, if $\w \in H_\lambda$ then 
$q[0] \w \in E_\lambda \subset \image d$ and hence we still get an analogue of Proposition \ref{prop_6}
as follows. Consider for any $\lambda>0$ the subspace 
$$
L_\lambda:=\overline{\image ([r](H_\lambda))},
$$ 
and the pairings
\begin{equation}
\begin{split}
&s_{\infty,\lambda} = s_\infty \restriction H_\lambda: H_\lambda \times H_\lambda \to \C, \\
&s_\lambda: L_\lambda \times L_\lambda \to \C, \quad ([u],[v]) \mapsto \int_{\widetilde{M}} u\wedge v.
\end{split}
\end{equation}

\begin{prop}\label{signatures-lambda} The pairing $s_{\infty,\lambda}$ descends to the pairing $s_\lambda$ 
with same $\Gamma$-signatures 
$$
\textup{sign}_\Gamma(s_{\infty,\lambda})  = \textup{sign}_\Gamma(s_\lambda).
$$
\end{prop}

\begin{proof}
The proof is a verbatim repetition of Proposition \ref{prop_6}.
\end{proof}

The next result proves the first statement in our main Theorem \ref{main-2}.

\begin{thm} The de Rham $\Gamma$-signature and the Hodge $\Gamma$-signature coincide.
\begin{equation}
	\left(\textup{sign}^{\Gamma}_{{\rm Ho}} (\overline{M}_{\Gamma,\infty})):= \right) \textup{sign}_\Gamma(s_\infty) = \textup{sign}_\Gamma(s)
	\left(=:\textup{sign}^{\Gamma}_{{\rm dR}} (\overline{M}_\Gamma,\partial \overline{M}_\Gamma))\right).
\end{equation}
\end{thm}

\begin{proof}
Since by Corollary \ref{codim} the $\Gamma$-codimensions are right-continuous and in fact arbitrarily small, see \eqref{small-codimensions}, i.e. for any $\epsilon>0$ there
exists $\lambda>0$ sufficiently small such that 
\begin{equation}\label{HLe}
\textup{codim}_\Gamma (H_\lambda \subset \cH_{(2)}^*(\widetilde{M}_\infty)) < \epsilon,
\quad \textup{codim}_\Gamma(L_\lambda \subset \overline{\image [r]}) < \epsilon.
\end{equation}
The first and the second estimate in \eqref{HLe} imply that 
\begin{equation}
	\left| \textup{sign}_\Gamma(s_\infty) - \textup{sign}_\Gamma(s_{\infty, \lambda}) \right| < \epsilon,
	\quad \left| \textup{sign}_\Gamma(s) - \textup{sign}_\Gamma(s_{\lambda}) \right| < \epsilon,
\end{equation}
respectively, By Proposition \ref{signatures-lambda} we conclude
\begin{equation}
	\left| \textup{sign}_\Gamma(s) - \textup{sign}_\Gamma(s_{\infty}) \right| < 2 \epsilon.
\end{equation}
We finally obtain,  by taking $\epsilon \to 0$, that the two $\Gamma$-signatures coincide: $\textup{sign}_\Gamma(s) = \textup{sign}_\Gamma(s_{\infty})$.
\end{proof}

Note that exactly as in Corollary \ref{image-r} the argument shows that 
$[r]$ is injective with dense image.

\section{Signature formula on stratified Witt--spaces with boundary}\label{main-section}

In this section we prove the signature formula in our main Theorem \ref{main-1}, that is,
the signature theorem on a compact smoothly stratified space $(M,g)$ with boundary, satisfying the geometric Witt assumption.
Our proof adapts the classical argument of Atiyah, Patodi and Singer \cite{APSa} to the singular setting.
Recall that the signature operator $D$ has the following form in the collar neighborhood $[0,1) \times \partial M$
of the boundary
\begin{equation}\label{product-form-D}
D = \sigma \left( \frac{\partial}{\partial x} + B \right),
\end{equation}
where the tangential operator $B$ acting on differential forms over $\partial M$ is essentially
self-adjoint by the geometric Witt assumption. We identify $B$ with its unique self-adjoint extension,
which is discrete by \cite[Theorem 1.1]{ALMP1} and defines the closed extension of $D$ with Atiyah-Patodi-Singer
boundary conditions by taking the graph-closure of the core domain \eqref{core-domain}. $D$ is notationally
identified with the resulting closed extension.\medskip

We write $D_\infty$ for the signature operator on $M_\infty$. Note that the Hodge Dirac 
operator $\slashed{\partial}_\infty = d+d^t$ on $M_\infty$ is related to $D_\infty$ by
$\slashed{\partial}_\infty = D_\infty \oplus D^t_\infty$. 
By the geometric Witt assumption, the minimal and maximal domains for 
$\slashed{\partial}_\infty$ coincide, and hence same holds for $D_\infty$ and $D^t_\infty$
\begin{equation}\label{min-max-same}
\begin{split}
&\dom_{\min}(D_\infty) = \dom_{\max}(D_\infty) \equiv \dom (D_\infty),\\ 
&\dom_{\min}(D^t_\infty) = \dom_{\max}(D^t_\infty) \equiv \dom (D_\infty^*).
\end{split}
\end{equation}
We henceforth identify $D_\infty$ notationally with its unique closed extension. 
\begin{defn}
We call any differential form $u$ on $M_\infty$ an extended $L^2$-section 
	if $u\restriction M \in L^2\Omega^*(M)$ and $u \restriction \partial M \times (-\infty,0] = g + u_\infty$ 
	where $g \in L^2\Omega^*(\partial M \times (-\infty,0])$ and $u_\infty \in \ker B$. We call $u_\infty$
	the limiting value of $u$.
\end{defn}

By construction, the extended $L^2$-solutions have an element of $\ker B$ as a limiting 
value at minus infinity of the cylindrical end $\partial M \times (-\infty,0]$. Furthermore,
in full accordance with the notation of \cite{APSa} we proceed with the following definition.

\begin{defn}\label{h}
We write $D_\infty$ for the signature operator on $M_\infty$. 
\begin{enumerate}
	\item We denote the space
	of extended $L^2$-sections, solving $D_\infty u = 0$	by\\ $\textup{ext-}\ker_{(2)} D_\infty$, where
	$D_\infty u$ is defined in the distributional sense. We call such solutions extended $L^2$-solutions and write
	$$
	h_\infty(D) := \dim \, \{ u_\infty \in \ker B \mid u \in \textup{ext-}\ker_{(2)} D_\infty \}.
	$$
	\item We denote the space of extended $L^2$-sections, solving $D^t_\infty u = 0$	
	by\\ $\textup{ext-}\ker_{(2)} D^*_\infty$, where $D^t_\infty u$ is defined in the distributional sense. 
	We call such solutions extended $L^2$-solutions and write
	$$
	h_\infty(D^*) := \dim \, \{ u_\infty \in \ker B \mid u \in \textup{ext-}\ker_{(2)} D^*_\infty \}.
	$$
	\item We denote the kernel of $D_\infty$ as $\ker_{(2)} D_\infty$. By definition, it 
	stands for the space of $L^2$-sections, solving $D_\infty u = 0$.
	We denote the kernel of $D^*_\infty$ as $\ker_{(2)} D^*_\infty$. By definition, it 
	stands for the space of $L^2$-sections, solving $D^t_\infty u = 0$. We write
	$$
	h(D):= \dim \ker_{(2)} D_\infty, \quad h(D^*):= \dim \ker_{(2)} D^*_\infty.
	$$
\end{enumerate}
\end{defn}

By construction, the various dimensions in the Definition \ref{h} are related by
\begin{equation}\label{hh}
\begin{split}
&\dim \textup{ext-}\ker_{(2)} D_\infty = h(D) + h_\infty(D), \\
&\dim \textup{ext-}\ker_{(2)} D^*_\infty = h(D^*) + h_\infty(D^*).
\end{split}
\end{equation}

\begin{prop}\label{prop3_2}
$$
\ker D \cong \ker_{(2)} D_\infty, \quad \ker D^* \cong \textup{ext-}\ker_{(2)} D^*_\infty.
$$
\end{prop}

\begin{proof}
	Proof is exactly as in \cite[Proposition 3.11]{APSa}, using disreteness of $B$. 
\end{proof}

\begin{cor}\label{cor3_3}
	\begin{equation*}
		\ind D = \ker D - \ker D^* = h(D)  - h(D^*) - h_\infty(D^*).
	\end{equation*}
\end{cor}

\begin{proof}
The statement follows from Proposition \ref{prop3_2} and \eqref{hh}.
\end{proof}

The next result is concerned with the question if solutions to $D_\infty^t D_\infty u = 0$ are in fact
solutions to $D_\infty u = 0$. While this is classical in the compact smooth setting, this is not obvious
in the singular non-compact case. 	
	
\begin{prop}\label{prop3_4}
Consider the space of extended $L^2$-harmonic forms $\textup{ext-}\cH_{(2)}^*(M_\infty)$, 
which is by definition the space of extended $L^2$-sections $u$ solving $d_\infty u=d^t_\infty u=0$. 
Here $d_\infty$ denotes the exterior derivative on $M_\infty$. Recall the involution $\tau$ from \eqref{tau}. 
Its action on differential forms over $M_\infty$ decomposes the differential forms into the $(+1)$- and 
$(-1)$-eigenspaces of $\tau$, written as $\Omega^+(M_\infty)$ and $\Omega^-(M_\infty)$, respectively. We define 
\begin{align*}
\textup{ext-}\cH_{(2)}^*(M_\infty)^\pm := \textup{ext-}\cH_{(2)}^*(M_\infty) \cap \Omega^\pm(M_\infty).
\end{align*}
Then the following statements hold
	\begin{equation}\label{useful0}
		\begin{split}
		&\ker_{(2)} D_\infty = \ker_{(2)} D_\infty^t D_\infty \\
		&\textup{ext-}\ker_{(2)} D_\infty = \textup{ext-}\ker_{(2)} D^t_\infty D_\infty = \textup{ext-}\cH_{(2)}^*(M_\infty)^+.
		\end{split}
	\end{equation}
Same holds with $D,D^t$ interchanged
\begin{equation}\label{useful1}
		\begin{split}
		&\ker_{(2)} D^*_\infty = \ker_{(2)} D_\infty D_\infty^t, \\
		&\textup{ext-}\ker_{(2)} D^*_\infty = \textup{ext-}\ker_{(2)} D_\infty D^t_\infty = \textup{ext-}\cH_{(2)}^*(M_\infty)^-.
		\end{split}
	\end{equation}
\end{prop}

\begin{proof}
The proof is partly an adaptation of \cite[Proposition 3.15]{APSa}, using 
discreteness of $B^2$ and the $L^2$-Stokes Theorem \ref{Stokes-thm} for $D$, which is non-trivial in our singular setting.
Note first that over the cylindric part $(-\infty, 0] \times \partial M \subset M_\infty$ we have
$$
D_\infty = \sigma \left( \frac{\partial}{\partial x} + B \right), 
\quad D_\infty^t D_\infty = - \frac{\partial^2}{\partial x^2} + B^2.
$$
Consider the spectrum $\{\mu, \psi_\mu\}$
of eigenvalues and eigensections of $B^2$, where the eigensections 
$\{\psi_\mu\}_\mu$ form an orthonormal basis of $L^2\Omega^*(\partial M)$ and all eigenvalues $\mu \geq 0$.
Then any $u \in L^2\Omega^*(M_\infty)$ with $D_\infty^t D_\infty u = 0$ can be written
over the cylindric part $(-\infty, 0]_x \times \partial M \subset M_\infty$ as
\begin{align}\label{u-plus-minus-expansion}
u(x) = \sum_{\mu \geq 0} \left(a_\mu e^{\mu x} + b_\mu e^{-\mu x}\right) \psi_\mu.
\end{align}
Since $u \in L^2\Omega^*(M_\infty)$, $a_0=0$ and $b_\mu = 0$ for all eigenvalues $\mu$.
Applying $D$ to the  expansion \eqref{u-plus-minus-expansion}, we conclude
\begin{equation}\label{exponential-decay}
\|u(x)\|_{L^2\Omega^*(\partial M)}, \|Du(x)\|_{L^2\Omega^*(\partial M)} \leq C e^{\alpha x}, \quad \textup{for some} \quad \alpha, C > 0.
\end{equation}
We now want to apply $L^2$-Stokes Theorem \ref{Stokes-thm} for $D$ on the compact stratified space 
$M_x := [x,0] \times \partial M \cup_{\partial M} M
\subset M_\infty$ for almost all $x<0$. Similar to Figure \ref{fig:CutOff} we consider the following set of 
cutoff functions. The cutoff functions define functions on the collar $[x,0] \times \partial M \subset M_x$
and extend trivially to the interior of $M_x$. We still denote these extensions by 
$\phi, \psi , \chi \in C^\infty(M_x)$.

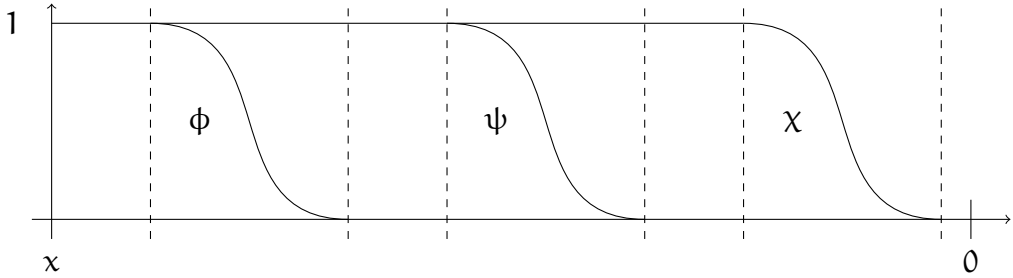
\begin{figure}[h]
\begin{center}

\begin{tikzpicture}[scale=1.3]
\draw[->] (-0.2,0) -- (9.7,0);
\draw[->] (0,-0.2) -- (0,2.2);

\draw (-0.2,2) node[anchor=east] {$1$};

\draw (9.3,-0.2) node[anchor=north] {$0$} -- (9.3,0.2);
\draw (0,-0.45) node {$x$};
\draw (0,2) -- (7,2);
\draw (1,2) .. controls (2.4,2) and (1.6,0) .. (3,0);
\draw[dashed] (1,-0.2) -- (1,2.2);
\draw[dashed] (3,-0.2) -- (3,2.2);
\draw (4,2) .. controls (5.4,2) and (4.6,0) .. (6,0);
\draw[dashed] (4,-0.2) -- (4,2.2);
\draw[dashed] (6,-0.2) -- (6,2.2);
\draw (7,2) .. controls (8.4,2) and (7.6,0) .. (9,0);
\draw[dashed] (7,-0.2) -- (7,2.2);
\draw[dashed] (9,-0.2) -- (9,2.2);

\draw (1.5,1) node {$\phi$};
\draw (4.5,1) node {$\psi$};
\draw (7.5,1) node {$\chi$};

\end{tikzpicture}

\caption{The cutoff functions $\phi, \psi$ and $\chi$.}
\label{fig:CutOff2}
\end{center}
\end{figure}

We compute for any $u \in \dom_{\max}(D)$ and $v \in \dom_{\max}(D^t)$
(where $D$ stands for the signature operator on $M_x$ here) 
\begin{equation}\label{Green-a}
\begin{split}
\langle Du, v\rangle_{L^2\Omega^*(M_x)} &= \langle D (\psi u), v\rangle_{L^2\Omega^*(M_x)} + \langle D (1-\psi) u, v\rangle_{L^2\Omega^*(M_x)}
\\ &= \langle D (\psi u), \chi v\rangle_{L^2\Omega^*(M_x)} + \langle D (1-\psi) u, (1-\phi) v\rangle_{L^2\Omega^*(M_x)}.
\end{split}
\end{equation}
The function $(1-\psi) u$ can be viewed as an element of $\dom (D_\infty)$.
The function $(1-\chi) v$ can be viewed as an element of $\dom (D^*_\infty)$.
Hence we compute 
\begin{equation}\label{Green-b}
\begin{split}
\langle D (1-\psi) u, (1-\phi) v\rangle_{L^2\Omega^*(M_x)} &= \langle (1-\psi) u, D^t (1-\phi) v \rangle_{L^2\Omega^*(M_x)}
\\ & = \langle (1-\psi) u, (1-\phi) D^t v \rangle_{L^2\Omega^*(M_x)}
\\ & + \langle (1-\psi) u, D^t (1-\phi ) v -  (1-\phi) D^t v \rangle_{L^2\Omega^*(M_x)} 
\\ & = \langle (1-\psi) u, D^t  v \rangle_{L^2\Omega^*(M_x)},
\end{split}
\end{equation}
where in the last equality we used that by construction, 
$(1-\phi) (1-\psi) = (1-\psi)$ and the fact that, by a combination of the Leibniz rule for the exterior derivative and 
\eqref{Leibniz-dt}, the support of $D^t (1-\phi ) v -  (1-\phi) D^t v$ is contained in the support of $d(1-\phi)$,
which is disjoint from the support of $(1-\psi)$. This is why $\langle (1-\psi) u, D^t (1-\phi ) v -  (1-\phi) D^t v \rangle_{L^2\Omega^*(M_x)}$
vanishes. \medskip

Next, we compute by the $L^2$-Stokes Theorem \ref{Stokes-thm} using the product structure \eqref{product-form-D} of $D$ and
essential self-adjointness of $B$ for almost all $x\in (-\infty, 0)$
\begin{equation}\label{Green-c}
\begin{split}
\langle D (\psi u), \chi v\rangle_{L^2\Omega^*(M_x)} 
&= \langle \psi u, D^t (\chi v) \rangle_{L^2\Omega^*(M_x)} + \langle \sigma q[x]u , q[x]v \rangle_{L^2\Omega^*(\partial M_x)} \\
& = \langle \psi u, D^t (\chi  v) - \chi D^t v \rangle_{L^2\Omega^*(M_x)} + \langle \psi u, \chi D^t v \rangle_{L^2\Omega^*(M_x)} \\
& + \langle \sigma q[x]u , q[x]v \rangle_{L^2\Omega^*(\partial M_x)} = \langle \psi u, D^t v \rangle_{L^2\Omega^*(M_x)} 
\\ &+  \langle \sigma q[x]u , q[x]v \rangle_{L^2\Omega^*(\partial M_x)},
\end{split}
\end{equation}
where in the last equality we used the fact that by construction, 
$\chi \psi = \psi$ and again the fact that the support of $D^t (\chi  v) - \chi D^t v$ 
lies in the support of $d\chi$, which is disjoint from the support of $\psi$. This is why 
$\langle \psi u, D^t (\chi  v) - \chi D^t v \rangle_{L^2\Omega^*(M_x)}$ vanishes.
\medskip

By \eqref{Green-b} and \eqref{Green-c}, we conclude from \eqref{Green-a} 
for almost all $x\in (-\infty, 0)$
\begin{equation}
\begin{split}
\langle Du, v\rangle_{L^2\Omega^*(M_x)} &= \langle (\psi u), D^t v\rangle_{L^2\Omega^*(M_x)} + \langle (1-\psi) u, D^t v\rangle_{L^2\Omega^*(M_x)}
\\ &+ \langle \sigma q[x]u , q[x]v \rangle_{L^2\Omega^*(\partial M_x)} = \langle u , D^t v\rangle_{L^2\Omega^*(M_x)} 
\\ &+ \langle \sigma q[x]u , q[x]v  \rangle_{L^2\Omega^*(\partial M_x)}.
\end{split}
\end{equation}
From here we conclude for any $L^2$-solution to $D_\infty^t D_\infty u = 0$
(thus $u$ and $Du$ are in particular solutions to elliptic equations and therefore smooth in the regular part,
and hence admit a well-defined restriction to $\partial M_x$)
\begin{equation}\label{D-Stokes}
\begin{split}
\| D_\infty u \|^2_{L^2\Omega^*(M_\infty)} &= \lim_{x\to -\infty} 
\langle Du, Du\rangle_{L^2\Omega^*(M_x)} \\ &= \lim_{x\to -\infty} 
\left(\langle Du, Du\rangle_{L^2\Omega^*(M_x)} - \langle D^t D u, u\rangle_{L^2\Omega^*(M_x)}\right)
\\ &= \lim_{x\to -\infty} \langle - q[x]Du , \sigma q[x]u  \rangle_{L^2\Omega^*(\partial M_x)} = 0,
\end{split}
\end{equation}
where in the last we used the exponential decay in \eqref{exponential-decay}.
This proves 
$$
\ker_{(2)} D_\infty = \ker_{(2)} D_\infty^t D_\infty.
$$
For the statement on the extended $L^2$-sections, solving $D^t_\infty D_\infty u = 0$, 
the series representation \eqref{u-plus-minus-expansion} still holds, with 
$a_0=0$ and $b_\mu = 0$ for all eigenvalues $\mu>0$. In particular, the series may admit
constant terms (times $\phi_0$). Thus, in contrast to \eqref{exponential-decay}, 
we can only conclude that $u(x)$ is bounded as $x\to -\infty$. However, applying $D_\infty$
to $u$ removes the constant terms and we can still conclude 
\begin{equation}\label{exponential-decay-2}
 \|Du(x)\|_{L^2\Omega^*(\partial M)} \leq C e^{\alpha x}, \quad \textup{for some} \quad \alpha, C > 0.
\end{equation}
Proceeding verbatim to the previous case, we find 
$$
\textup{ext-}\ker_{(2)} D_\infty = \textup{ext-}\ker_{(2)} D^t_\infty D_\infty.
$$
The statements for $D$ and $D^t$ interchanged are discussed similarly. It remains to prove
the relations to the extended $L^2$-harmonic forms
\begin{equation}\label{ext-D-vs-harmonic}
\begin{split}
&\textup{ext-}\cH_{(2)}^*(M_\infty)^+ = \textup{ext-}\ker_{(2)} D_\infty, \\
&\textup{ext-}\cH_{(2)}^*(M_\infty)^- = \textup{ext-}\ker_{(2)} D^*_\infty.
\end{split}
\end{equation}
Indeed, the inclusions $\subseteq$ are obvious. For the converse inclusions, consider e.g.
$u \in \textup{ext-}\ker_{(2)} D_\infty$. Using Proposition \ref{prop3_4}, this is an extended $L^2$-solution 
to $D_\infty^t D_\infty u \equiv (d^t_\infty d_\infty + d_\infty d^t_\infty)u= 0$. Using \eqref{u-plus-minus-expansion}, we may conclude similar to 
\eqref{exponential-decay}, using the fact that $B^2$ commutes with the exterior derivative on $\partial M$,
that $\|u(x)\|_{L^2\Omega^*(\partial M)}$ is bounded and $\| (d_\infty u) (x)\|_{L^2\Omega^*(\partial M)}, \| (d^t_\infty u) (x)\|_{L^2\Omega^*(\partial M)}$ 
are exponentially decaying as $x\to -\infty$. Then as in \eqref{D-Stokes}, we conclude using the $L^2$-Stokes theorem in Theorem
\ref{Stokes-thm} 
\begin{equation}
\begin{split}
\| d_\infty u \|^2_{L^2\Omega^*(M_\infty)} + \| d^t_\infty u \|^2_{L^2\Omega^*(M_\infty)} &=
 \lim_{x\to -\infty} \langle q[x] du , q[x] u  \rangle_{L^2\Omega^*(\partial M_x)} 
\\ &- \lim_{x\to -\infty} \langle q[x] d^tu , q[x] u  \rangle_{L^2\Omega^*(\partial M_x)} = 0,
\end{split}
\end{equation}
This proves \eqref{ext-D-vs-harmonic}. 
\end{proof}

Our next two results relate extended solutions to the kernel of $B$.

\begin{prop}\label{prop3_5}	
	\begin{equation*}
		\dim \ker B = h_\infty (D) + h_\infty (D^*)
	\end{equation*}
\end{prop}

\begin{proof} The proof of \cite[(3.25)]{APSa} carries over verbatim. We nevertheless repeat the
arguments here to point out the precise results of the paper we use here. Indeed, consider
$D$ and $D^t$ in the collar neighborhood of $\partial M$
\begin{equation}\label{product-D-t}
D = \sigma \left( \frac{\partial}{\partial x} + B \right),
\quad D^t = -\sigma \left( \frac{\partial}{\partial x} - B \right).
\end{equation}
Recall Theorem \ref{index-main} and Corollary 
\ref{cor3_3} on the index of $D$
\begin{align*}
\ind \, D = \int\limits_M L(M) + 
\sum_{\alpha \in A} \int\limits_{\ Y_\alpha} b_{\alpha} 
- \frac{\dim \ker B + \eta(B)}{2} = h(D)  - h(D^*) - h_\infty(D^*).
\end{align*}
Repeating the arguments of Theorem \ref{index-main} and Corollary 
\ref{cor3_3} with $D$ replaced by $D^t$ (with Atiyah Patodi Singer boundary conditions
given in terms of the negative spectral projection of $B$) we conclude in view of \eqref{product-D-t}
\begin{align*}
\ind \, D^* &= \int\limits_M -L(M) + 
\sum_{\alpha \in A} \int\limits_{\ Y_\alpha} (-b_{\alpha}) 
- \frac{\dim \ker B - \eta(B)}{2} \\ &= h(D^*)  - h(D) - h_\infty(D).
\end{align*}
Adding up the two index formulae for $D$ and $D^*$, we arrive at the statement.
\end{proof}

\begin{prop}\label{prop3_6}
	\begin{equation*}
		h_\infty (D) = h_\infty (D^*) = \frac{\dim \ker B}{2}. 
	\end{equation*}
\end{prop}

\begin{proof}
In view of Proposition \ref{prop3_5} it suffices to prove $h_\infty(D) = h_\infty (D^*)$.
Consider the space of extended $L^2$-harmonic forms $\textup{ext-}\cH_{(2)}^*(M_\infty)$.
Clearly $\cH_{(2)}^*(M_\infty) \subset \textup{ext-}\cH_{(2)}^*(M_\infty)$. 
The restriction to $M\subset M_\infty$ extends $[r]: \cH_{(2)}^*(M_\infty) \to \coh^* (M)$. 
We consider the sequence of maps
\begin{equation*}
\begin{split}
&\textup{ext-}\cH_{(2)}^*(M_\infty)  \xrightarrow{\quad \beta \quad } \coh^*(M)  \xrightarrow{\quad [q] \quad } \coh^*(\partial M)\\
&\qquad  \qquad u  \qquad \longmapsto \qquad  [u\vert_M]  \quad \longmapsto \quad [q u].
\end{split}
\end{equation*}
Any $u\in \textup{ext-}\cH_{(2)}^*(M_\infty)$ can be written over $\partial M \times (-\infty, 0]$
as $u_\infty + g$, where $g \in L^2\Omega^*(\partial M \times (-\infty, 0])$
and $u_\infty = u_0 + u_1 \wedge dx$ with $u_0,u_1 \in \cH_{(2)}^*(\partial M)$ independent of $x$.
By Proposition \ref{prop_3} we find for the composition $[q] \circ \beta$
$$
[q] \circ \beta : \textup{ext-}\cH_{(2)}^*(M_\infty) \to  \coh^*(\partial M),
\quad u \mapsto u_0.
$$
We write $\beta^\pm := \beta \restriction \textup{ext-}\cH_{(2)}^*(M_\infty)^\pm$, where
the restrictions 
$$
\textup{ext-}\cH_{(2)}^*(M_\infty)^\pm = \textup{ext-}\cH_{(2)}^*(M_\infty) \cap \Omega^\pm(M_\infty),
$$
were introduced in Proposition \ref{prop3_4}. Given $u \in \textup{ext-}\cH_{(2)}^*(M_\infty)^\pm$, 
we find that $\tau(u_0) = u_1 \wedge dx$ up to a sign.
Hence $u_0=0$ implies $u_1=0$. Consequently, if 
$u \in \ker [q] \circ \beta^\pm$, then $u_\infty = 0$. We therefore conclude
\begin{align}\label{qbeta}
\ker \left([q] \circ \beta^\pm\right) = \cH_{(2)}^*(M_\infty) \cap \Omega^\pm(M_\infty)=: \cH_{(2)}^*(M_\infty)^\pm.
\end{align}
Note that $\cH_{(2)}^*(M_\infty)^+= \ker_{(2)} D_\infty$ and $\cH_{(2)}^*(M_\infty)^- = \ker_{(2)} D^*_\infty$.
Combining \eqref{ext-D-vs-harmonic} and \eqref{qbeta} we arrive at the following relations
\begin{equation}\label{h-im}
\begin{split}
&h_\infty (D) = \dim \frac{\textup{ext-}\ker_{(2)} D_\infty}{\ker_{(2)} D_\infty} 
= \dim \frac{\textup{ext-}\cH_{(2)}^*(M_\infty)^+}{\ker_{(2)} D_\infty} 
= \dim (\textup{im} [q] \circ \beta^+), \\
&h_\infty (D^*) = \dim \frac{\textup{ext-}\ker_{(2)} D^*_\infty}{\ker_{(2)} D^*_\infty} 
= \dim \frac{\textup{ext-}\cH_{(2)}^*(M_\infty)^-}{\ker_{(2)} D^*_\infty} 
= \dim (\textup{im} [q] \circ \beta^-).
\end{split}
\end{equation}
Consider the exact sequence (from the long exact sequence in cohomology)
$$
\coh^*(M) \xrightarrow{\quad [q] \quad } \coh^*(\partial M)  \xrightarrow{\quad [\delta] \quad } \coh^*(M,\partial M),
$$
where $[\delta]$ is the usual connecting homomorphism. Poincare duality implies that $\textup{im} [q]$ is dual to 
its orthogonal complement and hence
$$
\dim \textup{im} [q] = \frac{\dim \coh^*(\partial M)}{2} = \frac{\dim \ker B}{2}.
$$
Thus in view of \eqref{h-im} we conclude
\begin{align*}
&h_\infty (D) = \dim (\textup{im} [q] \circ \beta^+) \leq \dim \textup{im} [q] = \frac{\dim \ker B}{2}, \\
&h_\infty (D^*) = \dim (\textup{im} [q] \circ \beta^-) \leq \dim \textup{im} [q] = \frac{\dim \ker B}{2}.
\end{align*}
The statement now follows from Proposition \ref{prop3_5}.
\end{proof}

We can now prove our first main result in Theorem \ref{main-1}. 
Using \eqref{s-s-infty} and \eqref{cor3_3} we obtain the following representation of the signature
\begin{equation}\label{equality+formula}
	\textup{sign}(s) = \textup{sign}(s_\infty) \equiv \ind \, D_\infty = \ind \, D + h_\infty(D^*).
\end{equation}
Using the index theorem im Theorem \ref{index-main} we conclude
\begin{equation}\label{signature-theorem-intermediate}
\begin{split}
&\textup{sign}(s) = \ind \, D + h_\infty(D^*) \\ &= \int\limits_M L(M) + 
\sum_{\alpha \in A} \int\limits_{\ Y_\alpha} b_{\alpha} 
- \frac{\dim \ker B + \eta(B)}{2} + h_\infty(D^*) \\ &= 
\int\limits_M L(M) + 
\sum_{\alpha \in A} \int\limits_{\ Y_\alpha} b_{\alpha} 
- \frac{\eta(B)}{2},
\end{split}
\end{equation}
where we used Proposition \ref{prop3_6} in the last equality. 
This proves the signature formula in Theorem \ref{main-1}, since $\eta(B) = 2 \eta(B_{\textup{even}})$
$$
\textup{sign}^{(2)}_{{\rm dR}} (\overline{M}, \partial \overline{M}) = \int_{M} L(M) + \sum_{\alpha \in A} \int\limits_{\ Y_\alpha} b_{\alpha} - \eta(B_{\textup{even}}).
$$

\begin{remark}
Using the analysis developed in \cite{AGR17} and the arguments in 
\cite[Section 6.5 and Section 9.3]{MelATP} it should be possible
to give a $b$-edge-calculus proof of the signature formula on Witt spaces with boundary. We have privileged a 
Atiyah-Patodi-Singer approach because it presents slightly more complicated 
arguments that will be used below on Galois coverings of Witt spaces with boundary.
\end{remark}

\section{The $L^2$-Gamma index theorem on $\widetilde{M}_\infty$ 
and the signature formula}\label{main-2-section}

Consider as before the Galois coverings $\overline{M}_\Gamma, \overline{M}_{\Gamma,\infty},\partial \overline{M}_\Gamma$ of 
$\overline{M}, \overline{M}_{\infty},\partial \overline{M}$, respectively. The regular part of $\overline{M}_\Gamma$ is denoted 
by $\widetilde{M}$, the regular part of $\overline{M}_{\Gamma,\infty}$ by $\widetilde{M}_\infty$, and the regular part of $\partial \overline{M}_\Gamma$ by $\partial \widetilde{M}$.\medskip

We now proceed towards the index theorem on the Galois covering $\widetilde{M}_\infty$.
We follow the general strategy of Vaillant \cite{Va}, where at various points crucial steps have to be 
altered due to presence of singularities. Hereby a central role is played by a perturbed signature 
operator. More precisely, over $(-\infty, 0) \times \partial \widetilde{M} \subset \widetilde{M}_\infty$
the equivariant lift $\widetilde{D}_\infty$ of the signature operator $D_\infty$ on $M_\infty,$ is given by the product form
\begin{align}
\widetilde{D}_\infty =  \widetilde{\sigma} \left( \frac{d}{dx} + \widetilde{B}\right),
\end{align}
where $\widetilde{\sigma}$ and $\widetilde{B}$ are the equivariant lifts of $\sigma$ 
and $B$ to $\partial \widetilde{M}$, respectively. The main technical issue is that 
$\widetilde{D}_\infty$ is not $\Gamma$-Fredholm, since $\widetilde{B}$ is not invertible.
Vaillant \cite{Va} overcomes this issue by perturbing the tangential operator $ \widetilde{B}$,
in such a way that it admits a spectral gap around zero. This makes the perturbation 
$\Gamma$-Fredholm. The $L^2$-Gamma index theorem for $\widetilde{D}_\infty$
then follows by establishing an index theorem for its $\Gamma$-Fredholm perturbation and taking limits.

\subsection{$\Gamma$-Fredholm perturbation of $\widetilde{D}_\infty$}
Consider for any $\varepsilon > 0$
\begin{equation*}
		\Pi_\varepsilon (x) := \left\{ \begin{array}{ll}
		1, & x \in (-\varepsilon,\varepsilon),\\
		0, & \text{elsewhere}
		\end{array} \right.
\end{equation*}
This defines a spectral projection $\Pi_\varepsilon (\widetilde{B})$ by spectral calculus and we obtain
perturbations of $\widetilde{B}$ and $\widetilde{D}$ as follows. Consider an auxiliary 
function $\theta \in C^\infty_0(-\infty,0)$ with $\theta \restriction (-\infty,-1] \equiv 1$ and 
 $\theta \restriction [-1/2,0) \equiv 0$. This defines a smooth function on $(-\infty, 0) \times \partial \widetilde{M}$
which we extend trivially by zero to $\widetilde{M}_\infty$. For any $u\geq0$ we now set
\begin{equation}
		\widetilde{D}_{\varepsilon,u} := 
		\widetilde{D}_\infty + \theta(x) \cdot \widetilde{\sigma} 
		\left( u -  \widetilde{B} \, \Pi_\varepsilon (\widetilde{B})\right).
\end{equation}
This is a bounded perturbation of $\widetilde{D}_\infty$. Moreover,
if we write $\widetilde{\slashed{\partial}}_\infty$ for the Hodge Dirac operator on $\widetilde{M}_\infty$, we find
\begin{equation}
		\widetilde{D}_{\varepsilon,u} \oplus \widetilde{D}^t_{\varepsilon,u} = 
		\widetilde{\slashed{\partial}}_\infty + \theta(x) \cdot \widetilde{\sigma} \oplus \widetilde{\sigma}^t 
		\left( u - \widetilde{B} \, \Pi_\varepsilon (\widetilde{B})\oplus  \widetilde{B} \, \Pi_\varepsilon (\widetilde{B})\right),
\end{equation}
which is a bounded perturbation of $\widetilde{\slashed{\partial}}_\infty$, which keeps the 
domains invariant by \cite[V, Theorem 4.3]{Kato5}. By the geometric Witt assumption 
the maximal and minimal domains of the Hodge Dirac operator on $\widetilde{D}_\infty$ coincide. Hence,
cf. \eqref{min-max-same}
\begin{equation}
\dom(\widetilde{D}_{\varepsilon,u}) = \dom (\widetilde{D}_\infty) 
\quad \left(= \dom_{\min}(\widetilde{D}_\infty) = \dom_{\max}(\widetilde{D}_\infty)\right).
\end{equation}
Over the cylinder $(-\infty, 0) \times \partial \widetilde{M}$
the perturbed operator $\widetilde{D}_{\varepsilon,u}$ is of the form
\begin{equation}
\widetilde{D}_{\varepsilon,u} 
= \widetilde{\sigma} \left( \frac{d}{dx} + \widetilde{B}(x)\right), 
\quad \widetilde{B}(x) := \widetilde{B} + \theta(x) \left( u - \widetilde{B} \, \Pi_\varepsilon (\widetilde{B}) \right).
\end{equation}
Thus, $\widetilde{B}(x)$ is a bounded perturbation of $\widetilde{B}$ for any $x\leq 0$.
By the same argument as above, the domains of $\widetilde{B}(x)$ and $\widetilde{B}$ coincide
\begin{equation}
\dom(\widetilde{B}(x)) = \dom (\widetilde{B}) 
\quad \left(= \dom_{\min}(\widetilde{B}) = \dom_{\max}(\widetilde{B})\right).
\end{equation}
Note that for $x\leq -1$, the parameters $\varepsilon>0$ and $u\neq 0$, the tangential operator $\widetilde{B}(x)$
admits a spectral gap around zero. We can now prove that $\dom(\widetilde{D}_{\varepsilon,u})$ is 
$\Gamma$-Fredholm by constructing its right-parametrix up to a $\Gamma$-compact remainder. 
We point out that our argument differs from the approach taken by Vaillant \cite[\S 6.1]{Va}.

\begin{thm}\label{D-Fredholm}
$\widetilde{D}_{\varepsilon,u}$ is $\Gamma$-Fredholm for $\varepsilon, u> 0$.
\end{thm}

\begin{proof}
We construct a right-parametrix for $\widetilde{D}_{\varepsilon,u}$ by a gluing argument from an interior and 
a boundary parametrix. Existence of a left-parametrix up to a $\Gamma$-compact remainder follows by repeating
the construction for $\widetilde{D}^*_{\varepsilon,u}$ verbatim, and taking adjoints. 
We proceed in three steps: construct the interior parametrix; construct the
boundary parametrix, using crucially the spectral gap around zero in $\widetilde{B}(x), x\leq -1$; glue these local parametrices 
together to a full right-parametrix to conclude the statement.

\subsubsection*{Step 1: Interior parametrix for $\widetilde{D}_{\varepsilon,u}$}
Consider a double of $\widetilde{M} \cup (\partial \widetilde{M} \times [-4,0])$,
which we denote by $\widetilde{M}_c$.  Since $\widetilde{B}(x)$ is independent
of $x$ for $x\in [-4,-1]$, the perturbed operator $\widetilde{D}_{\varepsilon,u}$
extends smoothly to an operator $\widetilde{D}^{c}_{\varepsilon,u}$
on the double $\widetilde{M}_c$. Similarly, $\widetilde{D}_\infty$ yields a self-adjoint operator
$\widetilde{D}^c$ on the double, which is simply the signature operator on $\widetilde{M}_c$ with 
$$
\widetilde{D}^c = \widetilde{D}^{c}_{0,0}\, , \quad \dom(\widetilde{D}^c) = 
\dom (\widetilde{D}^{c}_{\varepsilon,u})
\quad \left(= \dom_{\min}(\widetilde{D}^c) = \dom_{\max}(\widetilde{D}^c)\right).
$$
Since $\widetilde{D}^{c}_{\varepsilon,u}$ is self-adjoint, the inverse 
$(i+\widetilde{D}^{c}_{\varepsilon,u})^{-1}$ exists and we set 
\begin{equation}
Q_{\textup{int}}:= (i+\widetilde{D}^{c}_{\varepsilon,u})^{-1}: 
L^2\Omega^*(\widetilde{M}_c) \to \dom(\widetilde{D}^c).
\end{equation}

\subsubsection*{Step 2: Boundary parametrix for $\widetilde{D}_{\varepsilon,u}$}
Consider the cylinder $\partial \widetilde{M} \times \R$ and the operator 
$\widetilde{\sigma} \left( \partial_x + \widetilde{B}(- \, 1)\right)$ on the cylinder. 
Under the Fourier transform $\mathscr{F}$ along $\R$ the operator transforms
as follows 
$$
\mathscr{F}^{-1} \circ \widetilde{\sigma} \left( \frac{d}{dx} + \widetilde{B}(-\, 1)\right) \circ \mathscr{F}
= \widetilde{\sigma} \left( i \xi + \widetilde{B}(- \, 1)\right).
$$
Since by construction, $\widetilde{B}(- \, 1)$ admits a spectral gap around zero for
$\varepsilon, u > 0$, $i \xi + \widetilde{B}(- \, 1)$ is invertible for any $\xi \in \R$ 
including $\xi = 0$. Hence we can define the boundary parametrix by 
\begin{equation}
Q_{\textup{cyl}}:= \mathscr{F} \circ \left( \widetilde{\sigma} \left( i \xi + \widetilde{B}(- \, 1)\right)
\right)^{-1} \circ \mathscr{F}^{-1}.
\end{equation}
The $L^2\Omega^*(\partial \widetilde{M})$ operator norm of 
the inverse  $(i \xi + \widetilde{B}(- \, 1))^{-1}$ is $O(\|\xi\|)$ as $\|\xi\|\to \infty$. Hence we find
that the boundary parametrix acts as follows
\begin{equation}
Q_{\textup{cyl}}: L^2\Omega^*(\partial \widetilde{M} \times \R) \to 
H^1(\R) \widehat{\otimes} L^2\Omega^*(\partial \widetilde{M}) \cap L^2\Omega^*(\R) \widehat{\otimes} \dom(\widetilde{B}).
\end{equation}

\subsubsection*{Step 3: Full parametrix for $\widetilde{D}_{\varepsilon,u}$}
Consider cutoff functions $\phi, \psi, \chi \in C^\infty(\R)$
as illustrated in Figure \ref{fig:CutOff}. The functions $\chi$ and $(1-\psi)$
extend naturally to smooth functions on the double $\widetilde{M}_c$ being identically
$1$ on both copies on $\widetilde{M}$. The cutoff functions $\phi, \psi$ extend naturally to smooth functions on the cylinder 
$\partial \widetilde{M} \times \R$, being identically $1$ on $(-\infty, -3] \times \partial \widetilde{M}$.

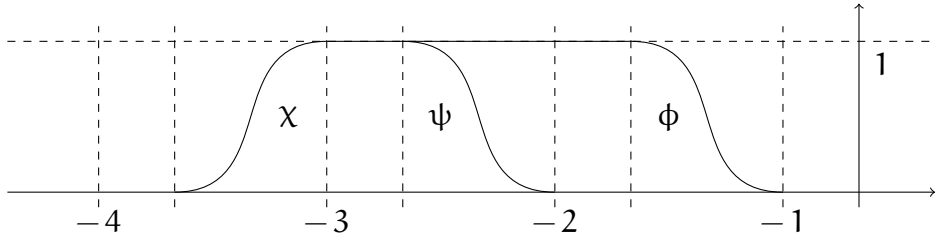
\begin{figure}[h]
\begin{center}

\begin{tikzpicture}[scale=1]
\draw[->] (-4.2,0) -- (8,0);
\draw[->] (7,-0.2) -- (7,2.5);
\draw[dashed] (-4.2,2) -- (8,2);

\draw (0,2) -- (4,2);
\draw (7.3,1.7) node {$1$};

\draw (1,2) .. controls (2.4,2) and (1.6,0) .. (3,0);
\draw[dashed] (1,-0.2) -- (1,2.2);
\draw[dashed] (3,-0.2) -- (3,2.2);
\draw (3,-0.4) node {$-\, 2$};

\draw (-3,-0.4) node {$-\, 4$};
\draw[dashed] (-2,-0.2) -- (-2,2.2);
\draw[dashed] (-3,-0.2) -- (-3,2.2);

\draw (4,2) .. controls (5.4,2) and (4.6,0) .. (6,0);
\draw[dashed] (4,-0.2) -- (4,2.2);
\draw[dashed] (6,-0.2) -- (6,2.2);
\draw (6,-0.4) node {$-\, 1$};

\draw (-2,0) .. controls (-0.6,0) and (-1.4,2) .. (0,2);
\draw[dashed] (0,-0.2) -- (0,2.2);
\draw (0,-0.4) node {$-\, 3$};

\draw (1.5,1) node {$\psi$};
\draw (4.5,1) node {$\phi$};
\draw (-0.5,1) node {$\chi$};

\end{tikzpicture}

\caption{The cutoff functions $\phi, \psi$ and $\chi$.}
\label{fig:CutOff}
\end{center}
\end{figure}

We define the full parametrix by 
\begin{align}\label{Q-def}
Q := \chi \cdot Q_{\textup{int}} \cdot (1-\psi) + \phi \cdot Q_{\textup{cyl}} \cdot \psi.
\end{align}
 We obtain by explicit computations ($c$ denotes the Clifford multiplication)
\begin{align*}
\widetilde{D}_{\varepsilon,u} \circ Q &=  c(d\chi) \, Q_{\textup{int}} \cdot (1-\psi) + c(d\phi) \, Q_{\textup{cyl}} \cdot \psi
 \\ &+ \chi \cdot \widetilde{D}^{c}_{\varepsilon,u} Q_{\textup{int}} \cdot (1-\psi) +
\phi \cdot \widetilde{\sigma} \left( \partial_x + \widetilde{B}(- \, 1)\right) Q_{\textup{cyl}} \cdot \psi
\\ &= \textup{Id} + c(d\chi) \, Q_{\textup{int}} \cdot (1-\psi) + c(d\phi) \, Q_{\textup{cyl}} \cdot \psi
- i \chi \cdot Q_{\textup{int}} \cdot (1-\psi).
 \end{align*}
 Writing $\dom_{\textup{comp}}(\widetilde{D}_\infty)$ for elements in the domain of 
 $\widetilde{D}_\infty$ with compact support in the cylindrical direction of $\widetilde{M}_\infty$,
 we conclude\footnote{Here, the right parametrix $Q$ strikingly produces a 
 remainder with values in compactly supported forms along the cylinder. This is however not the case
 for the left parametrix, cf. Remark \ref{compact-error}.}
 \begin{equation}
\widetilde{D}_{\varepsilon,u} \circ Q - \textup{Id}: L^2\Omega^*(\widetilde{M}_\infty) \rightarrow \dom_{\textup{comp}}(\widetilde{D}_\infty)
\hookrightarrow L^2\Omega^*(\widetilde{M}_\infty),
 \end{equation}
where  the second map is the natural inclusion $\widetilde{\iota}:\dom_{\textup{comp}}(\widetilde{D}_\infty)
\hookrightarrow L^2\Omega^*(\widetilde{M}_\infty)$ We now show that 
$\widetilde{\iota}$, and consequently also $\widetilde{D}_{\varepsilon,u} \circ Q - \textup{Id}$ as a map on $L^2\Omega^*(\widetilde{M}_\infty)$, 
is $\Gamma$-compact. Consider first an isomorphism of (right) $\Gamma$-modules
			\begin{equation*}
				\begin{split}
				L^2\Omega^*(\widetilde{M}_\infty) & \to L^2\Omega^*(M_\infty) \otimes L^2\Omega^*(\Gamma)\\
				\w & \mapsto \sum_{\gamma \in \Gamma} (\w \gamma)\vert_{\mathcal{F}} \otimes \gamma^{-1}
				\end{split}
			\end{equation*}
where $\mathcal{F}$ is the fundamental domain of the 
Galois-covering $\widetilde{M}_\infty \to M_\infty$ and $(\w \gamma)(q) = \w (q \gamma^{-1})$ 
for any $q\in \mathcal{F}$. Clearly $L^2\Omega^*(\mathcal{F}) = L^2\Omega^*(M_\infty)$, since $\mathcal{F} \subset M_\infty$ is dense.
$\Gamma$-compactness of $\widetilde{\iota}$ follows from the commutative diagramm:
			\begin{equation}\label{cpt-incl}
				\xymatrix@R=3mm@C=3mm{\dom_{\textup{comp}}(\widetilde{D}_\infty)
				\ar[rr] \ar@{^{(}->}[dd]_{\widetilde{\iota}} & & \dom_{\textup{comp}}(D_\infty)\otimes L^2\Omega^*(\Gamma) 
				\ar@{^{(}->}[dd]^{\iota \, \otimes \, Id} \\ & \circlearrowleft &\\
				L^2\Omega^*(\widetilde{M}_\infty) \ar[rr] & & L^2\Omega^*(M_\infty) \otimes L^2\Omega^*(\Gamma)}
			\end{equation}
where as before \emph{comp} refers to compact support in cylindrical direction, $\iota$ 
is the inclusion map, and the horizontal maps are isometries. Note that 
$\iota:  \dom_{\textup{comp}}(D_\infty) \hookrightarrow L^2\Omega^*(M_\infty)$ is a compact 
embedding, cf. \cite[Proposition 5.9]{ALMP1}, and hence $\widetilde{\iota}$ is $\Gamma$-compact. Thus $\widetilde{D}_{\varepsilon,u} \circ Q - \textup{Id}$
is $\Gamma$-compact as a map on $L^2\Omega^*(\widetilde{M}_\infty)$, and the same statement for a left-inverse follows by 
repeating the construction for $\widetilde{D}^*_{\varepsilon,u}$ and taking adjoints. 
This proves the statement.
\end{proof}

\begin{remark}\label{compact-error}
In the proof above we have viewed $Q$ as a right-parametrix, where the remainder
by construction mapped to $\dom_{\textup{comp}}(\widetilde{D}_\infty)$ and hence was
easily seen to be $\Gamma$-compact. We can also take $Q$ as a left parametrix for 
$\widetilde{D}_{\varepsilon,u}$ and compute similarly
\begin{align}\label{Q-left}
Q \circ \widetilde{D}_{\varepsilon,u} 
= \textup{Id} + \left(\chi Q_{\textup{int}} - \phi Q_{\textup{cyl}} \right) c(d\psi).
\end{align}
In that case, however, the remainder does not map to $\dom_{\textup{comp}}(\widetilde{D}_\infty)$
and thus its $\Gamma$-compactness cannot be established by a straightforward argument.
\end{remark}

\subsection{Finite propagation speed on $\widetilde{M}_\infty$}

Using the spectral theorem, we conclude that for any $\xi_0 \in \Omega^*_0(\widetilde{M}_\infty)$ with compact support, 
and any self-adjoint operator $T$ in $L^2\Omega^*(\widetilde{M}_\infty)$,
there exists a unique solution $\xi(t) := e^{itT} \xi_0$ of the wave equation $\partial_t \xi(t) 
-iT \xi(t) = 0$ with the initial condition $\xi(0) = \xi_0$. Here we set specifically
$$
T := \widetilde{D}_{\varepsilon,u} \oplus \widetilde{D}^*_{\varepsilon,u}, 
\quad \dom(T) := \dom_{\max}(T) = \dom_{\max}(\widetilde{\slashed{\partial}}_\infty).
$$

\begin{thm}\label{propagation}
Consider $C:=(a,b) \times \partial \widetilde{M} \subset \widetilde{M}_\infty$. for any $a < b < 0$.
We write $B(C,r):= (a-r,b+r)  \times \partial \widetilde{M} \subset \widetilde{M}_\infty$ for any $r<-b$.
Then the norm $\|\xi(t)\|_{L^2\Omega^*(B(C,r-t))}$ is decreasing in $t$. Thus, if $\supp \, \xi_0 
\subset (a,b) \times \partial \widetilde{M}$, then $\supp \, \xi(t) \subset (a-t,b+t) \times \partial \widetilde{M}$,
i.e. the propagation speed of the wave operator $e^{itT}$ is $\leq 1$ along the cylinder.
\end{thm}

\begin{proof}
The proof is classical, see \cite[Proposition 5.5]{Roe}. We just need some care when dealing with 
domains. We start with the following computation, where we write $\textup{dvol}$ for the volume form induced 
by the metric $\widetilde{g}_\infty$, $\langle \cdot , \cdot\rangle$ for the induced pointwise scalar product on the 
exterior algebra of the cotangent bundle, and $| \cdot |$ for the corresponding pointwise norm. Assume $t>0$ for simplicity.
\begin{align*}
\partial_t \|\xi(t)\|^2_{L^2\Omega^*(B(C,r-t))} = & \ \partial_t \! \! \! \! \int\limits_{B(C,r-t)}  | \, \xi(t)(p) \, |^{\, 2} \, \textup{dvol}(p)
= - \! \! \! \! \int\limits_{\partial B(C,r-t)} | \, \xi(t)(p) \, |^{\, 2} \, \textup{dvol}(p) 
\\ & + \! \! \! \! \int\limits_{B(C,r-t)} \langle \xi(t)(p), iT\xi(t)(p)\rangle + \langle iT \xi(t)(p),\xi(t)(p)\rangle \textup{dvol}(p).
\end{align*}
Since $\xi(t) \in \dom_{\max}(T)$ we conclude from self-adjointness of $B$
\begin{align*}
\partial_t \|\xi(t)\|^2_{L^2\Omega^*(B(C,r-t))} &= \int\limits_{B(C,r-t)} \partial_x 
\langle \xi(t)(x,q), \widetilde{\sigma} \oplus \widetilde{\sigma}^t \xi(t)(x,q)\rangle dx \, \textup{dvol}(q)
\\ &-\int\limits_{\partial B(C,r-t)}  | \, \xi(t)(q) \, |^{\, 2} \,  \textup{dvol}(q)
\\ &\leq \int\limits_{\partial B(C,r-t)} \left( \|\widetilde{\sigma} \oplus \widetilde{\sigma}^t \| - 1\right) | \, \xi(t)(q) \, |^{\, 2}   \textup{dvol}(q).
\end{align*}
Since $\|\widetilde{\sigma} \oplus \widetilde{\sigma}^t \| = 1$, we conclude that the norm
$\|\xi(t)\|_{L^2\Omega^*(B(C,r-t))}$ is indeed decreasing in $t$. In particular, if $\xi_0 \equiv 0$ on
a cylinder $(a,b) \times \partial \widetilde{M}$, then $\xi(t)\equiv 0$ on $(a+t,b-t) \times \partial \widetilde{M}$
for any $t\geq 0$. The statement on supports now follows. 
\end{proof} \\[-7mm]

\subsection{Sobolev embedding theorem and a Garding inequality on $\widetilde{M}_\infty$}

In order to pass from the finite propagation speed to uniform heat
kernel estimates, we need a Sobolev embedding theorem along with the 
Garding inequality in the setting of $\widetilde{M}_\infty$,
which is notably not a manifold with bounded geometry because
of the structure of the iterated cone-edge metric near the singularity, so that the arguments of 
\cite[\S 3]{Va} do not apply directly. \medskip

We first define weighted Sobolev spaces
on $\widetilde{M}_\infty$ with values in any vector bundle $E$ associated to $T\widetilde{M}_\infty$,
and study their embedding and multiplication properties. The vector bundle 
$E$ shall be the exterior algebra of $T^*\widetilde{M}_\infty$, but the statement of 
the Sobolev embedding theorem below holds for other associated vector bundles as well. \medskip

The weights of the Sobolev spaces below are defined in terms of $\rho: \widetilde{M}_\infty \to (0,\infty)$,
defined as the distance to the singular stratum $S= \{Y_\sigma\}_{\sigma}$ of $\widetilde{M}_\infty$. 
Recall that any point of $Y_\sigma$ has a tubular neighborhood $\cU_\sigma$, which is the total space of a fibration 
$\phi_\sigma: \cU_\sigma \to \phi_\sigma (\cU_\sigma) \subseteq Y_\sigma$ 
with fibers given by cones $\overline{\mathscr{C}(F_\sigma)}$ with link $F_\sigma$ being 
compact smoothly stratified space of lower depth. Denote by $x_\sigma$ the radial function of the
cone $\overline{\mathscr{C}(F_\sigma)}$. Then $\rho$ can be given locally near the strata by 
a locally finite product $\rho = \Pi_{Y_\sigma \in S} \, x_\sigma$. We set for any $d = (d_\sigma)_{\sigma} \subset \R$
\begin{align}\label{rho-d}
\rho^d := \prod\limits_{Y_\sigma \in S} x_\sigma^{d_\sigma},
\end{align}
extended to a positive smooth function in the interior away from strata. 
The next definition follows Pacini \cite[(5.1), (5.2)]{Pacini}.

\begin{defn}\label{product-spaces}
Consider $s \in \N_0$ and $\delta \in \R$. 
Let $\nabla$ denote the Levi Civita connection on $E$, induced by the $\Gamma$-equivariant 
iterated cone-edge metric $\widetilde{g}_\infty$.
We consider the space $L^2\Omega^*(\widetilde{M}_\infty)$ of square-integrable sections of $E$ with respect to the volume form of $\widetilde{g}_\infty$ and
the pointwise norm $| \cdot |$ on fibres of $E$ induced by $\widetilde{g}_\infty$. We introduce the 
multi-index $m= (\dim \mathscr{C}(F_\alpha))_\alpha \subset \N$
and define \medskip

\begin{enumerate}
\item We define the Sobolev space $H^s_\delta (E)$ as 
the closure of compactly supported smooth sections $C^\infty_0(E)$ under
\begin{align*}
\left\| \w \right\|_{H^s_\delta}= \sum_{k=0}^s \| \rho^{k-\delta-\frac{m}{2}} 
| \, \nabla^k \w \, | \, \|_{L^2\Omega^*(\widetilde{M}_\infty)}.
\end{align*}
Note that $L^2\Omega^*(\widetilde{M}_\infty) = H^0_{-\frac{m}{2}}(E)$ by construction. 
\item We define the Banach space $C^{s}_\gamma(E)$ 
as the closure of  $C^{\infty}_{0}(E)$ under 
\begin{align*}
\left\| \w \right\|_{C^s_\delta} = \sum_{k=0}^s 
\sup_{q \in \widetilde{M}_\infty} \left(\rho^{k-\delta}| \, \nabla^k \w \, | \right) (q).
\end{align*}
\end{enumerate}
\end{defn}

Notice that in the definition of weighted Sobolev spaces $H^s_\delta (E)$ above, 
there is a shift by $-m/2$, when compared to the Sobolev spaces in e.g.
in \cite{ALMP1}. \medskip

The function $\rho$ satisfies all the assumptions of \cite[Theorem 4.7]{Pacini}, 
cf. the argument of \cite[Example 4.9]{Pacini} on a model cone, that is a local argument and hence carries over to iterated
cone-edge singularities. As asserted by Pacini \cite[Corollary 6.8, Remark 6.9]{Pacini} we find the following 
analogue of the standard properties in the stratified non-compact setting.

\begin{thm}\label{embedding-theorem} 
The spaces in Definition \ref{product-spaces} admit the following properties. 
\begin{enumerate}
     \item For $\beta \lvertneqq \delta$ we have 
	$C^{s}_{\delta}(E) \subset H^s_{\beta}(E)$.
	\item For $N > \dim M/2$ and $\beta \leq \delta$ we have
	$H^{s+N}_{\delta}(E) \subset C^{s}_{\beta}(E)$.
	\end{enumerate}
\end{thm}

As for a Garding inequality, we do not assert such a statement for general elliptic operators. We 
rather consider the $\Gamma$-Fredholm operator $\widetilde{D}_{\varepsilon,u}$.

\begin{prop}\label{Garding}
Consider any multi-index $\A = (\alpha_\sigma)_{\sigma} \subset \R^+$. Set 
$D_{\A} := \rho^{\A} \circ \widetilde{D}_{\varepsilon,u} \circ \rho^{-\A}$
and consider the multi-index $\beta := -\frac{m}{2} + \A$.
Then for any $s\in \N_0$ and $\delta \in (0,1)$
there exists a constant $c>0$ such that
$$
\forall \w \in H^{s+1}_{\beta+\delta}(E): 
\quad \| \w \|_{H^{s+1}_{\beta}(E)} \leq \| \w \|_{H^{s+1}_{\beta+\delta}(E)} \leq c \left(
\| D_{\A} \w \|_{H^{s}_{\beta}(E)} + \| \w \|_{H^{s}_{\beta}(E)}
\right).
$$
\end{prop}

\begin{proof}
The operator $\widetilde{D}_{\varepsilon,u}$ is $\Gamma$-Fredholm
with a parametrix $Q$ defined in \eqref{Q-def}.
We can view $Q$ as a left-parametrix with remainder $R$ obtained in \eqref{Q-left}.
Both $Q$ and $R$ are continuous operators
\begin{align}
Q, R: L^2\Omega^*(\widetilde{M}_\infty) = H^0_{-\frac{m}{2}}(E) \to \dom (\widetilde{D}_\infty).
\end{align}
By employing the arguments in \cite{ALMP2} and the well-known
correspondence between operators twisted by the Mishchenko von Neumann bundle and
operators on the covering, see   \cite[Appendix E]{Piazza-Schick}, 
one can prove that 
the domain of $\widetilde{D}_\infty$ is included in the intersection of $H^1_{-\frac{m}{2}+\delta}(E)$ for any
$\delta \in (0,1)$.
We conclude that
\begin{align}
Q, R: H^s_{-\frac{m}{2}}(E) \to H^{s+1}_{-\frac{m}{2}+\delta}(E).
\end{align}
are continuous for $s=0$. 
Define
\begin{equation}\label{QR-mapping-0}
Q_{\A} := \rho^{\A} \circ Q \circ  \rho^{-\A} \;\;\;\text{and}\;\;\;
R_{\A} := \rho^{\A} \circ R \circ \rho^{-\A} 
\end{equation}
Then by combining  \cite{ALMP2} and  \cite[Theorems 3.7 and 4.3]{AGR17} we obtain that 
\begin{equation}\label{QR-mapping}
\begin{split}
Q_{\A} : H^s_{\beta}(E) \to H^{s+1}_{\beta+\delta}(E), \\
R_{\A}: H^s_{\beta}(E) \to H^{s+1}_{\beta+\delta}(E).
\end{split}
\end{equation}
continuously.
By construction, $Q_{\A} \circ D_{\A} = \textup{Id} + R_{\A}$ and hence we compute for any 
$\w \in H^{s}_{\beta+\delta}(E)$, using \eqref{QR-mapping} in the second inequality
\begin{align*}
\| \w \|_{H^{s+1}_{\beta+\delta}(E)} & \leq \| Q_{\A} \circ D_{\A} \w \|_{H^{s+1}_{\beta+\delta}(E)}
+ \| R_{\A} \w \|_{H^{s+1}_{\beta+\delta}(E)} \\ & \leq  c \left(
\| D_{\A} \w \|_{H^{s}_{\beta}(E)} + \| \w \|_{H^{s}_{\beta}(E)}\right).
\end{align*}
\end{proof}

\subsection{Uniform off-diagonal large time heat kernel estimates on $\widetilde{M}_\infty$}

We can now establish uniform heat kernel estimates. Consider the reference operator 
\begin{equation}\begin{split}
\widetilde{S}_{\varepsilon,u} 
= \widetilde{\sigma} \left( \frac{d}{dx} \otimes \textup{Id}_{L^2\Omega^*(\partial \widetilde{M})}
+ \textup{Id}_{L^2\Omega^*(\R)} \otimes \widetilde{B}(-1) \right) \equiv
\widetilde{\sigma} \left( \frac{d}{dx} + \widetilde{B}(-1)\right), 
\end{split}\end{equation}
in $L^2 (\R \times \partial \widetilde{M})$. Its domain is given by
$\dom(\widetilde{S}_{\varepsilon,u} ) := (H^1(\R) \widehat{\otimes} L^2\Omega^*(\partial \widetilde{M}))
\cap (L^2\Omega^*(\R)\widehat{\otimes} \dom (\widetilde{B}))$, which is a Hilbert space with the inner product
given by the sum of the inner products on $H^1(\R) \widehat{\otimes} L^2\Omega^*(\partial \widetilde{M})$
and on $L^2\Omega^*(\R)\widehat{\otimes} \dom (\widetilde{B})$.
Our main result of this subsection is 
an analogue of \cite[Proposition 6.2]{Va}.

\begin{thm}\label{off-diagonal}
Fix any $p=(x,q), p'=(x',q') \in (-\infty, 0) \times \partial \widetilde{M}$. Consider the multi-index 
$m= (\dim \mathscr{C}(F_\alpha))_\alpha \subset \N$, and any $t>0$ and $\delta\in (0,1)$.
We write $C(k,\varepsilon,u)>0$ for constants depending only on the data in the brackets.
\begin{enumerate}
\item Consider $r_1>0$ sufficiently small and assume that $|x-x'|>2r$. Then
\begin{align*}
\left|  \left(\widetilde{D}_{\varepsilon,u}^k e^{-t\widetilde{D}_{\varepsilon,u}^2} \right) (p,p')\right| 
\leq C(k,\varepsilon, u) \left(\rho(p) \rho(p')\right)^{-\frac{m}{2}+\delta}
\\ \times\exp \left(- \frac{(|x-x'|-r_1)^2}{6t}\right).
\end{align*}
\item Consider $r_2>0$ sufficiently small and assume that $x,x' < -2r_2-1$. Then
\begin{align*}
\left|  \left(\widetilde{D}_{\varepsilon,u}^k e^{-t\widetilde{D}_{\varepsilon,u}^2} - 
\widetilde{S}_{\varepsilon,u}^k e^{-t\widetilde{S}_{\varepsilon,u}^2} \right)(p,p')\right| 
\leq C(k,\varepsilon,u) \left(\rho(p) \rho(p')\right)^{-\frac{m}{2}+\delta} \\ 
\times \exp \left(- \frac{(\min \, \{\, |x|,|x'|\, \}-r_2-1)^2}{6t}\right).
\end{align*}
\end{enumerate}
Here, $\rho^d$ has been defined in \eqref{rho-d} for any multi-index $d$.
\end{thm}

\begin{proof} We denote all positive uniform constants by $C$.
Consider a non-negative function $\phi_x \in C^\infty_0(-\infty, 0)$ such that $\phi_x$ is identically $1$ in an open neighborhood of $x<0$.
It lifts to a smooth function on the cylinder $(-\infty, 0) \times \partial \widetilde{M}$ and extends trivially
to $\widetilde{M}_\infty$. We consider $f(x):=x^k e^{-tx^2}$ and the corresponding operator
$f(\widetilde{D}_{\varepsilon,u})$, defined by spectral calculus. Recall that $m= (\dim \mathscr{C}(F_\alpha))_\alpha \subset \N$.

\subsubsection*{Step 1: Uniform estimate in the second argument}
We set $\alpha = \frac{m}{2} - \delta$ and compute for any $N> \frac{\dim M}{2}$ using 
the Sobolev embedding of Theorem \ref{embedding-theorem} in the second inequality,
and the Garding inequality of Proposition \ref{Garding}, with $\alpha = \frac{m}{2}-\delta$, in the third inequality.  
\begin{align*}
| \rho^{\A}(p) \rho^{\A}(p') & f(\widetilde{D}_{\varepsilon,u})(p,p')| \leq
\| \rho^{\A}(p) \phi_x(p) \left(  \rho^{\A} \phi_{x'} f(\widetilde{D}_{\varepsilon,u})\right) (p, \cdot)\|_{C^0_0(E)}
\\ &\leq C \ \| \rho^{\A}(p) \phi_x(p) \left(  \rho^{\A} \phi_{x'} f(\widetilde{D}_{\varepsilon,u})\right) (p, \cdot) \|_{H^N_0(E)} 
\\ &\leq C \sum_{j=0}^N \| \rho^{\A}(p) \phi_x(p) \left( D^j_{\A} \rho^{\A} \phi_{x'} f(\widetilde{D}_{\varepsilon,u})\right) (p, \cdot) \|_{H^0_{-\delta}(E)}
\\ &= C \sum_{j=0}^N \| \rho^{\A}(p) \phi_x(p) \left( \rho^{\delta -\frac{m}{2}} D^j_{\A} \rho^{\A} \phi_{x'} f(\widetilde{D}_{\varepsilon,u})\right) (p, \cdot) \|_{L^2\Omega^*(\widetilde{M}_\infty)},
\end{align*}
where we denoted all positive uniform constants by $C$, and identified the functions $\phi_x, \phi_{x'}$ and $\rho^{\A}$
with operators acting by multiplication with the respective function.
Noting that $D_{\A} := \rho^{\A} \circ \widetilde{D}_{\varepsilon,u} \circ \rho^{-\A}$,
we compute further
\begin{equation}\label{estimate1}
\begin{split}
| \rho^{\A}(p) \rho^{\A}(p') & f(\widetilde{D}_{\varepsilon,u})(p,p')| 
\\ &\leq C \sum_{j=0}^N \| \rho^{\A}(p) \phi_x(p) \left( \widetilde{D}^j_{\varepsilon,u} \phi_{x'} f(\widetilde{D}_{\varepsilon,u})\right) (p, \cdot) \|_{L^2\Omega^*(\widetilde{M}_\infty)}
\\ &=: C \sum_{j=0}^N \| \xi_j (p,\cdot) \|_{L^2\Omega^*(\widetilde{M}_\infty)}.
\end{split}
\end{equation}
Note that by construction $f(\widetilde{D}_{\varepsilon,u}) (p, \cdot)$ and $\xi_j(p,\cdot)$ lie in the domain of $\widetilde{D}^L_{\varepsilon,u}$ 
for any $L \in \N$ and hence we obtain by self-adjointness
\begin{align*}
\|\xi_j(p,\cdot)\|^2_{L^2\Omega^*(\widetilde{M}_\infty)} &= \int_{\widetilde{M}_\infty} 
\rho^{\A}(p) \phi_x(p) \left(  \widetilde{D}^j_{\varepsilon,u} \phi_{x'} f(\widetilde{D}_{\varepsilon,u})\right) (p, p') \xi_j(p,p') \textup{dvol}(p')
\\ &= \int_{\widetilde{M}_\infty} 
\rho^{\A}(p) \phi_x(p) \left( \phi_{x'} f(\widetilde{D}_{\varepsilon,u})\right) (p, p') \, \widetilde{D}^j_{\varepsilon,u} \xi_j(p,p') \textup{dvol}(p')
\\ &= \left| \left( \rho^{\A} \phi_x f(\widetilde{D}_{\varepsilon,u}) \right) \left[ \phi_{x'} \widetilde{D}^j_{\varepsilon,u} \xi_j(p,\cdot) \right] (x) \right|.
\end{align*}
Below we use $\widetilde{D}_{\varepsilon,u} \phi_{x'} = c(d\phi_{x'}) + \phi_{x'} \widetilde{D}_{\varepsilon,u}$, 
where $c(d \phi_{x'})$ is the Clifford multiplication by $d \phi_{x'}$, as estimate
\begin{equation}\label{Clifford1}
\begin{split}
\|\xi_j(p,\cdot)\|^2_{L^2\Omega^*(\widetilde{M}_\infty)} &\leq 
\left| \left(  \rho^{\A} \phi_x f(\widetilde{D}_{\varepsilon,u}) \right) \left[ \widetilde{D}_{\varepsilon,u} \phi_{x'} \widetilde{D}^{j-1}_{\varepsilon,u} \xi_j(p,\cdot) \right] (x) \right|
\\ &+ \left| \left(  \rho^{\A} \phi_x f(\widetilde{D}_{\varepsilon,u}) \right) \left[ c(d \phi_{x'}) \widetilde{D}^{j-1}_{\varepsilon,u} \xi_j(p,\cdot) \right] (x) \right|. 
\end{split}
\end{equation}
Replacing $\phi_{x'}$ by a smooth 
compactly supported non-negative function in $C^\infty_0(-\infty, 0)$ which is identically one on 
$\supp \, d\psi_{x'} \cup \supp \, \psi_{x'}$, and denoting this new function by $\phi_{x'}$ again to simplify notation, we arrive iteratively at the following 
intermediate estimate
\begin{equation}\label{Clifford2}
\begin{split}
\|\xi_j(p,\cdot)\|^2_{L^2\Omega^*(\widetilde{M}_\infty)} &\leq 
C \sum_{\ell = 0}^j \left| \left( \rho^{\A} \phi_x f(\widetilde{D}_{\varepsilon,u}) \right) \left[ \widetilde{D}^{\ell}_{\varepsilon,u} \phi_{x'} \xi_j(p,\cdot) \right] (x) \right| 
\\ & = C \sum_{\ell = 0}^j \left| \left( \rho^{\A} \phi_x \widetilde{D}^{\ell}_{\varepsilon,u} f(\widetilde{D}_{\varepsilon,u})  \phi_{x'} \xi_j(p,\cdot)\right) (x) \right|. 
\end{split}
\end{equation}
\subsubsection*{Step 2: Uniform estimate in the first argument}
We compute, applying Sobolev embedding and Garding inequality again
\begin{align*}
\|\xi_j(p,\cdot)\|^2_{L^2\Omega^*(\widetilde{M}_\infty)} &
\leq  C  \sum_{\ell = 0}^j\|  \rho^{\A} \phi_x \widetilde{D}^{\ell}_{\varepsilon,u} f(\widetilde{D}_{\varepsilon,u})  \phi_{x'} \xi_j(p,\cdot) \|_{H^0_0(E)}
\\ & \leq C \sum_{k=0}^N \sum_{\ell = 0}^j \| D^{k}_{\A}  \rho^{\A} \phi_x \widetilde{D}^{\ell}_{\varepsilon,u} f(\widetilde{D}_{\varepsilon,u})  \phi_{x'} \xi_j(p,\cdot) \|_{H^0_{-\delta}(E)}
\\ & \leq C \sum_{k=0}^N \sum_{\ell = 0}^j \| \widetilde{D}^k_{\varepsilon,u} \phi_x \widetilde{D}^{\ell}_{\varepsilon,u} f(\widetilde{D}_{\varepsilon,u})  
\phi_{x'}  \|_{L^2 \to L^2} \cdot \|\xi_j(p,\cdot)\|_{L^2\Omega^*(\widetilde{M}_\infty)}
\\ & \leq C \sum_{\ell = 0}^{N+j} \| \phi_x \widetilde{D}^{\ell}_{\varepsilon,u} f(\widetilde{D}_{\varepsilon,u})  
\phi_{x'}  \|_{L^2 \to L^2} \cdot \|\xi_j(p,\cdot)\|_{L^2\Omega^*(\widetilde{M}_\infty)},
\end{align*}
where in the last inequality we have argued as in \eqref{Clifford1} and \eqref{Clifford2}. Dividing both sides of the inequality by
$\|\xi_j\|_{L^2\Omega^*(\widetilde{M}_\infty)}$, we arrive at 
\begin{align*}
\|\xi_j(p,\cdot)\|_{L^2\Omega^*(\widetilde{M}_\infty)}  \leq C \sum_{\ell = 0}^{N+j} \| \phi_x \widetilde{D}^{\ell}_{\varepsilon,u} f(\widetilde{D}_{\varepsilon,u})  
\phi_{x'}  \|_{L^2 \to L^2}.
\end{align*}
Plugging this into \eqref{estimate1} we find
\begin{equation*}
\begin{split}
| \rho^{\A}(p) \rho^{\A}(p')  f(\widetilde{D}_{\varepsilon,u})(p,p')| 
&\leq C \sum_{j=0}^N \| \xi_j (p,\cdot) \|_{L^2\Omega^*(\widetilde{M}_\infty)} \\
&\leq C \sum_{\ell = 0}^{2N} \| \phi_x \widetilde{D}^{\ell}_{\varepsilon,u} f(\widetilde{D}_{\varepsilon,u})  
\phi_{x'}  \|_{L^2 \to L^2}.
\end{split}
\end{equation*}
\subsubsection*{Step 3: Estimate using finite propagation speed}
The operator norm above can now be studied via the spectral representation
\begin{equation}
\widetilde{D}^{\ell}_{\varepsilon,u} f(\widetilde{D}_{\varepsilon,u}) 
= \frac{1}{2\pi} \int_\R \left( (i\partial_s)^\ell \, \widehat{f}(s) \right) e^{is \widetilde{D}_{\varepsilon,u}} ds,
\end{equation}
where $\widehat{f}$ denotes the Fourier transform of $f$, and the representaton is 
in fact valid for any Schwartz function $f$. Using the finite propagation speed result in 
Theorem \ref{propagation}, we can now continue our estimate as follows.
\begin{equation}\label{estimate-final}
\begin{split}
| \rho^{\A}(p) \rho^{\A}(p') & f(\widetilde{D}_{\varepsilon,u})(p,p')| 
\\ &\leq C \sum_{\ell = 0}^{2N} \int_\R \left| \widehat{f}^{(\ell)}(s) \right| \cdot
\|  \phi_x e^{is \widetilde{D}_{\varepsilon,u}} \phi_{x'}   \|_{L^2 \to L^2} ds
\\ &\leq C \sum_{\ell = 0}^{2N} \int\limits_{\R \backslash I(x,x')} \left| \widehat{f}^{(\ell)}(s) \right| ds,
\end{split}
\end{equation}
where $I(x,x') := (-d,d)$ with $d$ being the distance between the supports of 
$\phi_x$ and $\phi_{x'}$. Choosing the supports of $\phi_x$ and $\phi_{x'}$ are sufficiently
small, $d=|x-x'|-r$ for a sufficiently small number $r>0$. From here on, the statement 
(1) of the Theorem follows from the explicit 
representation of $\widehat{f}^{(\ell)}(s)$ in terms of Hermite polynomials, and the computation 
of \cite[(3.6)]{Va}. \medskip

For the statement (2) of the Theorem, note that $e^{is \widetilde{D}_{\varepsilon,u}} \xi = e^{is \widetilde{S}_{\varepsilon,u}} \xi$
for $\supp \xi \subset (-\infty, -1) \times \partial \widetilde{M}$ and $|s|$ smaller than the distance of support of $\xi$ to
$\{-1\} \times \partial \widetilde{M}$, by uniqueness of solutions and finite propagation speed. 
Thus we obtain with similar arguments as before 
\begin{equation}\label{estimate-final2}
\begin{split}
| \rho^{\A}(p) \rho^{\A}(p') & \left( f(\widetilde{D}_{\varepsilon,u}) - f(\widetilde{S}_{\varepsilon,u})\right) (p,p')| 
\\ &\leq C \sum_{\ell = 0}^{2N} \int_\R \left| \widehat{f}^{(\ell)}(s) \right| \cdot
\|  \phi_x \left( e^{is \widetilde{D}_{\varepsilon,u}} - e^{is \widetilde{S}_{\varepsilon,u}} \right) \phi_{x'}   \|_{L^2 \to L^2} ds
\\ &\leq C \sum_{\ell = 0}^{2N} \int\limits_{\R \backslash J(x,x')} \left| \widehat{f}^{(\ell)}(s) \right| ds,
\end{split}
\end{equation}
where $J(x,x') := (-d',d')$ with $d'$ being the minimal distance between the supports of 
$\phi_x$ and $\phi_{x'}$ to $\{-1\} \times \partial \widetilde{M}$. From here on, the 
statement (2) of the Theorem follows as in \cite[(3.7)]{Va}.
\end{proof}

\subsection{Uniform on-diagonal large time heat kernel estimates on $\widetilde{M}_\infty$}

The estimates of Theorem \ref{off-diagonal} are concerned with the large time
behaviour of the heat kernel away from the diagonal. We complement the 
subsection with a result on a uniform large time heat kernel asymptotics at the diagonal. 
\medskip

Consider a fundamental domain $\mathcal{F}_\infty$ of the Galois covering $\widetilde{M}_\infty$,
which can be identified with $M_\infty$ up to a subset of measure zero.
Consider a smooth cutoff function $\phi \in C^\infty (M_\infty)$
smooth up to the singular strata, such that $\phi \equiv 0$ on $(-\infty, -N-1) \times \partial M \subset M_\infty$
for some $N\in \N$ and $\phi \equiv 1$ on the complement of $(-\infty, -N) \times \partial M$. We lift $\phi$ 
to a smooth function on $\widetilde{M}_\infty$ with $\supp \, \phi \subset \mathcal{F}_\infty$. Furthermore, let
$\pi: \mathcal{F}_\infty \to \mathcal{F}_\infty \times \mathcal{F}_\infty, p\mapsto (p,p)$ be the inclusion into the diagonal.
Let $\pi^* (E \boxtimes E)$ be the corresponding pullback bundle. 

\begin{thm}\label{on-diagonal}
Consider the orthogonal kernel projection $P_{\ker \widetilde{D}_{\varepsilon,u}}$ of $\widetilde{D}_{\varepsilon,u}$. 
Then for any $\delta \in (0,1)$ and $p\in \mathcal{F}_\infty$, $\phi e^{-t\widetilde{D}_{\varepsilon,u}^2}(p,p)$ converges to 
$\phi P_{\ker \widetilde{D}_{\varepsilon,u}}(p,p)$ uniformly in $C^0_{-m+2\delta}(\mathcal{F}_\infty, \pi^* (E \boxtimes E))$, as $t\to \infty$.
\end{thm}

\begin{proof}
Repeating the arguments of \eqref{estimate-final}, we find for $f(x):= e^{-tx^2}$ 
\begin{equation*}
\begin{split}
| \rho^{m-2\delta}(p)e^{-t\widetilde{D}_{\varepsilon,u}^2} (p,p)| 
\leq C \sum_{\ell = 0}^{2N} \int\limits_{\R} \left| \widehat{f}^{(\ell)}(s) \right| ds.
\end{split}
\end{equation*}
The Fourier transform $\widehat{f}$ and its derivatives are computed explicitly by
$$
\widehat{f}^{(\ell)}(s) = \frac{C_\ell}{t^{(\ell+1)/2}} H_\ell \left(\frac{s}{\sqrt{4t}}\right)
e^{-\frac{s^2}{4t}},
$$
where $H_\ell$ is the $\ell$-th Hermite polynomial and $C_\ell$ is a universal constant, 
depending on $\ell$. Consequently we find after substitution
\begin{equation*}
\begin{split}
| \rho^{m-2\delta}(p) e^{-t\widetilde{D}_{\varepsilon,u}^2} (p,p)| 
\leq C \sum_{\ell = 0}^{2N} t^{-\ell / 2} \int\limits_{\R} \left| H_\ell (s) \right| e^{-s^2} ds
\leq C.
\end{split}
\end{equation*}
This shows that $(\phi e^{-t\widetilde{D}_{\varepsilon,u}^2}(p,p))_{t\in \R^+}$ is uniformly bounded in 
$C^0_{-m+2\delta}(\mathcal{F}_\infty, \pi^* (E \boxtimes E))$ as $t\to \infty$. Thus it admits a convergent subsequence $(\phi e^{-t_j\widetilde{D}_{\varepsilon,u}^2}(p,p))_{j\in \N}$
with $t_j\to \infty$ as $j\to \infty$. The limit of that subsequence must be $\phi P_{\ker \widetilde{D}_{\varepsilon,u}}(p,p)$,
since $e^{-t\widetilde{D}_{\varepsilon,u}^2}(p,p)$ converges to $P_{\ker \widetilde{D}_{\varepsilon,u}}(p,p)$ on compact subsets as $t\to \infty$, by
the proof of \cite[Proposition 15.11]{Roe}. The statement now follows from the general topological fact that 
if any subsequence of $(\phi e^{-t\widetilde{D}_{\varepsilon,u}^2}(p,p))_{t\in \R^+}$ admits a convergent subsequence with 
the same limit, then $\phi e^{-t\widetilde{D}_{\varepsilon,u}^2}(p,p)$ must converge to that limit as $t\to \infty$.
\end{proof}

\subsection{Proof of an index theorem on $\widetilde{M}_\infty$}

Consider for any $N\in \N$ a smooth $\Gamma$-invariant cutoff function $\phi_N \in C^\infty (\widetilde{M}_\infty)$, 
smooth up to the singular strata, such that $\phi_N \equiv 0$ on $(-\infty, -N-1) \times \partial \widetilde{M}$
and $\phi_N \equiv 1$ on the complement of $(-\infty, -N) \times \partial \widetilde{M}$.

\begin{prop}\label{eP}
The operators $\phi_N \, e^{-t\widetilde{D}_{\varepsilon,u}^2} \phi_N$ and 
$\phi_N \, \widetilde{D}_{\varepsilon,u}^2 e^{-t\widetilde{D}_{\varepsilon,u}^2} \phi_N$
are $\Gamma$-trace class and, taking the $\Gamma$-super traces, we find 
as $t\to \infty$
\begin{align*}
\textup{Tr}_\Gamma \left( \phi_N \, e^{-t\widetilde{D}_{\varepsilon,u}^2} \phi_N \right) \to 
\textup{Tr}_\Gamma \left( \phi_N \, P_{\ker \widetilde{D}_{\varepsilon,u}} \phi_N \right).
\end{align*}
\end{prop}

\begin{proof}
The microlocal heat-kernel construction of \cite{AGR17} extends to $e^{-t\widetilde{D}_{0,u}^2}$, 
where $u>0$ simply shifts the tangential operator by a constant. Hence, 
$\phi_N e^{-t\widetilde{D}_{0,u}^2}$ and $\phi_N \widetilde{D}_{0,u} e^{-t\widetilde{D}_{0,u}^2}$ are 
$\Gamma$-Hilbert Schmidt. Thus $\phi_N \, e^{-t\widetilde{D}_{0,u}^2} \phi_N$ 
and $\phi_N \, \widetilde{D}_{0,u}^2 e^{-t\widetilde{D}_{0,u}^2} \phi_N$ are $\Gamma$-trace class.
In order to pass to $\varepsilon >0$, we express the heat-kernel for $\widetilde{D}_{\varepsilon,u}^2$ 
as follows. By construction
\begin{equation*}
\begin{split}
\widetilde{D}_{\varepsilon,u} & = \widetilde{\sigma} \left( \frac{d}{dx} + \widetilde{B}(x)\right) 
= \widetilde{D}_{0,u} - \widetilde{\sigma} \, \theta(x) \widetilde{B} \, \Pi_\varepsilon (\widetilde{B}), \\
\widetilde{D}_{\varepsilon,u}^2 & = \widetilde{D}_{0,u}^2 - \left(\theta(x) \widetilde{B} \, \Pi_\varepsilon (\widetilde{B})\right)^2
 - c(d\theta)  \, \widetilde{\sigma} \, \widetilde{B} \, \Pi_\varepsilon (\widetilde{B})  \\& + 2 \, \theta(x) \, \, \widetilde{B} \, \Pi_\varepsilon (\widetilde{B}) 
  \left(\widetilde{B} + \theta(x) u\right)  =: \widetilde{D}_{0,u}^2 + R.
\end{split}
\end{equation*}
Now the heat-kernels for $\widetilde{D}_{\varepsilon,u}^2$ and $\widetilde{D}_{0,u}^2$ are related by a Volterra series
\begin{equation*}
\begin{split}
e^{-t\widetilde{D}_{\epsilon,u}^2}  = e^{-t \widetilde{D}_{0,u}^2} + \sum_{\kappa = 1}^{\infty} \int_{\Delta^\kappa} d \sigma_1 ... d\sigma_\kappa \, 
 e^{-\sigma_1 t \, \widetilde{D}_{0,u}^2} \circ (tR) \circ e^{-\sigma_2 t \, \widetilde{D}_{0,u}^2} \cdots (tR) \circ e^{-\sigma_\kappa t \, \widetilde{D}_{0,u}^2},
\end{split}
\end{equation*}
where $\Delta^\kappa := \{(\sigma_1,...,\sigma_\kappa) \in \R_+^\kappa \, \mid \, \sigma_1 + ... + \sigma_\kappa = 1 \}$.
Using again the microlocal heat-kernel construction of \cite{AGR17}, we conclude that the Schwartz kernels of $e^{-t \widetilde{D}_{\varepsilon,u}^2}$ 
and $ \widetilde{D}_{\varepsilon,u} e^{-t \widetilde{D}_{\varepsilon,u}^2}$ are still locally $L^2$-integrable up to the strata and 
$\phi_N e^{-t\widetilde{D}_{\varepsilon,u}^2}, \phi_N \widetilde{D}_{\varepsilon,u} e^{-t\widetilde{D}_{\varepsilon,u}^2}$ are 
$\Gamma$-Hilbert Schmidt. Hence the operators $\phi_N \, e^{-t\widetilde{D}_{\varepsilon,u}^2} \phi_N$ and 
$\phi_N \, \widetilde{D}_{\varepsilon,u}^2 e^{-t\widetilde{D}_{\varepsilon,u}^2} \phi_N$
are $\Gamma$-trace class, as claimed. Now the statement on the convergence of $\Gamma$-super traces follows 
Theorem \ref{on-diagonal} and from the fact that 
$$
\phi_N|_{\mathcal{F}_\infty} \cdot C^0_{-m+2\delta}(\mathcal{F}_\infty, \pi^* (E \boxtimes E)) \subset L^1(\mathcal{F}_\infty, E).
$$
\end{proof}

We can now establish an index theorem for $\widetilde{D}_{\varepsilon,u}$ by an adaptation of the
argument of \cite[Proposition 6.13]{Va} to our singular setting. We emphasize that as before, the proof
of \cite{Va} has to be amended significantly due to the singularities.

\begin{thm}\label{thm6_4}
The $\Gamma$-Fredholm operator $\widetilde{D}_{\epsilon,u}$ admits an index formula as $u\to 0$
\begin{equation}\label{ind-thm1}
\ind_\Gamma \widetilde{D}_{\epsilon,u} = 
\int\limits_M L(M) + 
\sum_{\alpha \in A} \int\limits_{\ Y_\alpha} b_{\alpha}
\\ - \frac{\eta_{\Gamma}(\widetilde{B}(-1))}{2} + \underline{o}(1).
\end{equation}
\end{thm}

\begin{proof}
We adapt the argument of \cite[Proposition 6.13]{Va} to our singular setting, and detail out only those elements of the proof,
where the presence of singularities requires a change of argument.
By Theorem \ref{D-Fredholm}, $\widetilde{D}_{\varepsilon,u}$ is $\Gamma$-Fredholm and, in particular, 
the projection $P_{\ker \widetilde{D}_{\varepsilon,u}}$ onto the null space of $\widetilde{D}_{\varepsilon,u}$ is $\Gamma$-trace class. Hence
\begin{align*}
\ind_\Gamma \tilde{D}_{\epsilon,u} = \textup{s-Tr}_\Gamma \left(P_{\ker \widetilde{D}_{\varepsilon,u}}\right)
= \lim_{N\to \infty} \textup{s-Tr}_\Gamma \left( \phi_N P_{\ker \widetilde{D}_{\varepsilon,u}} \phi_N\right).
\end{align*}
By Proposition \ref{eP} we conclude
\begin{align*}
\ind_\Gamma \widetilde{D}_{\epsilon,u} 
&= \lim_{N\to \infty} \lim_{t\to \infty} \textup{s-Tr}_\Gamma \left( \phi_N \, e^{-t \, \widetilde{D}_{\varepsilon,u}^2} \phi_N \right) \\
&= \lim_{N\to \infty} \left(  \textup{s-Tr}_\Gamma \left( \phi_N \, e^{-s \, \widetilde{D}_{\varepsilon,u}^2} \phi_N \right)
 -  \int_s^\infty \textup{s-Tr}_\Gamma \left( \phi_N \, \widetilde{D}_{\varepsilon,u}^2 e^{-t \, \widetilde{D}_{\varepsilon,u}^2} \phi_N \right) dt \right) \\
&= \lim_{N\to \infty} \left(  \textup{s-Tr}_\Gamma \left( \phi_N \, e^{-s \, \widetilde{D}_{\varepsilon,u}^2} \phi_N \right)
 -  \int_s^N \textup{s-Tr}_\Gamma \left( \phi_N \, \widetilde{D}_{\varepsilon,u}^2 e^{-t \, \widetilde{D}_{\varepsilon,u}^2} \phi_N \right) dt \right.
\\ &  \left. -  \int_N^\infty \textup{s-Tr}_\Gamma \left( \phi_N \, \widetilde{D}_{\varepsilon,u}^2 e^{-t \, \widetilde{D}_{\varepsilon,u}^2} \phi_N \right) dt \right)
=: \lim_{N\to \infty} \left( I_1 + I_2 + I_3 \right).
\end{align*}	
Note that the third term $I_3$ exists individually by a similar argument as in \cite[(7.10)]{PiVe}.
By an analogous argument as in \cite[p. 37-38]{Va}, $I_3$ vanishes as $N\to \infty$. In order to estimate 
$I_2$, we compute exactly as in \cite[(6.21)]{Va}
\begin{equation*}
\begin{split}
-2 \cdot \textup{s-Tr}_\Gamma \left( \phi_N \, \widetilde{D}_{\varepsilon,u}^2 e^{-t \, \widetilde{D}_{\varepsilon,u}^2} \phi_N \right)
= & \ \textup{s-Tr}_\Gamma \left( c(d\phi_N^2)  \, \widetilde{D}_{\varepsilon,u} e^{-t \, \widetilde{D}_{\varepsilon,u}^2} \right)
 \\ = & \ \textup{s-Tr}_\Gamma \left( c(d\phi_N^2)  \left(\widetilde{D}_{\varepsilon,u} e^{-t \, \widetilde{D}_{\varepsilon,u}^2} 
 - \widetilde{S}_{\varepsilon,u} e^{-t \, \widetilde{S}_{\varepsilon,u}^2} \right) \right) 
\\ & \  \textup{s-Tr}_\Gamma \left( c(d\phi_N^2)  \, \widetilde{S}_{\varepsilon,u} e^{-t \, \widetilde{S}_{\varepsilon,u}^2} \right).
 \end{split}
 \end{equation*}
By Theorem \ref{off-diagonal} we have the following uniform estimate
\begin{align*}
\left|  c(d\phi_N^2) \left(\widetilde{D}_{\varepsilon,u} e^{-t\widetilde{D}_{\varepsilon,u}^2} - 
\widetilde{S}_{\varepsilon,u} e^{-t\widetilde{S}_{\varepsilon,u}^2} \right)(p,p)\right| 
\leq C \rho(p)^{-m +2 \delta} \exp \left(- \frac{(N-2)^2}{6t}\right).
\end{align*}
Noting that $\rho^{-m+2\delta}$ is locally integrable up to the singular strata, we conclude 
exactly as in \cite[p.39]{Va}
\begin{align*}
&\left| \, \int_s^N \textup{s-Tr}_\Gamma \left( c(d\phi_N^2)  \left(\widetilde{D}_{\varepsilon,u} e^{-t \, \widetilde{D}_{\varepsilon,u}^2} 
 - \widetilde{S}_{\varepsilon,u} e^{-t \, \widetilde{S}_{\varepsilon,u}^2} \right) \right) dt \right| 
\\ &\quad  \leq C \int_s^N \exp \left(- \frac{(N-2)^2}{6t}\right) dt \leq 
C \left( N^2 e^{-N/c_1} + e^{-c_2/s}\right),
\end{align*}
for some uniform constants $C,c_1,c_2>0$. Consequently, in the limit $s\to 0$ and $N\to \infty$, 
we can replace $\widetilde{D}_{\varepsilon,u}$ by $\widetilde{S}_{\varepsilon,u}$ in $I_2$. Now repeating
the arguments of \cite[p.39]{Va} verbatim, we arrive at the following intermediate formula
\begin{align}
\ind_\Gamma \widetilde{D}_{\epsilon,u} 
= \lim_{N\to \infty} \lim_{s \to 0} \textup{s-Tr}_\Gamma \left( \phi_N \, e^{-s \, \widetilde{D}_{\varepsilon,u}^2} \phi_N \right)  
+ \frac{\eta_\Gamma (\widetilde{B}(-1))}{2},
\end{align}
in particular, the first limit exists. From there we can now follow \cite[p. 40-41]{Va} 
without any additional changes and deduce the statement with integrals in \eqref{ind-thm1}, involving 
$b_\alpha$, arising from the local asymptotics of $\textup{s-Tr}_\Gamma \left( e^{-s \, \widetilde{D}_{\infty}^2} \restriction \widetilde{M} \right)$
as $s\to 0$. This local asymptotics coincides with the asymptotics 
\eqref{heat-kernel-asymptotics} of 
$\textup{s-Tr} \left( e^{-s \, D_{\infty}^2} \restriction M \right)$ as $s\to 0$.
\end{proof}

\begin{remark}\label{Vai-differences}
The result of Theorem \ref{ind-thm1} is an analogue of \cite[Proposition 6.13]{Va} in the stratified setting.
In contrast to \cite{Va}, heat kernel estimates for large times do not hold uniformly here. In fact, Theorems
\ref{off-diagonal} and \ref{on-diagonal} assert that the heat kernel estimates are not uniform up to the singular strata.
Fortunately, the estimates are still integrable at the strata, so that we can estimate the 
$\Gamma$-traces accordingly, and still follow the general proof outline of \cite{Va}.
\end{remark}

\subsection{The $L^2$-Gamma Index theorem on $\widetilde{M}_\infty$ and the signature formula}\label{subsection:gamma-index}
Recall that $\widetilde{D}_\infty$ is not $\Gamma$-Fredholm, since $\widetilde{B}$ is not invertible. 
Nevertheless we can define an $L^2$-Gamma index of $\widetilde{D}_\infty$ by proving that
$\ker_{(2)} \widetilde{D}_\infty$ and $\ker_{(2)} \widetilde{D}^*_\infty$ have finite $\Gamma$-dimensions,
which is the analogue of \cite[Corollary 6.6]{Va}, albeit with a different proof.

\begin{prop}\label{finite-dim}
For any $\varepsilon > 0$, 
$\ker_{(2)} \widetilde{D}_{\varepsilon,0}$ and $\ker_{(2)} \widetilde{D}^*_{\varepsilon,0}$ have finite $\Gamma$-dimensions. 
\end{prop}

\begin{proof}
We first prove the statement for $\widetilde{D}_{\varepsilon, 0}$ employing the Browder-Garding decomposition. 
By the Browder-Garding decomposition on $\partial \widetilde{M}$ as stated in Theorem \ref{Browder-Garding}, 
there exist countably many sections $e_j: \R \to \dom ( \widetilde{B})$ such that
$\widetilde{B} e_j(\lambda) = \lambda e_j(\lambda)$. We consider for any $\w \in L^2\Omega^*(\widetilde{M})$
$$
(V \w)_j(\lambda,x) := \int_{\partial \widetilde{M}} ( \w(x,p), e_j(\lambda, p) )_{\widetilde{g}_{\partial \widetilde{M}}}
\textup{dvol}_{\widetilde{g}_{\partial \widetilde{M}}}(p).
$$ 
If $\w$ is a solution of either $\widetilde{D}_{\varepsilon, 0}$ or its adjoint, 
then $(V\w)_j(\lambda,x)$ satisfies for any $j$ over the cylinder $(-\infty, 0) \times \partial \widetilde{M}$
the following equations
\begin{equation}
\begin{split}
\left( \partial_x + \lambda - \theta(x) \lambda \Pi_{\varepsilon}(\lambda) \frac{}{} \right)  (V\w)_j(\lambda, x) &= 0, \quad
\textup{if} \ \widetilde{D}_{\varepsilon, 0} \w = 0, \\
\left( -\partial_x + \lambda - \theta(x) \lambda \Pi_{\varepsilon}(\lambda) \frac{}{} \right)  (V\w)_j(\lambda, x) &= 0, \quad
\textup{if} \ \widetilde{D}^*_{\varepsilon, 0} \w = 0.
\end{split}
\end{equation}
We conclude that solutions $(V\w)_j(\lambda, x)$ are of the following form
\begin{equation}
\begin{split}
(V\w)_j(\lambda, x) = \textup{const} \cdot e^{+ \lambda (x-\vartheta(x) \Pi_{\varepsilon}(\lambda))}, \quad
\textup{if} \ \widetilde{D}_{\varepsilon, 0} \w = 0, \\
(V\w)_j(\lambda, x) = \textup{const} \cdot e^{- \lambda (x-\vartheta(x) \Pi_{\varepsilon}(\lambda))}, \quad
\textup{if} \ \widetilde{D}^*_{\varepsilon, 0} \w = 0,
\end{split}
\end{equation}
where $\vartheta(x) = x$ for $x<-1$, extended to a smooth positive function on $\widetilde{M}_\infty$, 
bounded away from zero, and being identically $1$ on $\widetilde{M}$.
Consequently, due to the $L^2$-condition, $\ker_{(2)} \widetilde{D}_{\varepsilon, 0}$
and $\ker_{(2)} \widetilde{D}^*_{\varepsilon, 0}$ lie in $e^{\varepsilon \vartheta} L^2\Omega^*(\widetilde{M})$.
We now continue with outlining the argument for $\widetilde{D}_{\varepsilon, 0}$, the adjoint treated verbatim.
For any $\w \in  \ker_{(2)} \widetilde{D}_{\varepsilon, 0}$, by the argument above 
$e^{-\varepsilon \vartheta} \w \in L^2\Omega^*(\widetilde{M})$, and we compute
$$
\left(\widetilde{D}_{\varepsilon, 0} + \widetilde{\sigma} ( \varepsilon  \, d\vartheta) \right) 
e^{-\varepsilon \vartheta} \w = e^{-\varepsilon \vartheta} \widetilde{D}_{\varepsilon, 0} \w = 0.
$$
Hence $e^{-\varepsilon \vartheta} \w \in \dom \left( \widetilde{D}_{\varepsilon, 0} + \widetilde{\sigma} 
(\varepsilon \, d\vartheta)\right) \equiv \dom \left( \widetilde{D}_{\infty} \right)$.
This proves 
\begin{align}\label{kere}
\ker_{(2)} \widetilde{D}_{\varepsilon, 0} \subset e^{\varepsilon \vartheta} \dom \left( \widetilde{D}_\infty \right).
\end{align}
In the next step we prove that the inclusion $\iota: e^{\varepsilon \vartheta} \dom \left( \widetilde{D}_\infty \right) \hookrightarrow L^2\Omega^*(\widetilde{M})$
is $\Gamma$-compact. Indeed, as shown by \eqref{cpt-incl}, $\dom_{\textup{comp}} \left( \widetilde{D}_\infty \right) \hookrightarrow L^2\Omega^*(\widetilde{M})$
is $\Gamma$-compact. In particular
$$
\iota_N: e^{\varepsilon \vartheta} \phi_N \dom \left( \widetilde{D}_\infty \right) \hookrightarrow L^2\Omega^*(\widetilde{M})
$$
is $\Gamma$-compact. Now the sequence $(\iota_N)_N$ converges to identity in the operator norm as $N \to \infty$. 
Indeed for $\w \in e^{\varepsilon \vartheta} \dom \left( \widetilde{D}_\infty \right) $
\begin{align*}
\| \w - \iota_N \w \|^2_{L^2\Omega^*(\widetilde{M})} &\leq \int_{-\infty}^N \| \w (x) \|^2_{L^2\Omega^*(\partial \widetilde{M})} dx
\leq C e^{-2\varepsilon N} \int_{-\infty}^N e^{-2\varepsilon \vartheta(x)}\| \w (x) \|^2_{L^2\Omega^*(\partial \widetilde{M})} dx
\\ & \leq C e^{-2\varepsilon N} \| \w \|^2_{e^{\varepsilon \vartheta}L^2\Omega^*(\partial \widetilde{M})} \leq 
C e^{-2\varepsilon N} \| \w \|^2_{e^{\varepsilon \vartheta} \dom \left( \widetilde{D}_\infty \right)}.
\end{align*}
Hence $\iota$ is indeed $\Gamma$-compact and thus by \eqref{kere}, 
$\ker_{(2)} \widetilde{D}_{\varepsilon, 0}$ has finite $\Gamma$-dimension. Similar argument applies to 
$\widetilde{D}^*_{\varepsilon,0}$. 
\end{proof}

For $\widetilde{D}_{\infty}$ we find by a similar argument as 
in \eqref{kere} that for any $\lambda_0 > 0$, 
$$
\ker_{(2)} \widetilde{D}_{\infty} \subset e^{-\lambda_0 \vartheta} 
\dom \left( \widetilde{D}_\infty \right) \subset e^{-\lambda_0 \vartheta} L^2\Omega^*(\widetilde{M}_\infty).
$$
Similar, to the proof of $\Gamma$-compactness for $\iota$, we find that for any $\lambda_0>0$ the inclusion 
$e^{-\lambda_0 \vartheta} \dom \left( \widetilde{D}_\infty \right) 
\hookrightarrow e^{-2\lambda_0 \vartheta}  L^2\Omega^*(\widetilde{M}_\infty)$ is $\Gamma$-compact as well. Consequently,
we find $\ker_{(2)} \widetilde{D}_{\infty} \subset e^{-2\lambda_0 \vartheta}  L^2\Omega^*(\widetilde{M}_\infty)$
is of finite $\Gamma$-dimension. We argue verbatim for the adjoint. 
By \cite[Lemma 6.5 (c)]{Va} that $\Gamma$-dimension is independent of $\lambda_0>0$
and hence we set 
\begin{equation}\label{finite-dim2}
\begin{split}
&\dim_\Gamma \ker_{(2)} \widetilde{D}_{\infty} := \dim_\Gamma 
\left( \ker_{(2)} \widetilde{D}_{\infty} \subset e^{-2\lambda_0 \vartheta}  L^2\Omega^*(\widetilde{M}_\infty) \right) < \infty, \\
&\dim_\Gamma \ker_{(2)} \widetilde{D}^*_{\infty} := \dim_\Gamma 
\left( \ker_{(2)} \widetilde{D}^*_{\infty} \subset e^{-2\lambda_0 \vartheta}  L^2\Omega^*(\widetilde{M}_\infty) \right) < \infty.
\end{split}
\end{equation}

In view of Proposition \ref{finite-dim} and of \eqref{finite-dim}, we can now define the $L^2$-Gamma indices
for $\widetilde{D}_\infty$ and $\widetilde{D}_{\varepsilon,0}$, which in contrast to 
$\widetilde{D}_{\varepsilon,u}$ with $u> 0$, are not $\Gamma$-Fredholm.

\begin{defn} For any $\varepsilon > 0$ we define
\begin{equation*}
\begin{split}
&L^2-\ind_\Gamma \widetilde{D}_\infty := \dim_\Gamma \ker_{(2)} \widetilde{D}_\infty - \dim_\Gamma \ker_{(2)} \widetilde{D}^*_\infty, \\
&L^2-\ind_\Gamma \widetilde{D}_{\varepsilon,0} := \dim_\Gamma \ker_{(2)} \widetilde{D}_{\varepsilon,0} - \dim_\Gamma \ker_{(2)} \widetilde{D}^*_{\varepsilon,0}.
\end{split}
\end{equation*}
\end{defn}

In order to relate the $L^2$-Gamma index $L^2-\ind_\Gamma \widetilde{D}_\infty$ with the $\Gamma$-index of 
$\widetilde{D}_{\varepsilon, u}$, we need to introduce the concept of extended $L^2$-solutions in the setting
of Galois coverings. We follow the approach of \cite{Va} and define for any $\varepsilon \geq 0$ 
(recall $\widetilde{D}_{\varepsilon, 0} =  \widetilde{D}_{\infty}$ for $\varepsilon = 0$)  the extended solutions by
\begin{equation}
\begin{split}
&\textup{ext-}\ker_{(2)}  \widetilde{D}_{\varepsilon, 0} := \bigcap\limits_{\lambda_0 > 0} 
\left(\ker \widetilde{D}_{\varepsilon, 0} \subset e^{-\lambda_0 \vartheta} L^2\Omega^*(\widetilde{M}_\infty)\right), \\
&\textup{ext-}\ker_{(2)}  \widetilde{D}^*_{\varepsilon, 0} := \bigcap\limits_{\lambda_0 > 0} 
\left(\ker \widetilde{D}^*_{\varepsilon, 0} \subset e^{-\lambda_0 \vartheta} L^2\Omega^*(\widetilde{M}_\infty)\right).
\end{split}
\end{equation}

\begin{prop}\label{hg} Let us write $D= D^+$ and $D^*= D^-$. Then
\begin{equation}
\begin{split}
\dim_\Gamma \ker_{(2)} \widetilde{D}^\pm_{\varepsilon,0} = 
\lim_{u\to 0} \dim_\Gamma \ker_{(2)} \widetilde{D}^\pm_{\varepsilon,\pm u}, \\
\dim_\Gamma \textup{ext-}\ker_{(2)} \widetilde{D}^\pm_{\varepsilon,0} = 
\lim_{u\to 0} \dim_\Gamma \ker_{(2)} \widetilde{D}^\pm_{\varepsilon,\mp u}.
\end{split}
\end{equation}
\end{prop}

\begin{proof}
The  Browder-Garding decomposition, as stated in Theorem \ref{Browder-Garding}, asserts
the existence of countably many sections $e_j: \R \to \dom ( \widetilde{B})$ such that
$\widetilde{B} e_j(\lambda) = \lambda e_j(\lambda)$. We consider as before for any $\w \in L^2\Omega^*(\widetilde{M})$
$$
(V \w)_j(\lambda,x) = \int_{\partial \widetilde{M}} ( \w(x,p), e_j(\lambda, p) )_{\widetilde{g}_{\partial \widetilde{M}}}
\textup{dvol}_{\widetilde{g}_{\partial \widetilde{M}}}(p).
$$ 
Over the cylinder $(-\infty, 0) \times \partial \widetilde{M}$, 
any $(V\w)_j(\lambda,x)$ solves
\begin{equation}
\begin{split}
\left( \partial_x + \lambda + \theta(x) (u-\lambda \Pi_{\varepsilon}(\lambda)) \frac{}{} \right)  (V\w)_j(\lambda, x) &= 0, \quad
\textup{if} \ \widetilde{D}_{\varepsilon, u} \w = 0, \\
\left( -\partial_x + \lambda + \theta(x) (u-\lambda \Pi_{\varepsilon}(\lambda)) \frac{}{} \right)  (V\w)_j(\lambda, x) &= 0, \quad
\textup{if} \ \widetilde{D}^*_{\varepsilon, u} \w = 0.
\end{split}
\end{equation}
We conclude that solutions $(V\w)_j(\lambda, x)$ are of the following form
\begin{equation}
\begin{split}
(V\w)_j(\lambda, x) = \textup{const} \cdot e^{- u \vartheta(x)} e^{- \lambda (x-\vartheta(x) \Pi_{\varepsilon}(\lambda))}, \quad
\textup{if} \ \widetilde{D}_{\varepsilon, u} \w = 0, \\
(V\w)_j(\lambda, x) = \textup{const} \cdot e^{+ u \vartheta(x)} e^{+ \lambda (x-\vartheta(x) \Pi_{\varepsilon}(\lambda))}, \quad
\textup{if} \ \widetilde{D}^*_{\varepsilon, u} \w = 0,
\end{split}
\end{equation}
where $\vartheta'(x)=\theta(x)$, such that $\vartheta(x)=x$ for $x \leq -1$.
From here we conclude 
\begin{equation}
\begin{split}
\left( \ker \widetilde{D}^\pm_{\varepsilon,0}\subset  e^{ - u \vartheta} L^2\Omega^*(\widetilde{M}_\infty) \right) &= 
e^{ - u \vartheta} \ker_{(2)} \widetilde{D}^\pm_{\varepsilon, \mp u}, \\
\left( \ker \widetilde{D}^\pm_{\varepsilon,0}\subset  e^{ u \vartheta} L^2\Omega^*(\widetilde{M}_\infty) \right) &= 
e^{ u \vartheta} \ker_{(2)} \widetilde{D}^\pm_{\varepsilon, \pm u}.
\end{split}
\end{equation}
This can alternatively be deduced from the following relations
\begin{equation}
\begin{split}
\widetilde{D}^\pm_{\varepsilon,0} e^{ - u \vartheta} \w &= e^{ - u \vartheta}  \widetilde{D}^\pm_{\varepsilon, \mp u} \w, \\
\widetilde{D}^\pm_{\varepsilon,0} e^{ u \vartheta} \w &= e^{ u \vartheta}  \widetilde{D}^\pm_{\varepsilon, \pm u} \w.
\end{split}
\end{equation}
Now the statement follows.
\end{proof}

\begin{cor}\label{index-relations}
We define, still writing $D= D^+$ and $D^*= D^-$
\begin{equation*}
h_{\Gamma,\varepsilon}^\pm  := \dim_\Gamma \textup{ext-}\ker_{(2)} \widetilde{D}^\pm_{\varepsilon,0} -
\dim_\Gamma \ker_{(2)} \widetilde{D}^\pm_{\varepsilon,0}.
\end{equation*}
Then the $(L^2 -)$ Gamma indices of $\widetilde{D}_{\varepsilon,0}$ and $\widetilde{D}_{\varepsilon,\pm u}$ are related by
\begin{equation*}
\begin{split}
L^2-\ind_\Gamma \widetilde{D}_{\varepsilon,0} &= \lim_{u\to 0} \ind_\Gamma \widetilde{D}_{\varepsilon,u} +  h_{\Gamma,\varepsilon}^-
\\ &= \lim_{u\to 0} \ind_\Gamma \widetilde{D}_{\varepsilon,-u} - h_{\Gamma,\varepsilon}^+.
\end{split}
\end{equation*}
\end{cor}

\begin{proof}
The statement follows, since by Proposition \ref{hg}
\begin{equation*}
\begin{split}
\dim_\Gamma \ker_{(2)} \widetilde{D}^\pm_{\varepsilon,0} = 
\lim_{u\to 0} \dim_\Gamma \ker_{(2)} \widetilde{D}^\pm_{\varepsilon,\pm u}= 
\lim_{u\to 0} \dim_\Gamma \ker_{(2)} \widetilde{D}^\pm_{\varepsilon,\mp u} - h_{\Gamma,\varepsilon}^\pm.
\end{split}
\end{equation*}
\end{proof}

We now arrive at the $L^2$-Gamma index theorem for $\widetilde{D}_{\varepsilon,0}$.

\begin{thm} Let us write $\widetilde{B}_{\varepsilon} := \widetilde{B} (1-\Pi_{\varepsilon}(\widetilde{B}))$. Then
\begin{equation*}
\begin{split}
L^2-\ind_\Gamma \widetilde{D}_{\varepsilon,0}
= \int\limits_M L(M) + 
\sum_{\alpha \in A} \int\limits_{\ Y_\alpha} b_{\alpha}
- \frac{1}{2} \left( \eta_{\Gamma}(\widetilde{B}_{\varepsilon}) + h_{\Gamma,\varepsilon}^- - h_{\Gamma,\varepsilon}^+ \right).
\end{split}
\end{equation*}
\end{thm}

\begin{proof}
Let us write $\widetilde{B}_{\varepsilon, u}(x)$ for the tangential operator of $\widetilde{D}_{\varepsilon,u}$
on the cylinder $(-\infty, 0) \times \partial \widetilde{M}$. By Theorem \ref{thm6_4} and Corollary \ref{index-relations} we conclude
\begin{equation*}
\begin{split}
L^2-\ind_\Gamma \widetilde{D}_{\varepsilon,0} &=
\frac{1}{2} \lim_{u\to 0} \left( \ind_\Gamma \widetilde{D}_{\varepsilon,u} 
- \ind_\Gamma \widetilde{D}_{\varepsilon,-u} \right) + \frac{h_{\Gamma,\varepsilon}^- - h_{\Gamma,\varepsilon}^+}{2}
\\ &= \int\limits_M L(M) + 
\sum_{\alpha \in A} \int\limits_{\ Y_\alpha} b_{\alpha}
 - \frac{1}{4} \lim_{u\to 0} \left( \eta_{\Gamma}\left(\widetilde{B}_{\varepsilon, u}(-1)\right) - 
\eta_{\Gamma}\left(\widetilde{B}_{\varepsilon, -u}(-1)\right)\right) \\ &+ \frac{h_{\Gamma,\varepsilon}^- - h_{\Gamma,\varepsilon}^+}{2},
\end{split}
\end{equation*}
Now by \cite[Lemma 4.7]{Va} we conclude with $\widetilde{B}_{\varepsilon,0}(-1) = \widetilde{B}_{\varepsilon}$
$$
\lim_{u\to 0} \left( \eta_{\Gamma}\left(\widetilde{B}_{\varepsilon, u}(-1)\right) - 
\eta_{\Gamma}\left(\widetilde{B}_{\varepsilon, -u}(-1)\right)\right) =  2 \eta_{\Gamma}(\widetilde{B}_{\varepsilon}).
$$
This proves the statement.
\end{proof}

\begin{cor}\label{l2-index}
\begin{equation*}
\begin{split}
L^2-\ind_\Gamma \widetilde{D}_{\infty}
= \int\limits_M L(M) + 
\sum_{\alpha \in A} \int\limits_{\ Y_\alpha} b_{\alpha}
- \frac{\eta_{\Gamma}(\widetilde{B})}{2}.
\end{split}
\end{equation*}
\end{cor}

\begin{proof} The proof is based on \cite[Lemma 6.10]{Va}, \cite[Lemma 4.7 (b)]{Va}
and \cite[Proposition 6.15]{Va}, which are based on purely functional analytic arguments and hence 
all carry over verbatim to our setting. First, \cite[Lemma 6.10]{Va} asserts
$$
\lim_{\varepsilon \to 0} \left( L^2-\ind_\Gamma \widetilde{D}_{\varepsilon,0}\right) = 
L^2-\ind_\Gamma \widetilde{D}_{\infty}, \quad 
\lim_{\varepsilon \to 0} h_{\Gamma,\varepsilon}^\pm = h_{\Gamma}^\pm.
$$
Further, \cite[Lemma 4.7 (b)]{Va} asserts that as $\varepsilon \to 0+$
$$
\left|  \eta_{\Gamma}(\widetilde{B}) - \eta_{\Gamma}(\widetilde{B}_{\varepsilon}) \right|
\leq \textup{Tr}_\Gamma \, \Pi'_{\varepsilon}(\widetilde{B}) \to 0,
$$
where $\Pi'_{\varepsilon}$ is the characteristic function of $(-\varepsilon, \varepsilon) \backslash \{0\}$.
The statement now follows by \cite[Proposition 6.15]{Va}, which implies that $h_{\Gamma}^+ = h_{\Gamma}^-$.
\end{proof}

Our main result, the signature theorem, as stated in Theorem \ref{main-2}, 
is obtained as follows. First note that the $L^2$ Gamma index of $\widetilde{D}_\infty$ is simply the
Hodge $\Gamma$-signature $\textup{sign}_\Gamma (s_\infty)$ of Definition \ref{gamma-sign-def}. 
Comparing now \eqref{signature-theorem-intermediate} with Corollary \ref{l2-index}, we find, noting that 
$\eta_{\Gamma}(\widetilde{B}) = 2\eta_{\Gamma}(\widetilde{B}_{\textup{even}})$ and $\eta(B) = 2\eta(B_{\textup{even}})$
\begin{equation}
\begin{split}
\textup{sign}_\Gamma (s_\infty)
= \int\limits_M L(M) + 
\sum_{\alpha \in A} \int\limits_{\ Y_\alpha} b_{\alpha}
- \frac{\eta_{\Gamma}(\widetilde{B})}{2}.
\end{split}
\end{equation}
This proves the signature theorem in Theorem \ref{main-2}. Summarizing
\begin{align*}
 \textup{sign}^{\Gamma}_{{\rm dR}} (\overline{M}_\Gamma, \partial \overline{M}_\Gamma)=\textup{sign}^{\Gamma}_{{\rm Ho}} (\overline{M}_{\Gamma,\infty})
= \int_{M} L(M) + \sum_{\alpha \in A} \int\limits_{\ Y_\alpha} b_{\alpha} - \eta_\Gamma (\widetilde{B}_{\textup{even}}).
\end{align*}

\section{Geometric applications}\label{sect:applications}

\subsection{Fundamental groups with torsion.} 
In this subsection we elaborate on a 
result of \cite{AP}, in turn directly inspired by a result of Chang-Weinberger \cite{ChWe}. 

\begin{thm}\label{CW-AP}
Let $\overline{X}$ be a Witt space of dimension $4\ell-1,$ $\ell>1,$ with regular part $X$. 
We assume  that $\pi_1 (X)$ has an element of finite order and that
$i_*: \pi_1(X)\longrightarrow \pi_1(\overline{X})$ is injective.
Then, there is an infinite number of Witt spaces $\{\overline{N}_j\}_{j\in \mathbb{N}}$ that are stratified-homotopy equivalent 
to $\overline{X}$, where $\overline{N} _i$ is not stratified diffeomorphic to $\overline{N}_k$ for $i\not= k$.
\end{thm}
The argument in  \cite[Corollary 7.6 and Remark 12]{AP} was given for Witt spaces of depth 1; however, it extends 
easily to the case of arbitrary depth, once the definition of the Cheeger-Gromov rho-invariant for the signature 
operator and its stratified diffeomorphism invariance are extended to arbitrary depth. 
Since, as we have remarked in \S \ref{G-eta-subsection}, the existence of the rho invariant 
follows from the work of Albin and Gell-Redman \cite[\S 6]{AGR17}, and the stratified diffeomorphism invariance
follows from the computation by Cheeger and Gromov in \cite{ChGr}, cf. \cite[Theorem 1.6]{PiVe} 
(that computation extends verbatim from depth 1 to depth $> 1$), we see that the proof
in  \cite[Corollary 7.6 and Remark 12]{AP} holds in fact for Witt spaces of arbitrary depth.

\subsection{More on fundamental groups with torsion}

The proof of theorem \ref{CW-AP} rests on a delicate limit argument due to Piazza and Schick \cite{Piazza-Schick}.
There is an alternative argument to the above result, more in line with the present paper.
Let us discuss this argument in the smooth case, so now $\overline{X}=X$ is a smooth manifold
of dimension $4\ell -1$. This means that we are simply reviewing carefully the proof of Chang-Weinberger \cite{ChWe}, with an eye
toward Witt spaces. In the smooth case there is a homomorphism from the $L$-group $L_{4\ell}(\mathbb{Z}\Gamma)$ to $\R$
$$
\alpha: L_{4\ell}(\mathbb{Z}\Gamma) \to \mathbb{R}.
$$ 
Its definition is explained in the work of Higson-Roe \cite{HiRo} and we review it briefly here.
Using the geometric characterization of the L-groups we consider $x \in  L_{4\ell}(\mathbb{Z}\Gamma)$
given by 
$$x  = \left[ (M,\partial M)\xrightarrow{F} (X\times [0,1],\partial( X\times [0,1])), u:X\to B\Gamma \right],
$$
with $F|_{\partial}:= f$ a homotopy equivalence and $u$ a classifying map
$$u^* E\Gamma= \widetilde{X}=\text{universal cover of X}.$$ 
The Higson-Roe map
is defined in the following way: glue $M$ and $X\times [0,1]$ through $f$ and obtain a space $Z_f$. This is not a manifold
but is a Poincar\'e space, so that its {\bf topological} signature is still well defined. A similar construction can be done
with the $\Gamma$-coverings, obtaining $\widetilde{Z}_f$. So we construct
$$
\widetilde{Z}_f= F^* (\widetilde{X}\times [0,1])\cup_{\widetilde{f}} \widetilde{X}\times [0,1]),
$$
with $\widetilde{f}$ the $\Gamma$-equivariant lift of $f$.
After these preliminaries we can define $\alpha$ by
$$
\alpha  (x)  :=
{\rm sign}^\Gamma (\widetilde{Z}_f)- {\rm sign}(Z_f).
$$
Higson-Roe \cite{HiRo} show that this is well defined.
Using properties of the mapping cylinder of $f$ and of $\widetilde{f}$, see 
\cite[Lemma 2]{Chang}, one can show that
%
$$ \alpha  (x) =
{\rm sign}^\Gamma (\widetilde{Z},\partial \widetilde{Z})-  {\rm sign}(Z,\partial Z).$$
 At this point we can use the equivalence between the 
topological signature and the de Rham and Hodge signatures and express this difference as a rho invariant,
thanks to the 
APS and Vaillant signature formulae. Now, because of the presence of an element of finite order
in $\pi_1 (X)$, we know from Chang-Weinberger that there exists an infinite number of classes $x_j\in 
L_{4\ell}(\mathbb{Z}\Gamma)$ such that $\alpha(x_i)\not= \alpha (x_k)$ for $i\not= k$. These classes $x_j$
are given by 
$$
x_j= \left[ (M_j,\partial M_j)\xrightarrow{F_j} (X\times [0,1],\partial( X\times [0,1])), u:X
\to B\Gamma \right],
$$
with $\partial M_j=N_j\sqcup X$ and $F_j |_{\partial M_j}= f_j\sqcup {\rm Id}_X$ with $f_j$ a homotopy equivalence.
Then
$$
0\not= \alpha (x_i)- \alpha (x_k)= \rho^\Gamma (\widetilde{N}_i)- 
\rho^\Gamma (\widetilde{N}_k),
$$
so that $N_i$ is not diffeomorphic to $N_k$. 
 
\medskip
\noindent
This is the path we could also take  in the Witt case.
Indeed, using \cite{Friedman-McLure}
we can extend the $\alpha$ homomorphism of Higson-Roe to the Browder-Quinn L-group
of a Witt space $\overline{X}$ and then, proceeding as above, use the two signature formulae proved in this paper for Witt spaces with boundary, in order to express the image through $\alpha$ of a class $x\in L_{{\rm BQ}} (\overline{X}\times [0,1])$ in  the Browder-Quinn L-group of
$\overline{X}$, in terms of a Cheeger-Gromov
 rho invariant of the type considered in this article. 
Thus, on the one hand using the hypothesis in Theorem \ref{CW-AP}
we can follow the proof in \cite{AP} and find an infinite number of elements $x_j\in L_{{\rm BQ}} (\overline{X}\times [0,1])$
such that $\alpha(x_i)\not= \alpha (x_k)$ for $i\not= k$. On the other hand, 
if $$x_j= (\overline{M}_j,\overline{N}_j \sqcup \overline{X})\xrightarrow{(F;f \sqcup {\rm Id})} (\overline{X}\times [0,1],
\overline{X}\times \{0\},\overline{X}\times \{1\})\xrightarrow{{\rm Id}} \overline{X}\times [0,1]\,,$$
then we find, {\it using the two signatures formulae proved in this paper}, that 
$$0\not= \alpha(x_i)- \alpha (x_k)= \rho^\Gamma (\widetilde{N}_i) - \rho^\Gamma (\widetilde{N}_k),$$
provided there is compatibility of topological
and analytical signatures on Witt manifolds with boundary.  
Granted this latter result we have thus constructed  an infinite number of Witt spaces that
are stratified homotopy equivalent to $\overline{X}$ but pairwise not stratified diffeomorphic; indeed, their Cheeger-Gromov
rho-invariants are pairwise distinct.


\subsection{Torsion free fundamental groups and stratified homotopy invariance of the Cheeger-Gromov 
rho-invariant on Witt spaces} \ \medskip

\noindent Let $\overline{N}$ and $\overline{N}'$ be  two smoothly stratified Witt spaces {\it without boundary}
 of dimension $k$,
 with $k$ odd, $k=2n-1$. We assume that $\overline{N}$ and $\overline{N}'$ are stratified-homotopy-equivalent.
We assume that the fundamental group $\Gamma:=\pi_1(\overline{N})=\pi(\overline{N}')$ of our Witt  
spaces is {\it torsion free} and satisfies the Baum-Connes conjecture for the {\it maximal} $C^*$-algebra, denoted here
simply as $C^*\Gamma$;
this means, by definition, that the assembly map
$$\mu:K_* (B\Gamma) \to K_* (C^* \Gamma)$$
is bijective.
Examples of discrete groups satisfying these two properties are given by torsion free amenable groups or by torsion free
discrete subgroups of $SO(n,1)$ and $SU(n,1)$. \medskip

Let $f: \overline{N}\to B\Gamma$ be the classifying map for the universal cover of $\overline{N}$, denoted 
$ \overline{N}_\Gamma$.
Since $\mu:K_* (B\Gamma) \to K_* (C^* \Gamma)$ is injective we can use 
 Theorem 6.8 (see in particular (6.1)) in \cite{ALMP3} and conclude that  
$$
f_* L^{{\rm GM}}_* (\overline{N})\in H_* (B\Gamma,\mathbb{Q}),
$$ is a stratified homotopy invariant. Here
 $L^{{\rm GM}}_* (\overline{N})\in H_* (\overline{N},\mathbb{Q})$ denotes the 
Goresky-MacPherson homology L-class. 
We now use  \cite[Lemma 1]{Chang} and in particular the isomorphism 
\begin{align*}
\iota: \Omega^{{\rm Witt}}_n (B\Gamma)\otimes\mathbb{Q} \rightarrow \bigoplus_{k\geq 0} H_{n-4k} (B\Gamma,\mathbb{Q}), \quad
[\overline{X},g:\overline{X}\to B\Gamma] \mapsto g_* L^{{\rm GM}}_* (\overline{X}), 
\end{align*}
in order to conclude that 
$$[\overline{N},f]=[\overline{N}',f']\;\;\text{in}\;\; \Omega^{{\rm Witt}}_n (B\Gamma)\otimes\mathbb{Q}.$$
Thus, up to taking a suitable number of copies of $\overline{N}$ and $ \overline{N}'$, there is a Witt bordism 
$(\overline{W},F)$ between $(\overline{N},f)$ and $(\overline{N}',f')$, with $F:\overline{W}\to B\Gamma$
restricting to $f$ and $f'$ on the boundary $\partial \overline{W}= \overline{N}\sqcup \overline{N}'$.
We thus get a Witt Galois covering $\overline{W}_\Gamma:= F^* E\Gamma$ with boundary
equal to the disjoint union of the universal coverings $\overline{N}_\Gamma$ $\overline{N}'_\Gamma.$
We now use the {\it surjectivity} of  $\mu:K_* (B\Gamma) \to K_* (C^* \Gamma)$ and Atiyah's
theorem on Galois coverings in the context of Witt spaces, proved in \cite[Theorem 7.1]{AP}, in order to conclude that 
$$
{\rm sign}^\Gamma (\overline{W}_\Gamma,\partial \overline{W}_\Gamma) - {\rm sign} (\overline{W},\partial
\overline{W})=0.
$$
We now crucially apply our signature formulae and get
$$\eta^{\Gamma} (\widetilde{N}) - \eta^{\Gamma} (\widetilde{N}') - 
\left( \eta ( N) - \eta (N') \right)=0,$$
with $\widetilde{N}$ equal to the regular part of $\overline{N}_\Gamma$.
This means that 
$$\rho^{\Gamma} (\widetilde{N})= \rho^{\Gamma} (\widetilde{N}').$$
Summarizing, we have proved the following result.

\begin{thm}
Let $\overline{N}$ be a compact smoothly stratified Witt space of dimension $2n-1$ without boundary and let 
$\overline{N}_\Gamma$ be
its universal cover, $\Gamma:=\pi_1 (\overline{N})$. Denote by $\widetilde{N}$ the regular part of $\overline{N}_\Gamma$.
Assume that $\pi_1 (\overline{N})$ is torsion-free and satisfies the maximal Baum-Connes conjecture. Then 
 $\rho^{\Gamma} (\widetilde{N})$ is a stratified homotopy invariant.
 \end{thm}

\end{document}